\documentclass{mta082}
\usepackage{graphicx,mathrsfs}

\usepackage{bm, tikz, pgfplots}
\usepackage{booktabs, upgreek}   
\usepackage{colortbl, cite}

\newtheorem{inftheorem}{Informal Theorem}

\def\vu{{\bm{u}}}
\def\vv{{\bm{v}}}
\def\vw{{\bm{w}}}
\def\vx{{\bm{x}}}
\def\vy{{\bm{y}}}
\def\vz{{\bm{z}}}
\def\vomega{{\bm{\omega}}}
\def\dvx{{\dot{\vx}}}
\def\dvy{{\dot{\vy}}}
\def\dvom{{\dot{\vomega}}}
\def\mA{{\bm{A}}}
\def\mB{{\bm{B}}}
\def\mD{{\bm{D}}}
\def\mI{{\bm{I}}}
\def\CC{{\boldsymbol C}}
\def\mC{{\bm{C}}}
\def\dvz{{\dot{\vz}}}
\def\LL{{\mathcal L}}
\def\eps{{\epsilon}}

\newcommand{\norm}[1]{\left\| #1 \right\|}
\newcommand{\E}{\mathbb{E}}

\usepackage{xcolor}         
\definecolor{DarkGreen}{rgb}{0.1,0.5,0.1}
\definecolor{DarkRed}{rgb}{0.5,0.1,0.1}
\definecolor{DarkBlue}{rgb}{0.1,0.1,0.5}
\usepackage{hyperref}       
\hypersetup{
	unicode=false,          
	pdftoolbar=true,        
	pdfmenubar=true,        
	pdffitwindow=false,     
	pdfnewwindow=true,      
	colorlinks=true,       
	linkcolor=DarkBlue,          
	citecolor=DarkGreen,        
	filecolor=DarkRed,    
	urlcolor=DarkBlue,          
	pdftitle={},
	pdfauthor={},
}

\begin{document}

\setcounter{page}{1}


\title{Last-Iterate Convergence\\ of Saddle-Point Optimizers via\\ High-Resolution Differential Equations%
\footnotetext{All authors contributed equally.}}
\author{
\ftauthor{Tatjana Chavdarova, Michael I. Jordan, Manolis Zampetakis}
\ftaddress{EECS, University of California, Berkeley, U.S.A.\\[-2pt]
\{tatjana.chavdarova, michael\_jordan, mzampet\}@\,berkeley.edu}}
\def\runtit{T.\,Chavdarova, M.\,I.\,Jordan, M.\,Zampetakis / Last-Iterate Convergence ...}


\maketitle


\ReceivedAccepted{July 5, 2022}{May 3, 2023}


\abstract{Several widely-used first-order saddle-point optimization methods yield an identical continuous-time 
ordinary differential equation (ODE) that is identical to that of the Gradient Descent Ascent (GDA) method 
when derived naively. However, the convergence properties of these methods are qualitatively different, even 
on simple bilinear games. Thus the ODE perspective, which has proved powerful in analyzing single-objective 
optimization methods, has not played a similar role in saddle-point optimization.\\
We adopt a framework studied in fluid dynamics -- known as High-Resolution Differential Equations (HRDEs) -- to 
design differential equation models for several saddle-point optimization methods. Critically, these HRDEs are 
distinct for various saddle-point optimization methods. Moreover, in bilinear games, the convergence properties 
of the HRDEs match the qualitative features of the corresponding discrete methods. Additionally, we show that 
the HRDE of Optimistic Gradient Descent Ascent (OGDA) exhibits \emph{last-iterate convergence} for general 
monotone variational inequalities. Finally, we provide rates of convergence for the \emph{best-iterate 
convergence} of the OGDA method, relying solely on the first-order smoothness of the monotone operator.}


\vm1
\keyw{Variational inequality, convergence, high resolution differential equations, saddle-point optimizers, 
continuous time methods.}


\vm2
\AMSsc{2020}{34H05, 49M15, 90C30.}


\section{Introduction}\label{sec:intro}

We study the convergence of min-max optimization methods by exploiting the interplay between analyses in 
continuous time and discrete time.  
Our basic setting is that of a zero-sum game in which two agents choose actions $\vx \in \mathcal{X}$ and 
$\vy \in \mathcal{Y}$, respectively, and share a loss/utility function, 
$f \colon \mathcal{X} \times \mathcal{Y} \to \R$, such that the first agent aims to minimize $f$ and the 
second agent aims to maximize $f$. 
Formally, we study the following unconstrained zero-sum game:
\vm1
\begin{equation} \label{eq:zs-g}
    \tag{ZS-G}
    \min_{\vx \in \R^{d_1}} \max_{\vy \in \R^{d_2}} f(\vx, \vy) \,,
\end{equation}

\vm3
where $f : \R^{d_1} \times \R^{d_2} \to \R$ is smooth and convex in $\vx$ and concave in $\vy$. 
A solution to this problem can be expressed as a \emph{saddle point} of $f$; 
i.e, a point $(\vx^\star,\vy^\star)$ such 
\begin{equation}\tag{SP}\label{eq:saddle_point}
\links{that}{45.5}
   f(\vx^\star,\vy)\leq f(\vx^\star,\vy^\star) \leq f(\vx,\vy^\star).
\end{equation}

\good

We can also express the min-max optimization problem as a \emph{variational inequality} (VI).  
Denote $\vz \triangleq (\vx, \vy) \in \R^d$ where $d = d_1 + d_2$,
and define the vector field $V : \R^d \to \R^d$  
and its Jacobian $J$ as follows:
\begin{align} \notag 
  V(\vz) \!=\! \begin{bmatrix}
              \nabla_{x} f(\vz) \\[1mm]
              - \nabla_{y} f(\vz) 
             \end{bmatrix} \,,  \qquad
  J(\vz) \!=\! \begin{bmatrix}
              \nabla_{x}^2 f(\vz)           &  \nabla_{y} \nabla_{x} f(\vz) \\[1mm]
             -\nabla_{x} \nabla_{y} f(\vz)  & -\nabla_{y}^2 f(\vz)
             \end{bmatrix}. 
\end{align}
We can now rewrite the convex-concave zero-sum game problem using the operator $V$ associated with $f$ as the 
following VI problem:
\begin{equation} \label{eq:vi} \tag{VI}
\emph{find }\ \ \vz^\star \quad\text{s.t.}\quad \langle \vz-\vz^\star, V(\vz)\rangle \geq 0, 
   \quad \forall \vz \in \R^d \,.
\end{equation}
When $V(\cdot)$ is monotone -- see Definition \ref{def:monotone} -- the problem is referred as a \emph{monotone 
variational inequality} (MVI) problem. 
This problem is fundamental in many fields -- such as optimization, economics, and multi-agent reinforcement 
learning \cite{Omidshafiei2017DeepDM}.  
It has recently seen renewed interest in the context of training Generative Adversarial Networks in machine 
learning \cite{goodfellow2014generative}.

A key difference between two-objective min-max training and single-objective minimization is that the 
Jacobian matrix of the gradient map in the latter case (defined in \S\ref{sec:methods}) is asymmetric, 
resulting in dynamics that can rotate around a fixed point [see Berard et al. \cite{Berard2020A} for the definition 
of the rotational component of $V$]. 
For example, in the simple bilinear game -- which is convex in $\vx$ and concave in $\vy$ -- the last iterate 
of the gradient descent ascent (GDA) method oscillates around the solution for an infinitesimal step size, 
and diverges otherwise \cite{daskalakis2018training} (see also \S\ref{sec:hrde_bilinear}). 
This behavior is undesirable for many applications of min-max optimization and consequently, the problem 
of designing algorithms with \emph{last-iterate} convergence has attracted significant attention 
\cite{goodfellow2016nips,daskalakis2018training,mescheder2018training,daskalakis2018limit,mazumdar2018convergence,
MertikopoulosPP18,adolphs2018local,chavdarova2019,golowich20}. 

Various methods have been proposed to resolve this problem, including the extragradient method [EG, Korpelevich  \cite{korpelevich1976extragradient}], optimistic gradient descent ascent [OGDA, Popov \cite{popov1980}], 
and the lookahead-minmax algorithm [LA, Chavdarova et al. \cite{chavdarova2021lamm}]. 
Many of the existing convergence proofs rely on relatively restrictive setups [see, e.g., Daskalakis et al. 
\cite{daskalakis2018training}], with the exception of Golowich et al. \cite{golowich20}, who show last-iterate 
convergence of EG for MVI problems.  
They do, however, require second-order smoothness assumptions.
Finally, building on Hsieh et al. \cite{hsieh2019}, Golowich et al. \cite{golowich2020noregret} obtain a \emph{best-iterate} rate for OGDA that depends on the initial distance to the solution.

Despite this progress for particular choices of $f$, the general understanding of min-max optimization 
problems remains very limited relative to minimization. 
To close this gap, we note that significant progress has been made in  single-objective optimization 
taking a continuous-time perspective; converting the original discrete-time algorithm to a continuous-time 
dynamical system, expressed as a system of ordinary differential equations (ODEs) 
\cite{polyak1964some,arrow1957,Helmke96optimizationand,schropp2000,SuBoydCandes2016,WibisonoWilsonJordan}. 
ODEs provide analytic tools such as Lyapunov theory which is more cumbersome in discrete time, and have 
yielded both upper and lower bounds on convergence rates which can then be translated back into discrete time. 
In the case of saddle-point problems, however, the advantages of continuous time have yet to be realized.  
Indeed, the known first-order methods, such as GDA, EG, OGDA, and LA, lead to exactly the same ODE in the 
limit when a naive conversion is carried out by taking the step size to zero \cite{hsieh2020limits}. 
This ODE is accordingly not informative regarding qualitative differences among these methods.

Hsieh et al. \cite{hsieh2020limits} also presented a negative result that highlights why min-max optimization is notably 
more challenging than single-objective minimization. 
In particular, they showed that although some existing methods converge on specific setups such as bilinear or 
convex-concave games, there exist a large class of problems for which all the popular min-max methods almost 
surely get attracted by a spurious limit cycle that consists of points that are not a solution of the problem. 
The core of their analysis is, however, based on the standard (low-resolution) ODE.  
Again, this conclusion is limited in scope, preventing an explanation of the well-known 
behavioral differences among min-max optimizers, including toy examples presented in \cite{hsieh2020limits} 
and in real-world applications such as GANs where extragradient and lookahead minmax consistently outperform GDA 
[see, e.g., Chavdarova et al. \cite{chavdarova2021lamm}]. 
This motivates our search for an ODE that more closely captures the convergence behavior of existing min-max 
optimization methods and helps in the design of new methods that avoid undesirable convergence behavior.

In summary, we pose the following questions: 
\begin{enumerate}
\item[1.] \emph{Can we design ODEs which closely model the convergence behavior of the discrete min-max optimization methods in continuous time?}
\item[2.]\vp1
\emph{Can we analyze the convergence of methods other than EG for general MVI problems?}
\emph{Moreover, can this be shown without parameter averaging and using only first-order smoothness assumptions?}
\end{enumerate}

We answer the first question affirmatively in this paper, showing how to design differential equations 
that converge to saddle points. 
We make use of a methodology from fluid dynamics \cite{fluidmechanicsbook} known as \emph{High-Resolution 
Differential Equations} (HRDEs).  
This methodology has recently been employed in the analysis of single-objective optimization problems by Shi et al. 
\cite{shi2018hrde}. 
Using this approach we derive HRDEs that differ among gradient descent ascent (GDA), extragradient (EG), 
optimistic gradient descent ascent (OGDA) and lookahead-minmax (LA). 
In accordance with the convergence properties of the discrete methods, we show that the HRDEs of EG, OGDA, 
and LA  all enjoy geometric convergence when applied to bilinear games  (see \S \ref{sec:hrde_bilinear}), 
while GDA diverges.

For the second question, we focus on the popular OGDA method which has the advantage relative to EG of 
requiring only one gradient query per parameter update. 
We show that the HRDE of OGDA converges even in the general setting of monotone variational inequalities 
(MVIs, defined in \S\ref{sec:preliminaries}).  
Note that this setting includes convex-concave problems (see \S \ref{sec:hr_ogda_monotone_vi}). 
To the best of our knowledge, this is the first presentation of a system of differential equations that 
provably converges for all monotone variational inequality problems. 

As we will see, the continuous-time analysis is particularly helpful in this setting when coupled with a 
discretization method.  
The analysis framework permits us to show that: 
(i) the best iterate of the discrete OGDA method converges with rate $O(1/\sqrt{t})$ -- to the best of 
our knowledge, this is the first convergence result for MVIs that uses only first-order smoothness and 
does not require averaging; and (ii) an implicit discretization of the HRDE of OGDA also has a convergence 
rate of order $\mathcal{O}(1/\sqrt{t})$; again we only use first-order smoothness in deriving this result.

We clarify that our positive results in the setting of OGDA are not in contradiction with the results of 
Hsieh et al. \cite{hsieh2020limits}, which refer to general non-convex problems. 
Indeed, our results take advantage of a particular structure of OGDA.  
In general, to understand which properties of a min-max optimization method allow for avoiding spurious limit cycles, we argue that it is useful to design HRDEs that closely model the discrete optimization methods.
\S\ref{sec:related_works} and~\ref{sec:contributions} provide additional context in which to place this argument.

\vm3
\subsection{Related work}\label{sec:related_works}

\vm1
Research on the saddle-point optimization problem for a convex-concave function $f$ dates to the
1960s \cite{lions-stampacchia1967,lewy-stampacchia1969}.
A milestone in this line of research is the fact that averaged (ergodic) iterates of both the 
\ref{eq:extragradient} and \ref{eq:ogda} methods achieve an optimal rate of $\mathcal{O}(\frac{1}{T})$ 
on  general MVI problems 
\cite{nemirovski2004prox,tsengLinearConvergenceIterative1995,tseng2008prox,hsieh2019,mokhtari2020convergence}.

For bilinear or strongly monotone games, several authors have established \emph{last-iterate} convergence 
[see, e.g., \cite{facchineiFiniteDimensionalVariationalInequalities2003, daskalakis2018training, liang2018interaction, 
gidel19-momentum,azizian20tight}; see also Appendix \ref{app:additional_related_works}]. 
Similarly, several authors have provided best-iterate convergence results for convex-concave problems 
\cite{Monteiro2010OnTC,mertikopoulos2019mirror,facchineiFiniteDimensionalVariationalInequalities2003}.
\cite{golowich20} prove the last-iterate convergence of \ref{eq:extragradient} at a rate 
$\smash{\mathcal{O}(\frac{1}{\vst{8}{0}\sqrt{T}})}$ on the more general problem of monotone \ref{eq:vi}s, under the smoothness 
assumption that the associated operator has a $\Lambda$-Lipschitz derivative [see Assumption 2 in \cite{golowich20}].

Many of these last-iterate results are established using the \emph{(stationary) canonical linear iterative} 
[CLI, \cite{arjevani16cli}] algorithmic framework, originally proposed for minimization. 
Another approach relies on the magnitude of the spectral radius of the linearization of the associated operator \cite{bertsekas1999nonlinear}, commonly used to analyze the stability of a method around fixed points \cite{Wang2020On,gidel19-momentum}.
While the former requires second-order smoothness assumptions for monotone VI analyses, the latter explicitly 
exploits the linearity of the operator.
On the other hand, the well-established Lyapunov stability theory can be used for possibly nonlinear dynamical 
systems, and if a Lyapunov function can be found, the convergence result holds \emph{globally}. 
Several works make use of Lyapunov theory in the context of games, for example: 
(i) Hemmat et al. \cite{hemmat2020LEAD} view the iterate as a particle in a dynamical system while modeling also 
its rotational force and adding a compensating force to guarantee convergence on quadratic games, resulting 
in a second-order update rule, and 
(ii) Fiez and Ratliff \cite{fiez2021local} establish local convergence guarantees of \ref{eq:gda} when using timescale separation, by combining Lyapunov stability  and guard maps \cite{Saydy1990GuardianMA}.

Zhang et al. \cite{zhang2021unified} propose a standardized way of finding the parameters of a quadratic 
Lyapunov function of a given first-order saddle-point optimizer in discrete time, using the theory of integral 
quadratic constraints. 
This approach is restricted, however, to strongly-monotone operators.

The use of HRDEs in optimization was introduced by Shi et al. \cite{shi2018hrde} in the context of single-objective 
minimization. 
The motivation for HRDEs, in this case, was that classical ODEs do not distinguish between Nesterov’s accelerated 
gradient \cite{Nesterov1983AMF,nesterov2013introductory} and Polyak’s heavy-ball method \cite{polyak1964some}. 
Inspired from analysis used in fluid mechanics where physical properties are investigated at different scales 
using various orders of perturbations \cite{fluidmechanicsbook}, Shi et al. \cite{shi2018hrde} show that HRDEs to 
distinguish between these two methods. 
Lu \cite{lu2021osrresolution} also focuses on obtaining more precise ODEs for saddle-point optimizers, by 
proposing a different derivation from that of \cite{shi2018hrde} and from ours. 
Appendix~\ref{app:additional_related_works} provides additional discussion on related work and further details 
on the differences and advantages of our derivation. 
We also note that in this paper, we focus on the deterministic case. 
In some cases, the stochastic variants of a convergent discrete methods on monotone VIs, may not converge 
\cite{chavdarova2019}, even if decreasing step sizes are used \cite{hsieh2020}. 
Convergence in this setting is beyond the scope of this paper.

\begin{table}[tb]
  \centering
\begin{tabular}{rl}
\toprule  
\textbf{Method} & \textbf{High-Resolution Differential Equation} \vst{12}{0}\\
& with $\dot{\vz}(t) \!=\! \vomega(t)$, $\beta \!=\! 2/(\text{step size})$, $\alpha$ is a hyperparameter of \ref{eq:lookahead_mm}. \\ \midrule
GDA & $\dvom(t) = - \beta \cdot \vomega(t) - \beta \cdot V(\vz(t))$ \\[3pt]
EG  & $\dvom(t) = - \beta \cdot \vomega(t) - \beta \cdot V(\vz(t)) + 2 \cdot J(\vz(t)) \cdot V(\vz(t))$ \\[3pt]
OGDA  & $\dvom(t) = - \beta \cdot \vomega(t) - \beta \cdot V(\vz(t)) - 2 \cdot J(\vz(t)) \cdot \vomega(t)$ \\[3pt]
LA2-GDA  & $\dvom(t) = - \beta \cdot \vomega(t) - 2 \alpha \beta \cdot V(\vz(t)) + 2 \alpha \cdot J(\vz(t)) \cdot V(\vz(t))$ \\[3pt]
LA3-GDA  & $\dvom(t) = - \beta \cdot \vomega(t) 
- 3 \alpha \beta \cdot V(\vz(t)) 
+ 6 \alpha \cdot J(\vz(t)) \cdot V(\vz(t)) 
$\\
\bottomrule  
\end{tabular}
\caption{List of derived HRDEs for several saddle-point optimization methods. 
LA$k$-GDA denotes \ref{eq:lookahead_mm} with $k$ steps, combined with \ref{eq:gda} as a base optimizer. 
The map $V(\cdot)$ is the monotone operator of the corresponding variational inequality; 
i.e., it is equal to minus the $\nabla f$ when the goal is to find a minimum of a convex function $f$ and 
$(-\nabla_x f, \nabla_y f)$ when we want to find a saddle point of a convex-concave function $f$, whereas $J$ 
is the Jacobian. 
Using $\dot{\vz}(t) = \vomega (t) $.}\label{tab:hrdes_summary}

\vm{4}
\end{table}

\subsection{Overview of contributions}\label{sec:contributions}

Table~\ref{tab:hrdes_summary} summarizes the HRDEs that we derive for a range of saddle-point optimizers -- \ref{eq:gda},~\ref{eq:extragradient},~\ref{eq:ogda} and \ref{eq:lookahead_mm}.  
Our main contribution, which we present in~\S\ref{sec:hrde_sp}, is to derive and analyze these HRDEs.  
We also present an exploration of the special case of bilinear games, where we use of HRDE framework to prove: 
(i) the divergence of GDA, 
(ii) the convergence of \ref{eq:extragradient} and \ref{eq:ogda}, and 
(iii) the convergence of \ref{eq:lookahead_mm} when combined with the divergent \ref{eq:gda} with two and three steps. 

Relative to ODE analyses, these HRDE-based results are more aligned with empirical performance for bilinear games.
We refer to \S\ref{sec:hrde_bilinear} for more details.

We additionally explore HRDE-based analyses of OGDA in the general setting of MVIs, where we show the following.

\vp1
\begin{inftheorem}\label{thm:ogda_hr_mvi_convergence_informal}
{\rm(see Theorem~\ref{thm:ogda_hr_mvi_convergence})}\ \ \label{thm:ogda_hr_mvi_convergence_informal}
{\it The HRDE of the OGDA method has last-iterate convergence for MVI problems.}
\end{inftheorem}

\vm1
We show that a simple discretization of the HRDE associated with OGDA yields the original OGDA method and we also show the following.

\begin{inftheorem}\label{thm:ogda_mvi_convergence_informal}
(see Theorem~\ref{thm:ogda_mvi2_convergence})\ \ 
{\it The best-iterate of the discrete-time OGDA converges with rate $O(1/\sqrt{t})$ for all MVI problems. 
Additionally, OGDA admits asymptotic last-iterate convergence for all MVI.}
\end{inftheorem}

Finally, we derive the last-iterate convergence rate of an implicit discretization of our HRDE of OGDA; 
see Theorem~\ref{thm:ogda_i_mvi_convergence}. 
Our results are particularly noteworthy in that they provide: 
(i) the first convergence proof in continuous time for general MVI problems, and 
(ii) more broadly, the only proof that does not rely on second-order smoothness of the associated operator and 
does not use averaging; 
see Table~\ref{tab:mvi_conv_results}.

\begin{table}[h]
    \centering
\begin{tabular}{@{}rll@{}} 
\begin{tabular}{@{}c@{}} Smoothness \\ assumptions   \end{tabular}  
& \begin{tabular}{@{}l@{}} Discrete \\ time   \end{tabular}   
& \begin{tabular}{@{}l@{}} Continuous \\ time \vst{0}{10}  \end{tabular} \\ \toprule
     Last Iterate &&\\
       1\textsuperscript{st}-- \& 2\textsuperscript{nd}--order
           & \begin{tabular}{@{}l@{}} \ref{eq:extragradient} \\
Golowich et al. \cite{golowich20}  \end{tabular} 
           &  $\diagup$ \\
       1\textsuperscript{st}--order 
&  \begin{tabular}{@{}l@{}} Implicit discretization \vst{12}{0}\\ of~\ref{eq:ogda_hrde}, Thm.~\ref{thm:ogda_i_mvi_convergence}  \end{tabular}
& \begin{tabular}{@{}l@{}} \ref{eq:ogda_hrde}, \\ Thm.~\ref{thm:ogda_hr_mvi_convergence}  \end{tabular} \\ \hline
    Best Iterate \vst{12}{0}&&\\
     1\textsuperscript{st}--order 
& \begin{tabular}{@{}l@{}}~\ref{eq:ogda},\\ Thm.~\ref{thm:ogda_mvi2_convergence}  \end{tabular} 
     &  $\diagup$     
    \end{tabular}

\caption{List of last and best-iterate convergence (as opposed to average-iterate convergence) results for general monotone variational inequality problems.}
\label{tab:mvi_conv_results}
\end{table}

\vm1
{\bf Outline.}\ \ In \S\ref{sec:preliminaries} we formally define the min-max optimization setting and describe 
several saddle-point optimization methods.
\S\ref{sec:hrde} describes explicitly the procedure to derive a $\mathcal{O}(\gamma)$-HRDE given a discrete optimization method.
In \S\ref{sec:hrde_sp} we derive the HRDEs of saddle-point optimization methods, and subsequently, in \S\ref{sec:hrde_bilinear} we use these HRDEs to analyze convergence on the bilinear game using stability tools from dynamical systems.
\S\ref{sec:hr_ogda_monotone_vi} presents our main result that proves last-iterate convergence of the OGDA method on monotone variational inequalities, using HRDEs and Lyapunov stability theory.

\section{Preliminaries}\label{sec:preliminaries}

\noindent\textbf{Notation.}
Vectors are denoted with bold lowercase letters, e.g., $\vv$, whereas bold uppercase letters denote matrices.  
We write
$\mA \succeq 0$ to denote that $\mA$ is a positive semidefinite matrix. 
We use indices to refer to a discrete-time sequence, e.g., $\vz_n$, and $\vz(t)$ for a continuous-time 
vector-valued function. 
For a complex number $c \in \CC$ we write $c = \mathfrak{R}(c) + i \mathfrak{I} (c)$ where $\mathfrak{R}(c)$ 
is the real and $\mathfrak{I}(c)$ is the imaginary part of $c$. 
The Euclidean norm of vector $\vv$ is denoted by $\norm{\vv}_2$, and the inner product in Euclidean space by 
$\langle\cdot,\cdot\rangle$. 
We use  $\norm{\mA}_F$ to denote the Frobenius norm of the matrix $\mA$; i.e., 
$\norm{\mA}_F = \sqrt{\sum_i \sum_j a_{i j}^2}$.

\vp1
\begin{definition}\label{def:monotone}
An operator $V\!:\! \R^d \to \R^d $ is \emph{monotone} if $\langle \vz-\vz', V(\vz)-V(\vz') \rangle \geq 0$ 
$\forall \vz, \vz' \in \R^d$. 
$V$ is {\it $\mu$-strongly monotone} if: 
$\langle \vz-\vz', V(\vz)-V(\vz') \rangle \geq \mu\|\vz-\vz'\|^2$ for all $\vz, \vz' \in \R^d$.
\end{definition}

\begin{assumption}\label{asp:firstOrderSmoothness}
(First-order Smoothness)\ \ 
    The operator $V: \R^d \to \R^d$ satisfies
  \emph{$L$-first-order smoothness}, or $L$-smoothness, if $V$ is a 
  $L$-Lipschitz function.
\end{assumption}

\begin{assumption}\label{asp:secondOrderSmoothness}
(Second-order Smoothness)\ \ 
Letting $V: \R^d \to \R^d$ be an operator and letting $J: \R^d \to \R^{d \times d}$ denote its Jacobian, we 
say that $V$ satisfies \emph{$L_2$-second-order smoothness} if $J$ is a $L_2$-Lipschitz function: 
\vm1
\[
\norm{J(\vz) - J(\vz')}_F \le L_2 \cdot \norm{\vz - \vz'}_2. 
\]
\end{assumption}

\vm2
To simplify our presentation we assume that $\vz^\star = 0$, and thus the monotonicity condition becomes:
\vm1
\begin{equation} \label{eq:vi_fact1}
    \langle V(\vz), \vz \rangle \geq 0, \qquad  \forall \vz \in \R^d.
\end{equation}
This is without loss of generality given that all of our arguments are translation invariant. Finally, a well-known implication of
Definition~\ref{def:monotone} is that:
\begin{equation} \label{eq:vi_fact2}
    J(\vz) \succeq 0, \qquad \forall \vz \in \R^d \,. 
\end{equation}

\subsection{Saddle-point optimization methods}\label{sec:methods}

\textbf{Gradient Descent Ascent.}
A natural extension of gradient descent to the setting of saddle-point optimization is as follows: 
\begin{equation} \tag{GDA}\label{eq:gda}
    \vz_{n+1} = \vz_n - \gamma V(\vz_n)  \,,
\end{equation}
where $\gamma\in [0,1]$ denotes a step size. 
Recalling that $V$ has the form $(-\nabla_x f, \nabla_y f)$ in the setting of zero-sum games, we see that this algorithm involves a descent step in the $x$ direction and an ascent step in the $y$ direction.  
The algorithm is accordingly referred to as \emph{gradient descent ascent} (GDA).

\textbf{Extragradient.} 
The \emph{extragradient} (EG) method uses a ``prediction'' step to obtain an extrapolated point $\vz_{t+\frac{1}{2}}$ 
using \ref{eq:gda}: $\vz_{n+\frac{1}{2}} \!=\! \vz_n - \gamma V(\vz_n)$, and which then adds the gradient at the extrapolated point to the current iterate $\vz_n$:
\begin{equation} \tag{EG} \label{eq:extragradient}
\begin{aligned}     
\vz_{n+1}  \!=\! \vz_n - \gamma V(  \vz_n - \gamma V(\vz_n) )  \,,
\end{aligned}
\end{equation}
where $\gamma\!\in\![0,1]$ denotes a step size.  
The EG method was proposed in a seminal paper by Korpelevich \cite{korpelevich1976extragradient}.
She also established the convergence of \ref{eq:extragradient} for monotone (constrained) VIs with an L-Lipschitz 
operator, but did not provide a convergence rate. 
Facchinei and Pang \cite{facchineiFiniteDimensionalVariationalInequalities2003} subsequently extended the 
convergence analysis to the general subclass of convex-concave functions $f$, where $\vx \in \mathcal{X}$ 
and $\vy \in \mathcal{Y}$ are closed convex sets.

\textbf{Optimistic Gradient Descent Ascent.}
The \emph{optimistic gradient descent ascent} (OGDA) algorithm of Popov \cite{popov1980} makes use of the 
previous two iterates in computing the update at time $n+1$:
\vm3
\begin{equation} \label{eq:ogda} \tag{OGDA}
\hskip16mm    \vz_{n+1} = \vz_{n} - 2\gamma V(\vz_n) + \gamma V(\vz_{n-1}) \,.
\end{equation}
\textbf{Lookahead-Minmax.}  Chavdarova et al. \cite{chavdarova2021lamm} proposed the \emph{Lookahead--Minmax} (LA) 
algorithm for min-max optimization, building on work by Zhang et al. \cite{Zhang2019} in the optimization setting.  
The LA algorithm is a general wrapper of a ``base'' optimizer where, at every step $t$: 
(i) a copy of the current iterate $\tilde \vz_n$ is made:  $\tilde \vz_n \leftarrow  \vz_n$, 
(ii) $\tilde \vz_n$ is updated $k \geq 1$ times, yielding $\tilde \vomega_{n+k}$, and finally 
(iii) the update $\vz_{n+1}$ is obtained as a point that lies on a line between the current $\vz_{n}$ iterate 
and the prediction $\tilde \vz_{n+k}$: 
\vm1
\begin{equation}
\tag{LA}
\vz_{n+1} \leftarrow \vz_n + \alpha (\tilde  \vz_{n+k} -  \vz_n), \quad \alpha \in [0,1] 
\,. \label{eq:lookahead_mm}
\end{equation}
In this paper, we restrict ourselves to the use of \ref{eq:gda} as a base optimizer for LA, and denote the algorithm as \emph{LA$k$-GDA}.

\subsection{ODE representations}\label{sec:methods_odes}

Ordinary differential equations (ODEs) are continuous-time models that can provide insight into the behavior 
of iterative methods. 
An example in the setting of gradient-based optimization is provided by the seminal work of Su et al. 
\cite[\S 2]{SuBoydCandes2016}, who derived an ODE modeling Nesterov’s accelerated gradient method \cite{Nesterov1983AMF}.  
In their work and in much of the ensuing literature the relationship between an iterative discrete-time algorithm 
and a corresponding ODE model is established by taking the step size to zero.

\vm1
Unfortunately, in the general setting of saddle-point optimization, this limiting procedure is not informative \cite{hsieh2020limits}.
Indeed, it is straightforward to show that all of the GDA, EG, OGDA, LA2-GDA, and LA3-GDA methods yield the same ODE using the standard limiting procedure, and thus this ODE is silent on distinctions among these methods. For a smooth
curve $\vz(t)$ defined for $t \ge 0$, introducing the ansatz $\vz_n \approx \vz(n \cdot \delta)$, a Taylor expansion gives:
\begin{gather}
\vz_{n+1} \approx \vz((n+1) \delta ) = \vz(n\delta) + \dot{\vz} (n\delta)\delta 
     + \sfrac{1}{2} \ddot{\vz}(n \delta) \delta^2 + \dots, \nonumber\\
\links{and thus}{38.5}
\label{eq:time-taylor} 
\vz_{n+1} - \vz_{n} \!\approx\!  \dot{\vz} (n\delta)\delta 
     + \sfrac{1}{2} \ddot{\vz} (n \delta) \delta^2 + \mathcal{O}(\delta^3) \,.
\end{gather}
When we substitute this expansion into the update rules for GDA, EG, OGDA, LA2-GDA, and LA3-GDA, and let 
$\delta = \gamma \to 0$, we obtain the following ODE in all cases:
\vm1
\begin{equation}\tag{GDA-ODE} \label{eq:ode-gda}
  \dot{\vz}(t) = - V(\vz(t)).
\end{equation}
This ODE, which we refer to as GDA-ODE, is thus not able to distinguish between GDA and the various more 
sophisticated algorithms that are known to yield improved convergence in discrete time.

\vm1
This conundrum motivates us to consider alternative approaches to designing ODE representations.  
The situation parallels that in optimization, where standard low-resolution ODE models discard the 
\emph{acceleration} terms \cite{shi2018hrde}.  
The problem is even more severe for min-max optimizers, where the low-resolution ODE fully eliminates the 
distinctive ways in which first-order methods deal with the \emph{rotational} vector field. 
Thus all methods yield the same ODE as that associated with GDA, which is non-converging even on the most 
simple bilinear games.

\section{High-resolution differential equations}\label{sec:hrde}

We turn to the high-resolution perspective on deriving ODE models.  
The key difference between high-resolution differential equations (HRDEs) and ODEs is that the $\mathcal{O}(\gamma)$ 
terms are preserved when deriving the differential equation, while all the terms $\mathcal{O}(\gamma^2)$ and higher 
are assumed to go to zero \cite{shi2018hrde}. 
We begin by outlining the basic steps to derive HRDEs, given a  resolution $r$ that has been selected a priori.
\begin{enumardot}
\item 
Write the update rule of the discrete method in the following general form: 
\begin{equation} \tag{GF} \label{eq:general_form}
  \sfrac{\vz_{n+1}-\vz_n}{\gamma} = \mathcal{U}(\vz_{n+k}, \dots, \vz_0) \,,
\end{equation}
where $\mathcal{U}$ denotes the update rule of the discrete algorithm; e.g., for \ref{eq:gda}, we have $\mathcal{U}\!\triangleq \! -V(\vz_n)$.

\item\good
Introduce the ansatz $\vz_n \!\approx\! \vz(n\delta)$. 

\item\vp1
Make the step size dependence of the terms of the resulting equation ``explicit'' using Taylor 
expansion up to degree $r$ in the step size (or smaller if the term is divided by a step size), leaving 
solely $\vz(n\delta)$-terms. 
We perform Taylor expansion both over time (iterates), denoting the step size as $\delta$, and over 
parameter space, denoting the step size as $\gamma$. 

\item\vp1
Choose a ratio $\delta\!=\!f_\delta(\gamma)$, e.g., $\delta\!=\!\gamma$, $\delta\!=\!\sqrt{\gamma}$ etc., and substitute into the expansion.

\item\vp1
Denote by $r'$ the largest step size degree of the terms on the right-hand side of~\ref{eq:general_form}, and set the effective degree to be $r\!=\!\min(r, r')$.\label{itm:effective_degree}

\item\vp1
Finally, take the limit $\gamma^{r+1}\rightarrow0$ and let $n\delta\rightarrow t$.

\item\vp1
Return the corresponding $\mathcal{O}(\gamma^r)$--ODE. 
Importantly, when discretizing the resulting $\mathcal{O}(\gamma^r)$--ODE use the same ratio between the step sizes, $\delta\!=\!f_\delta(\gamma)$.
\end{enumardot}

Note that step~\ref{itm:effective_degree} implies that if an $\mathcal{O}(\gamma^{r-1})$ models the update rule of the discrete algorithm exactly, then increasing the ODE resolution to $\mathcal{O}(\gamma^{r})$ will not provide any new insight into the convergence behavior of that method.  
We shall demonstrate this below in the case of~\ref{eq:gda}. 
In the remainder of our presentation, we choose $\delta\!=\!\gamma$ for simplicity, but see Appendix~\ref{app:hrdes_derivation} for different  choices of this ratio and further discussion.

\vm2
\section{High-resolution differential equations of saddle-point optimizers}\label{sec:hrde_sp}

In this section, we analyze high-resolution differential equations at the first level, where we keep all the
$\mathcal{O}(\gamma)$ terms and we set to zero all the $\mathcal{O}(\gamma^2)$ and higher terms. 
As we shall see, such a resolution is sufficient to model the behavior of competing methods on bilinear games 
and that of OGDA for MVIs.
Nonetheless, as we argue in \S\ref{sec:mvi:eg}, this resolution is insufficient to model the extragradient 
method's behavior for general monotone variational inequality problems, and higher resolutions may be needed.

\vm1
We turn to deriving the $\mathcal{O}(\gamma)$-ODEs for
\ref{eq:gda}, \ref{eq:extragradient}, \ref{eq:ogda}, \ref{eq:lookahead_mm}2-GDA
and \ref{eq:lookahead_mm}3-GDA.

\vp1
\textbf{GDA.} Combining \eqref{eq:gda} with \eqref{eq:time-taylor} we obtain: 
$(\dot{\vz} (n \cdot \delta)\delta +  \frac{1}{2} \ddot{\vz} (n \delta) \delta^2)/\gamma = -V(\vz(n \delta))$. 
Then, by setting $\delta = \gamma$ and keeping the $\mathcal{O}(\gamma)$ terms we have:
$\dot{\vz}(t) + \frac{\gamma}{2}  \ddot{\vz}(t) = - V(\vz(t))$, yielding the following first-order system of high-resolution differential equations in the phase-space representation:
\begin{equation}\tag{GDA-HRDE} \label{eq:gda-hrde}
\begin{split}
\dot{\vz}(t) & = \vomega(t) \\[1mm]
\dvom(t) & = - \sfrac{2}{\gamma} \cdot \vomega(t) - \sfrac{2}{\gamma} \cdot V(\vz(t)) \,.
\end{split}
\end{equation}

\vm3
Note that since the standard low-resolution ODE models the~\ref{eq:gda} method exactly, no information is lost, 
and there is no need to increase the resolution of ~\ref{eq:gda}.
However, it is instructive to see that the resulting $\mathcal{O}(\gamma)$-ODE diverges as well.

\textbf{EG.} Combining \eqref{eq:extragradient} with \eqref{eq:time-taylor} we have:
\vm1
$$
(\dot{\vz} (n\delta)\delta +  \sfrac{1}{2} \ddot{\vz} (n \delta) \delta^2)/\gamma \!=\! - V(\vz(n\delta)) 
     + \gamma J(\vz(n\delta)) V(\vz(n\delta)) + \mathcal{O}(\gamma^2).
$$ 
Then, setting $\delta \!=\! \gamma$ and keeping the $\mathcal{O}(\gamma)$ terms we obtain:
\begin{align}
& \dot{\vz}(t) + \sfrac{\gamma}{2} \ddot{\vz} (t) \!=\! -V(\vz(t)) + \gamma \cdot J(\vz(t)) \cdot V(\vz(t)),
\ \ \text{yielding:} \notag\\[1mm]
\tag{EG-HRDE}\label{eq:eg_hrde}
\begin{split}
& \dot{\vz}(t)  = \vomega(t) \\
& \dvom(t)  = - \sfrac{2}{\gamma} \cdot \vomega(t) - \sfrac{2}{\gamma} \cdot V(\vz(t)) + 2 \cdot J(\vz(t)) \cdot V(\vz(t)).
\end{split}
\end{align}
\textbf{OGDA.} Combining \eqref{eq:ogda} with \eqref{eq:time-taylor} we have: 
$$
(\dot{\vz} (n\delta)\delta + \sfrac{1}{2} \ddot{\vz} (n \delta) \delta^2)/\gamma 
     = - V(\vz(n\delta)) - \delta J\big(\vz(n\delta) \big) \dot{\vz} (n\delta).
$$
Setting $\delta = \gamma$ and keeping the $\mathcal{O}(\gamma)$ terms we obtain: 
\begin{align}
& \dot{\vz}(t) + \sfrac{\gamma}{2} \ddot{\vz}(t) =  -V(\vz(t)) - \gamma J(\vz(t))\dot{\vz}(t),
   \ \ \text{yielding:} \notag\\
\tag{OGDA-HRDE}\label{eq:ogda_hrde}
\begin{split}
&\dot{\vz}(t)  = \vomega(t) \\
&\dvom(t) = - \sfrac{2}{\gamma} \cdot \vomega(t) - \sfrac{2}{\gamma} \cdot V(\vz(t)) 
   - 2 \cdot J(\vz(t)) \cdot \vomega(t).
\end{split}
\end{align}
\textbf{LA2-GDA.} Combining the equation (\ref{eq:lookahead_mm}2-GDA) with \eqref{eq:time-taylor} we obtain:
$$
(\dot{\vz}(n\delta)\delta + \sfrac{1}{2} \ddot{\vz}(n \delta) \delta^2)/\gamma = - 2 \alpha V(\vz(n\delta)) 
   - \alpha \gamma J(\vz(n\delta)) V(\vz(n\delta)).
$$
Setting $\delta = \gamma$ and keeping the $\mathcal{O}(\gamma)$ terms we have that:
\begin{align}
& \dot{\vz}(t) + \frac{\gamma}{2} \ddot{\vz}(t) =  -2\alpha V(\vz(t)) - \alpha \gamma J(\vz(t)) V(\vz(t)) ,
\ \ \text{yielding:} \notag\\
\tag{LA2-GDA-HRDE}\label{eq:la2-gda_hrde}
\begin{split}
& \dot{\vz}(t) = \vomega(t) \\
& \dvom(t)  = - \sfrac{2}{\gamma} \vomega  - 2 \alpha \sfrac{2}{\gamma} \cdot V(\vz(t)) + 2 \alpha J(\vz(t)) V(\vz(t)).
\end{split} 
\end{align}
\textbf{LA3-GDA.} Analogously -- see Appendix \ref{app:la3-gda} for details -- we have:
\begin{equation}\tag{LA3-GDA-HRDE}\label{eq:la3-gda_hrde}
\begin{split}
  \dot{\vz}(t) & = \vomega(t) \\
  \dvom(t)     & = - \sfrac{2}{\gamma} \vomega(t) 
  - \sfrac{6 \alpha}{\gamma} V(\vz(t)) 
  + 6 \alpha \cdot J(\vz(t)) \cdot V(\vz(t)). \\
\end{split} 
\end{equation}
Table \ref{tab:hrdes_summary} provides a summary of the HRDEs that we have derived in this section.  
Note that, in contradistinction to the low-resolution models, here the HRDEs differ for the various methods. 
Interestingly, the HRDE representation emphasizes the similarity between \ref{eq:lookahead_mm}-\ref{eq:gda} 
with $k=2$ and \ref{eq:extragradient}, a connection that is not apparent from a mere inspection of their 
update rules.

\vm2
\subsection{On the uniqueness of the solution of the HRDEs}\label{sec:unique}

We have the following result on the uniqueness of the solution of the HRDEs for bilinear games when $V(\cdot)$ 
satisfies first- and second-order smoothness conditions. 

\vp1
\begin{proposition}\label{prop:unique_solution}
{\rm(Unique solution of the HRDEs)}\ \ 
Each of the derived HRDEs (listed in Table \ref{tab:hrdes_summary}) has a unique solution when applied to 
bilinear games. 
Additionally, for any monotone map $V\!:\!\R^d\mapsto\R^d$ that satisfies the first and second-order smoothness 
as per Assumptions \ref{asp:firstOrderSmoothness}  {\&} \ref{asp:secondOrderSmoothness}, the HRDE of OGDA 
has a unique solution.
\end{proposition}

\good

The proof of this proposition is in Appendix~\ref{app:unique_solution}.

\subsection{Equivalent forms of OGDA-HRDE for monotone variational\\ inequalities}\label{sec:ogda_hrde_equivalent}

The system of differential equations \eqref{eq:ogda_hrde} is known in Physics for the special case when 
$V(\vz) \!=\! \nabla \phi(\vz)$, where $\phi$ is a potential function, where the solution $\vz(t)$  is 
known to describe the position of a system that is affected by non-elastic shocks described by the 
potential $\phi$, see Attouch et al. \cite{attouch2012second}. 
Using the techniques from \cite{attouch2012second}, we can show that \eqref{eq:ogda_hrde} has an equivalent 
formulation that does not involve the Jacobian, thus it is more suitable for designing discrete-time 
algorithms.

\vp1
\begin{theorem} \label{thm:equivalent_ogda_hr_mvi}
Let $V$ be a continuously differentiable map, and $\vz_0, \vomega_0, \vw_0 \in \R^d$, such that 
\vm1
\[ 
\vw_0 = -\sfrac{2}{\gamma} \vomega_0 - \sfrac{4\gamma}{2} V(\vz_0) - \vz_0,
\]
  then these statements describe the same function $\vz(t)$:
\begin{enumar}
\item 
the tuple $(\vz(t), \vomega(t))$ is a solution of the equation \eqref{eq:ogda_hrde} with initial conditions 
$\vz(0) = \vz_0$ and $\vomega(0) = \vomega_0$,

\item\vp1
the tuple $(\vz(t), \vw(t))$ is a solution of the following system of equations
    \begin{equation} \tag{OGDA-HRDE-2} \label{eq:ogda_hrde2}
        \begin{split}
        \dot{\vz}(t)     & \!=\! - \kappa \cdot \vz(t) - \kappa \cdot \vw(t) - 2 V(\vz(t)) \\[1mm]
        \dot{\vw}(t)     & \!=\! - \kappa \cdot \vz(t) - \kappa \cdot \vw(t)
        \end{split} \,
    \end{equation}
    with $\kappa = \beta/2 = 1/\gamma $ and $\vz(0) = \vz_0$ and $\vw(0) = \vw_0$.
\end{enumar}
\end{theorem}

We refer to Appendix \ref{thm:hrde_apt_ogda} for the proof of Theorem \ref{thm:equivalent_ogda_hr_mvi}.

\subsection{OGDA as an explicit discretization of the \ref{eq:ogda_hrde2}\\ dynamics}\label{sec:ogda_exp_disc}

In this section, we show that we can derive the original discrete~\ref{eq:ogda} method as an \emph{explicit} discretization of the \eqref{eq:ogda_hrde2} system of
differential equations. In particular, if we apply the Euler method to \eqref{eq:ogda_hrde2} with $\kappa = 1/2 \gamma$ and $\gamma$ to be equal to 
the step size of the Euler discretization, then we have:
\begin{align}
  \vz_{n + 1} & = \sfrac{1}{2} (\vz_n - \vw_n) - 2 \gamma V(\vz_n) \nonumber \\
  \vw_{n + 1} & = \sfrac{1}{2} (\vw_n - \vz_n) \tag{OGDA-S} \label{OGDA-S}
\end{align}
where if we eliminate $\vw$ from the above system we get: 
\begin{align}
  \vz_{n + 1} = \vz_n - 2 \gamma V(\vz_n) + \gamma V(\vz_{n - 1}). \tag{OGDA} \label{OGDA}
\end{align}
\textbf{Interpretation of $\vw$.}
The above discretization directly reveals the interpretation of $\vw$, which is related to that of $\vomega$ 
in ~\ref{eq:ogda_hrde2}, and gives rise to a novel perspective of the OGDA method.
Denoting $\hat\vz_{n+1} = \frac{1}{2}(\vz_n+\hat\vz_n)$ the point in parameter space where OGDA \emph{applies} the vector field $-2\gamma V(\vz_n)$, i.e., $\vz_{n+1}=\hat\vz_{n+1}-2\gamma V(\vz_n)$, then $\vw_{n+1} = - \hat\vz_{n+1}$. 
In words, OGDA can be seen as a ``applying the vector field $-2\gamma V(\vz_n)$ of the current iterate $\vz_n$, 
at the mid-point between the current iterate $\vz_n$ and the point where the gradient was applied at the 
\emph{previous} step, i.e. $\hat\vz_n$.''

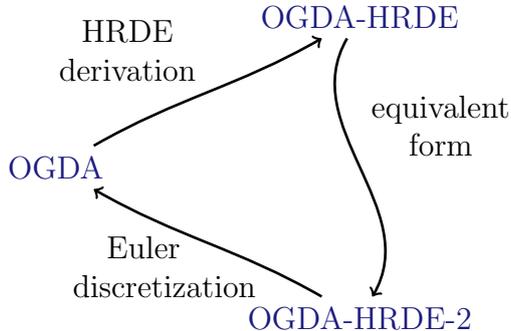
\begin{figure}[h]
\centering
\begin{tikzpicture}[node distance=3cm, auto]
    \node (eq1) at (0,0) {\ref{eq:ogda}};
    \node (eq2) at (4,2) {\ref{eq:ogda_hrde}};
    \node (eq3) at (4,-2) {\ref{eq:ogda_hrde2}};
    \draw[->, line width=1pt] (eq1) to[out=30,in=210] node[midway, above left, align=center,] {HRDE\\derivation} (eq2);
    \draw[->, line width=1pt] (eq2) to[out=240,in=60] node[midway, above right,align=center] {equivalent\\form} (eq3);
    \draw[->, line width=1pt] (eq3) to[out=150,in=-30] node[midway, below, align=center, pos=0.7] {Euler $\quad $  \\discretization} (eq1);
\end{tikzpicture}
    \caption{Illustration of the results in \S~\ref{sec:hrde_sp}: 
(i) starting from the discrete~\ref{eq:ogda} method, its continuous-time counterpart is obtained following 
the HRDE derivation (\S~\ref{sec:hrde})  yielding~\ref{eq:ogda_hrde};  
(ii) \S~\ref{sec:ogda_hrde_equivalent} presents an equivalent form ~\eqref{eq:ogda_hrde2} for monotone~\ref{eq:vi}s; 
and 
(iii) a standard explicit discretization of~\ref{eq:ogda_hrde2} (carried out in \S~\ref{sec:ogda_exp_disc}) 
reveals back the starting discrete~\ref{eq:ogda} method.}
\label{fig:ogda}
\end{figure}

\vm2
\textbf{OGDA summary.}
We observe that starting from~\ref{eq:ogda}, we derived~\ref{eq:ogda_hrde} and its equivalent form 
\ref{eq:ogda_hrde2} (under some assumptions), and that the explicit discretization of~\ref{eq:ogda_hrde2} 
reveals back the starting~\ref{eq:ogda}, see Fig.~\ref{fig:ogda}. 
This is interesting, as often there is some difference between the simple discretizations and the original 
method [e.g., as in \cite{shi2018hrde}] and attests to the validity of the HRDE approach. 
This indicates that~\ref{eq:ogda_hrde} ``closely models'' the~\ref{eq:ogda} method. 
Furthermore, in Appendix \ref{app:apt_ogda}, we show that~\ref{eq:ogda_hrde} closely models the (stochastic) 
~\ref{eq:ogda} method \emph{over time},  for strongly monotone VIs and when using decreasing step sizes. 
The result therein asserts that over time the errors between the discrete method and the continuous dynamics 
(due to discretization and stochasticity) do not accumulate in a biased way.

\section{Case study: Convergence of HRDEs on bilinear games} \label{sec:hrde_bilinear}

As we discussed in \S\ref{sec:related_works}, bilinear games have been used for comparative studies of GDA, 
EG, and OGDA. 
We augment the set of benchmark methods to include LA-GDA and use bilinear games to study the accuracy of 
HRDE models for each of these methods. 
We wish to ascertain how well the HRDEs model the corresponding discrete-time algorithms in terms of their 
asymptotic convergence. 

\vm1
We consider the following bilinear game: 
\vm1
\begin{equation} \tag{BG} \label{eq:bg}
\min_{\vx \in \R^{d_1}} \max_{\vy \in \R^{d_2}} \vx^\intercal \mA \vy \,,
\end{equation}

\vm3
where $\mA \in \R^{d_1} \times \R^{d_2}$ is a full-rank matrix.
By substituting the appropriate choices of $V(\cdot)$ and $J(\cdot)$ into the HRDEs (see Appendix~\ref{app:bilinear}), 
one can analyze the convergence of the corresponding method on the \ref{eq:bg} problem.
Moreover, as the obtained system is linear, this can be done using standard tools from dynamical systems theory.

\good

A linear dynamical system 
$
  \begin{bmatrix}
    \dvz^\intercal &
     \dvom^\intercal 
  \end{bmatrix}^\intercal = 
  \mC
  \begin{bmatrix}
    \vz^\intercal &
    \vomega^\intercal
  \end{bmatrix}^\intercal
$
is \emph{stable} if and only if the real parts of the eigenvalues of $\mC$ are all negative: $\mathfrak{R}(\lambda_i ) < 0, \enspace \forall \lambda_i \in \text{Sp}(\mC)$.
Furthermore, the \emph{Routh–Hurwitz stability criterion} provides a necessary and sufficient condition for the stability of the linear system and allows for determining if the system is stable without explicitly computing the eigenvalues of the matrix $\mC$.
Using the coefficients of its associated characteristic polynomial, the test  is  performed on the \emph{Hurwitz array},  or its generalized form when the coefficients are complex numbers \cite{xie1985}.

In the following theorem, we present a summary of Theorems~\ref{thm:bilinear_gda_convergence}--\ref{thm:bilinear_la3-gda_convergence} which we state and prove in Appendix~\ref{app:bilinear}.

\vp3
\begin{Theorem}
{\rm(Convergence on bilinear games)}\ \ 
For any $\gamma > 0$, the continuous-time dynamical system \ref{eq:gda-hrde} diverges for bilinear games. 
On the other hand, the HRDEs corresponding to \ref{eq:extragradient}, \ref{eq:ogda}, 
\ref{eq:lookahead_mm}2-\ref{eq:gda} and \ref{eq:lookahead_mm}3-\ref{eq:gda}, that is, \ref{eq:eg_hrde}, 
\ref{eq:ogda_hrde}, \ref{eq:la2-gda_hrde}, and \ref{eq:la3-gda_hrde} are all convergent, for any step size $\gamma$.
\end{Theorem}

\vp1
\begin{remark}
(\ref{eq:gda}: low resolution \ref{eq:ode-gda} vs. \ref{eq:gda-hrde} dynamics)\ \ 
The low resolution \ref{eq:ode-gda} for \ref{eq:gda}  cycles around the game's equilibrium. On the other hand, the \ref{eq:gda-hrde} dynamics diverges to infinity. 
The discrete \ref{eq:gda} method also diverges for any practical choice of the step size. Thus, the  \ref{eq:gda-hrde} more accurately models the \ref{eq:gda} method on this problem. \qed
\end{remark}

\vp1
\begin{remark}
(Sufficient conditions on $\gamma$ for convergence: on the difference between discrete and continuous time)\ \ 
The above theorem guarantees convergence of the given continuous dynamics, that is, of the \ref{eq:eg_hrde}, \ref{eq:ogda_hrde}, \ref{eq:la2-gda_hrde}, and \ref{eq:la3-gda_hrde} HRDEs for the~\ref{eq:bg} problem for 
any step size. 
However, the corresponding \ref{eq:extragradient}, \ref{eq:ogda}, \ref{eq:lookahead_mm}2-\ref{eq:gda} and 
\ref{eq:lookahead_mm}3-\ref{eq:gda} discrete methods converge on this problem only for sufficiently small 
step size. 
This is because the HRDEs accurately model the corresponding discrete methods for sufficiently small step 
size (as they are derived that way), where the conditions depend on the problem parameters (such as the 
Lipschitzness constant). 
Thus, these HRDE convergent analyses guarantee that such convergent step sizes for the discrete methods 
exist (and that do not exist for the GDA method, as \ref{eq:gda-hrde} is non convergent). 
Given a problem, one can perform a stability analysis of the finite difference schemes between the discrete 
method and the continuous dynamics in order to find the conditions on the step size or use numerical simulations 
to find convergent step size. \qed
\end{remark}


Note that it is straightforward to extend these results to the general setting of 
$$
f(\vx,\vy)\!=\!\vx^\intercal\mA\vy +  b^\intercal\vx + c^\intercal\vy,\ \ 
\text{with $b\!\in\! \R^{d_1}$, $c\!\in\! \R^{d_2}$.}
$$
In Figure~\ref{fig:bg} we present the results of a numerical experiment that confirms these theoretical results.  
We see that in contrast to the divergence of the low-resolution ODE, the discretized HRDEs corresponding to \ref{eq:extragradient}, \ref{eq:ogda}, ~\ref{eq:lookahead_mm}2-\ref{eq:gda} and 
\ref{eq:lookahead_mm}3-\ref{eq:gda} correctly model the asymptotic behavior of these algorithms on this problem.

\begin{figure}[th]
\centering
\includegraphics[width=.5\linewidth,trim={0cm .1cm .1cm .1cm},clip]{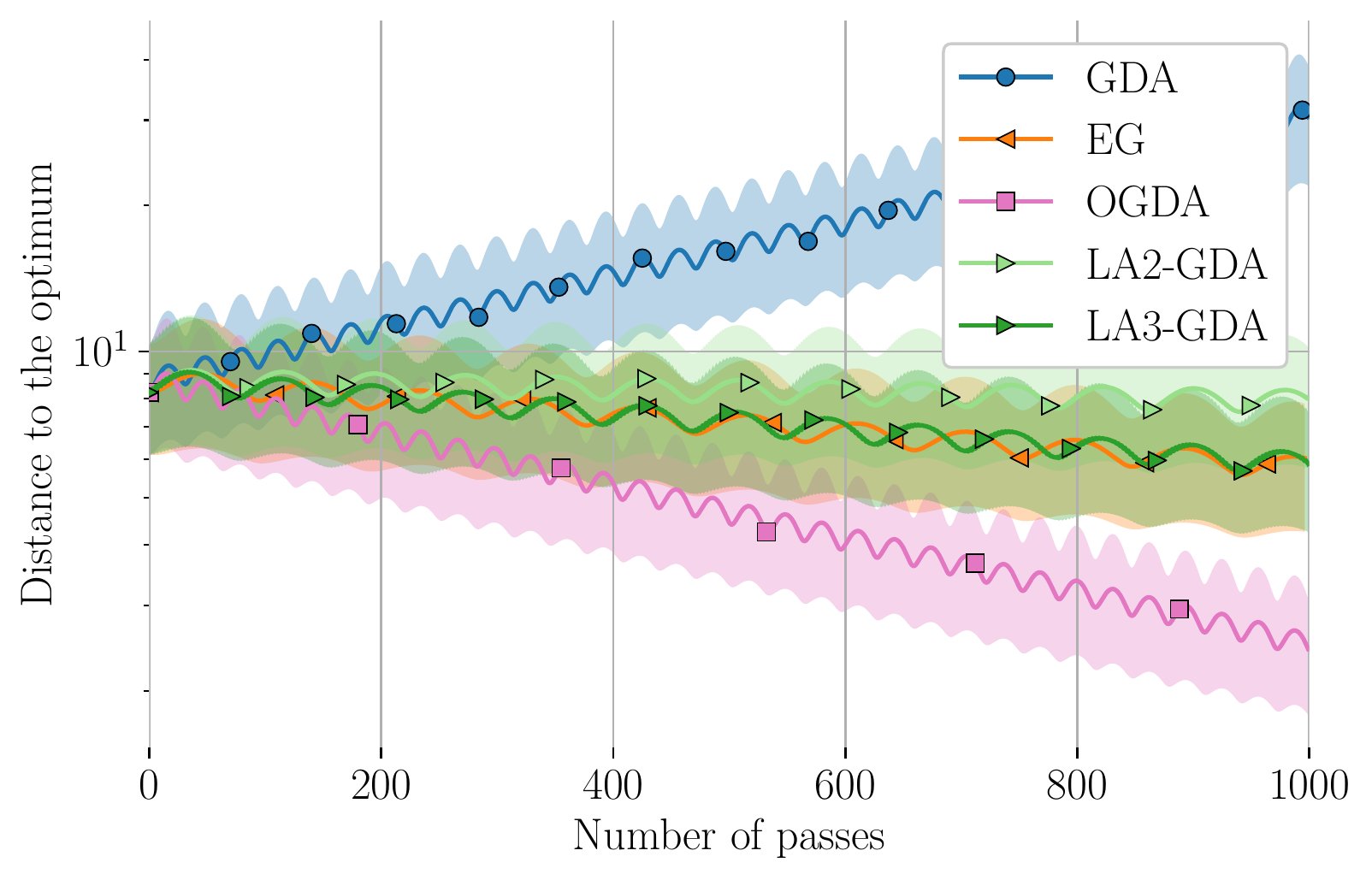}
\caption{
Distance to the optimum (y-axis) of the discrete~\ref{eq:gda},~\ref{eq:extragradient},~\ref{eq:ogda}, ~\ref{eq:lookahead_mm}2-\ref{eq:gda} and ~\ref{eq:lookahead_mm}3-\ref{eq:gda} methods on~\ref{eq:bg}, where the x-axis is normalized by the \emph{number of passes} (gradient queries), and where a fixed step size is used for all methods.}
\label{fig:bg}
\end{figure}

\vm3
\section{Convergence of OGDA for monotone variational inequalities}\label{sec:hr_ogda_monotone_vi}

In this section, we study the convergence of the \ref{eq:ogda_hrde} dynamics of the OGDA method 
in the general setting of MVIs. 
Our convergence analysis is based on Lyapunov functions in both continuous time and discrete time.  
We recall the definition of a Lyapunov function.

\vp1
\begin{definition}
(Lyapunov function)\ \ 
A scalar function $\LL: \R^d \mapsto \R$ is a \emph{Lyapunov function} for a
dynamical system $\dot{\vz} = h(\vz)$ and a limiting set $S \subseteq \R^d$ if: 
\begin{enumar}
\item\vm1
$\LL(\vz) = 0,\quad \forall \vz \in S$, 
    
\item\vp1
$\LL(\vz) > 0,\quad \forall \vz \in \R^d \setminus S$,
    
\item\vp1
$\dot{\LL} (\vz) < 0,\quad \forall \vz \in \R^d \setminus S$.
\end{enumar}
\end{definition}

\vm2
Demonstrating that a Lyapunov function exists guarantees the convergence of the continuous-time dynamics.  An analogous definition applies to the case of discrete-time dynamics.

\vm1
We assume that $V(\cdot)$ is continuously  differentiable, such that the two formulations \eqref{eq:ogda_hrde} and \eqref{eq:ogda_hrde2} are equivalent. 
We study the last-iterate convergence for any MVI under both formulations.

\subsection{Last-iterate convergence of OGDA-HRDE for MVIs}

\vm1
Our argument is based on a combination of two Lyapunov functions, as we explain in detail in Appendix~\ref{app:ogda_hr_mvi_convergence}. 
We start with the convergence of \eqref{eq:ogda_hrde}.

\vp2
\begin{theorem} \label{thm:ogda_hr_mvi_convergence}
Applying the high-resolution ODE for the OGDA dynamics {\rm OGDA- HRDE)} to the monotone variational 
inequality problem in Eq.~\eqref{eq:vi}, with initialization $\vz(0) = \vz_0$, $\vomega(0) = 0$, we have
\vm2
\[
\norm{V(\vz(t))}_2 \le O\left( \sqrt{\beta + \frac{L^2}{\beta}} \right) \cdot \frac{\norm{\vz_0}_2}{\sqrt{t}}. 
\]
Moreover, if the variational inequality is $\mu$-strongly monotone, as per Definition {\rm\ref{def:monotone}}, 
and $1/\rho = \frac{1}{\mu} + \frac{9}{2 \beta}$, then

\good
\vm{10}

\[ 
\norm{\vz(t)}_2 \le O\left( \sqrt{\sfrac{\beta^2 + L^2}{\beta^3}} \right) \cdot \norm{\vz_0}_2 \cdot \exp\left( - \rho  t \right) \,.
\]
\end{theorem}

\vm2
The strongly monotone guarantee is stronger in two ways: (i) the rate is geometric, $\exp(-\rho  t)$, instead 
of $1/\sqrt{t}$, (ii) we can bound the norm $\norm{\vz(t)}_2$, which implies a bound on $\norm{V(\vz(t))}_2$ 
due to the Lipschitzness of $V$, but not vice versa. 
We turn to the convergence of \eqref{eq:ogda_hrde2}.

\vp2
\begin{theorem} \label{thm:ogda_hr_mvi2_convergence}
Applying the high-resolution ODE for the OGDA dynamics {\rm(OGDA- HRDE-2)} to the monotone variational 
inequality problem in Eq.~\eqref{eq:vi}, with initialization $\vz(0) = \vz_0$, $\vw(0) = - \vz_0$  we obtain 
\[
\norm{V(\vz(t))}_2 \le O\left( \sqrt{\kappa + \sfrac{L^2}{\kappa}} \right) 
     \cdot \sfrac{\norm{\vz_0}_2}{\sqrt{t}}. 
\]
Moreover, if the \eqref{eq:vi} is $\mu$-strongly monotone, as per Definition {\rm\ref{def:monotone}}, then
\[
\norm{\vz(t)}_2 \le O\left( \sqrt{\sfrac{\kappa^2 + L^2}{\kappa^3}} \right) 
     \cdot \norm{\vz_0}_2 \cdot \exp\left( - \rho t \right) \,,
\]
where $1/\rho = O(\frac{1}{\mu} + \frac{1}{\kappa})$.
\end{theorem}

Appendix~\ref{app:ogda_hr_mvi2_convergence} provides the proof of Theorem~\ref{thm:ogda_hr_mvi2_convergence}. 
See also our discussion in Appendix~\ref{sec:mvi:eg} on obtaining an analogous result for~\ref{eq:extragradient}.

\subsection{Best-iterate convergence of OGDA for MVIs}\label{sec:best-iterate_ogda_mvi}

We use the techniques we developed thus far to establish some new convergence properties of OGDA in the 
setting of general monotone variational inequalities. 
Since \eqref{OGDA} is strongly related to the differential equation \eqref{eq:ogda_hrde2}, it is natural 
to expect that the Lyapunov functions used to prove Theorem~\ref{thm:ogda_hr_mvi2_convergence} may be useful 
to understand the convergence properties of \eqref{OGDA}. 
Indeed, we can show the following.

\vp1
\begin{theorem} \label{thm:ogda_mvi2_convergence}
If we apply the \eqref{OGDA} dynamics to the monotone variational inequality problem \eqref{eq:vi} with 
$L$-Lipschitz map $V$ and with initialization $\vz_1 = \vz_0$ and step size $\gamma \le \frac{1}{16 \cdot L}$, 
then we have that
\vm2
\[
\min_{i \in [n]} \norm{V(\vz_{i})}_2 \le O\left(\sqrt{\sfrac{1}{\gamma^2} + L^2} 
     \cdot \sfrac{\norm{\vz_0}_2}{\sqrt{n}} \right). 
\]
We also have that $\lim\limits_{n \to \infty} \norm{V(\vz_{n})}_2 = 0$.
\end{theorem}

\vm2
This yields two main conclusions: 
(i) asymptotically, the \eqref{OGDA} method converges to the solution of \eqref{eq:vi}, and 
(ii) the best iterate of the \eqref{OGDA} behaves similarly to the last-iterate convergence of EG shown by 
Golowich et al. \cite{golowich20}. 
Note that we don't require second-order smoothness of $V$ to obtain Theorem~\ref{thm:ogda_mvi2_convergence}; 
moreover, we don't require averaging to get our convergence rates. 
Both of these results are new results for \eqref{OGDA} in the general setting of MVIs. 
We refer to Appendix~\ref{app:ogda_mvi2_convergence} for the proof of Theorem \ref{thm:ogda_mvi2_convergence}.

\good

\subsection{Last-iterate convergence of an implicit discretization of HR-OGDA}\label{sec:implicitDiscretization}

Finally, this section analyzes the implicit discretization of~\ref{eq:ogda_hrde}.  
(See Appendix~\ref{sec:implicitDiscretizationProof} for the derivation).
\begin{equation} \tag{OGDA-I} \label{eq:ogda_i}
\begin{aligned}
\vz_{n + 1} & \!=\! \vz_{n} + \sfrac{\gamma}{2} \cdot \left( \vomega_{n + 1} + \vomega_n \right) \\
  \vomega_{n + 1} & \!=\! - V(\vz_{n + 1}) \!-\! \sfrac{1}{2} \left(V(\vz_{n + 1}) - V(\vz_n)\right)
\end{aligned}
\end{equation}
The following theorem shows that~\eqref{eq:ogda_i} converges for MVI problems.

\vp1
\begin{theorem} \label{thm:ogda_i_mvi_convergence}
  Applying the implicit discretized OGDA dynamics in \eqref{eq:ogda_i} to the
  monotone variational problem \eqref{eq:vi}, with initialization
  $\vz_0$, and $\vomega_0 = 0$, we obtain
\[ 
\norm{V(\vz_n)}_2 \le \mathcal{O}\left( \sqrt{\beta + \sfrac{L}{\beta}} \right) 
   \cdot \sfrac{\norm{\vz_0}_2}{\sqrt{n}}. 
\]
Moreover, if the \eqref{eq:vi} is $\mu$-strongly monotone, as per Definition \ref{def:monotone}, then
\[ 
\norm{\vz_n}_2 \le \mathcal{O}\left( \sqrt{\sfrac{\beta^2 + L}{\beta^3}} \right) 
     \cdot \norm{\vz_0}_2 \cdot \exp\left( - \mathcal{O}(\rho) \cdot n \right) \,, 
\]

\vm5
where $1/\rho = \frac{1}{\mu} + \frac{1}{\beta}$.
\end{theorem}
The proof of Theorem~\ref{thm:ogda_i_mvi_convergence} is presented in Appendix~\ref{sec:implicitDiscretizationProof}.

\section{Discussion}\label{sec:discussion}

Despite qualitative differences in their convergence, the fact that different saddle-point optimizers yield 
the same ODE motivated our study of high-resolution differential equations (HRDEs).  
Our work extends earlier work of Shi et al. \cite{shi2018hrde} to the setting of saddle-point optimization 
and monotone variational inequalities. 
We derived HRDEs corresponding to extragradient (EG), optimistic gradient descent ascent (OGDA), as well as 
lookahead-minmax when combined with gradient descent ascent (LAk-GDA).
Using these HRDEs we then established the convergence of these methods in continuous time on bilinear games, 
matching the qualitative convergence profiles of these methods in experiments.

Moreover, the HRDE representation allowed us to obtain convergence proofs in continuous time for the problem 
of monotone variational inequalities, specifically for the optimistic gradient descent ascent method. 
In this general setting, we showed
(i) last-iterate convergence of an implicit discretization of the HRDE dynamics, and
(ii) best-iterate convergence of an explicit discretization that corresponds to the original OGDA method.  
These results required only first-order smoothness assumptions.

There are several potential future directions for research in this vein. 
While HRDEs allow for modeling the observed performance of optimistic gradient descent ascent, the study of 
differential equations with an even higher resolution might open a path to proving analogous results for 
extragradient and lookahead-minmax. 
Also, it would be of interest to study different discretizations of the HRDEs we have presented; 
these could potentially generate new variants of the classical methods that have superior numerical performance.

\appendix

\section{Additional related work and discussion}\label{app:additional_related_works}

In this section, we provide additional overview and discussion to complement the discussion presented in \S\ref{sec:intro} and \S~\ref{sec:related_works}.

\vm2
\subsection{Convergence results on bilinear and strongly monotone games}

Several lines of work establish the last-iterate convergence of algorithms for bilinear or strongly monotone games. For example:
\begin{enumrm}
\item 
Daskalakis et al. \cite{daskalakis2018training} prove last-iterate convergence of \emph{optimistic mirror decent} (OMD) on bilinear games,

\item\vp1
Gidel et al. \cite{gidel19-momentum} prove the convergence of alternating GDA with negative momentum on bilinear games,

\item\vp1
Mokhatari et al. \cite{mokhtari2019unified} use a Proximal Point perspective to show the convergence of OGDA and EG on bilinear and strongly convex-strongly concave games, and

\item\vp1
Azizian et al. \cite{azizian20tight} establish linear convergence on bilinear and strongly monotone games, 
assuming that the singular values of the Jacobian of the joint vector field (defined in \S\ref{sec:preliminaries}) 
is bounded below by strictly positive constant.
\end{enumrm}
Loizou et al. \cite{loizou20} and Abernathy et al. \cite{abernethy2021lastiterate} focus on  Hamiltonian 
gradient descent, and prove its last-iterate convergence in deterministic and stochastic settings, respectively, 
both for bilinear and ``sufficiently'' bilinear (possibly non-convex) games.

\vm1
The lookahead-minmax algorithm [LA, \cite{chavdarova2021lamm}] algorithm was proposed to exploit the 
rotational game dynamics explicitly, by periodically restarting the current iterate to lie on a line 
between two $k$-step separated iterates produced by a given ``base'' optimizer. 
\ref{eq:lookahead_mm} is an attractive choice for min-max optimization due to its performance on both 
bilinear games and challenging real-world min-max applications such as GANs, as well as its negligible 
computational overhead. 
In this work, we proved its convergence on bilinear games when combined with the otherwise non-converging 
GDA -- a result that to the best of our knowledge has not shown before.

\subsection{Relating discrete-time methods to continuous-time dynamics}

There are two main approaches for finding a correspondence between a discrete-time optimization method 
and a  continuous-time dynamical system. 
In single-objective minimization, the most commonly used approach is to derive a low-resolution ODE.  
Often these ODEs differ for different optimization methods, and one uses heuristics to decide which 
discretization of the ODE seems most informative regarding the particular method being studied 
[see, e.g., \cite{guilherme2021}]. 
On the other hand, sometimes one obtains the same ODE when starting from different discrete-time methods, 
which prevents using the continuous-time ODE for making distinctions among these methods.  
This phenomenon typically occurs when there is a difference in the acceleration terms among the methods.

\vm1
In the context of game-theoretic problems, however, essentially all of the different first-order optimizers 
yield identical low-resolution ODEs \cite{hsieh2020limits}. 
Moreover, the resulting ODE is identical to that of GDA, an algorithm which is known to diverge for bilinear 
games and which consistently underperforms other popular min-max methods on challenging non-convex problems.

\good

\subsection{Comparisons to other HRDE analyses}\label{app:derivation_difference}

\textbf{Comparison to Shi et al. \cite{shi2018hrde}.}
Our HRDE methodology is closely related to that of Shi et al. \cite{shi2018hrde}, but differs in the left-hand side
 of \ref{eq:general_form}. 
More precisely, in our approach we perform Taylor expansion of \emph{all} the $\vz$ terms whose iteration index is different from $n$ for consistency. 
This is different from \cite{shi2018hrde} where Taylor expansion is done \emph{solely} on the right-hand side of Eq.~\ref{eq:general_form}.
Note that this change yields an identical left-hand side for many of the methods.  
The analysis in \cite{shi2018hrde} focuses on comparing Nesterov’s accelerated gradient method and  Polyak’s heavy-ball method, and using Taylor expansion on the left-hand side of~\ref{eq:general_form} does not bring novel interpretable insights on the differences between the two methods.

However, in general we believe that the HRDE derivation method should extend the low-resolution ODE methodology by consistently using Taylor expansion on \emph{all} terms whose index differs from $n$. 
Indeed, it is straightforward to show that when not performing Taylor expansion on all the $\vz$ terms whose index is different than $n$ it can be the case that two significantly different discrete methods can have identical continuous-time dynamics.  
This is done by deriving one version of the dynamics using $\gamma\rightarrow0$ and the other using $\gamma^2\rightarrow0$.
Moreover, in Theorem~\ref{thm:hrde_apt_ogda} in Appendix \ref{app:apt_ogda}, we show that over time intervals, the~\ref{eq:ogda_hrde2} more closely models the original discrete~\ref{eq:ogda} method with decreasing step size. 
It is worth noting that this result is stronger than showing that the error at a fixed step is upper bounded since the error can occur in a biased way over time -- ometimes called ``global'' (over time) and a ``local'' (at fixed step) discretization error \cite{guilherme2021}.

To summarize, our methodology to derive high-resolution ODEs  generalizes the standard low-resolution ODE derivation. 
The resulting HRDE is more complex -- e.g., the $\mathcal{O}(\gamma)$--Resolution ODE is a second-order ODE -- but 
(i) we are ultimately interested in finding continuous-time dynamics that accurately models the original discrete 
method relative, and (ii) there exist ways to re-write higher-order ODEs as systems of lower-order equations, as 
we show in \S\ref{sec:ogda_hrde_equivalent}, which simplifies the analysis.

\textbf{Comparison to Lu \cite{lu2021osrresolution}.}
Lu \cite{lu2021osrresolution} proposes a novel method to derive ODEs of higher precision via a derivation that 
consists of a recursive procedure to obtain the coefficients of the higher resolution  ODE.  
The ODEs derived by Lu \cite{lu2021osrresolution} of GDA and EG are very similar [see Corollary 1 of 
\cite{lu2021osrresolution}]: 
\begin{align*}
 \text{GDA:}\qquad \dvz (t) &= - V\big(\vz(t)\big) - \sfrac{\gamma}{2} \cdot J\big(\vz(t)\big) \cdot V\big(\vz(t)\big)  \\
 \text{EG:}\qquad \dvz (t) &= - V\big(\vz(t)\big) + \sfrac{\gamma}{2} \cdot J\big(\vz(t)\big) \cdot V\big(\vz(t)\big).
\end{align*}
One can observe that the two differential equations differ only in the sign of the second term on the right-hand 
side, and these terms depend linearly on the step size, which is assumed to be positive. 
Our generalized derivation on the other hand, yields rather different HRDEs in these two cases, yielding different convergence guarantees for bilinear games.

\section{Proofs relevant to the HRDE derivation}\label{app:hrdes}

This section provides a detailed discussion of the derivation of HRDEs.

\subsection{On the derivation of HRDEs}\label{app:hrdes_derivation}

As explained in detail in Shi et al. \cite{shi2018hrde}, the key difference between high-resolution differential 
equations and ODEs is that the $\mathcal{O}(\gamma)$ terms are preserved when deriving the differential 
equation, while all the terms $\mathcal{O}(\gamma^2)$ and higher are assumed to go to zero. 
As there is no compelling reason to consider terms solely up to first order in $\gamma$, it is natural 
to study differential equations that also preserve, for example, the $\mathcal{O}(\gamma^2)$ terms while 
assuming that the $\mathcal{O}(\gamma^3)$ terms go to zero. 
In this paper, we only analyze the first level of high-resolution differential equations, where we keep 
all the $\mathcal{O}(\gamma)$ terms and we set to zero all the $\mathcal{O}(\gamma^2)$ and higher-order terms. 
We show that this resolution is sufficient to capture the different behavior of these methods on in the setting 
of bilinear games. 
Additionally, as we show in \S\ref{sec:hr_ogda_monotone_vi}, this level of resolution is sufficient to 
prove the last-iterate convergence of the OGDA dynamics in general monotone variational inequality problems. 
Nonetheless, we further argue in \S\ref{sec:mvi:eg} that this resolution may \emph{not} be sufficient to model 
the behavior of the extragradient method for general monotone variational inequality problems, and higher 
resolution may be needed.

\subsubsection{Different choice of step size ratio}\label{app:step-size_ratio}

In \S\ref{sec:hrde}, we clarify that although $\delta=\gamma$ is the standard choice when deriving the 
low-resolution ODE, one can use a different ratio. 
As an example, we consider the GDA method for simplicity and use a different ratio than the one considered 
in the main paper (where  $\delta=\gamma$).

From the~\ref{eq:gda} update rule we have:
\begin{equation}\label{appeq:gda_starting}
    \sfrac{\vz_{n+1}-\vz_n}{\gamma} = - V(\vz_n) \,.
\end{equation}
Substituting, we obtain:
$$
\vz_{n+1} \approx \vz\big((n+1)\delta\big) = \vz(n\delta) + \delta \dot{\vz} (n\delta) 
   + \sfrac{\delta^2}{2}\ddot{\vz}(n\delta) + \mathcal{O}(\delta^3)
$$
in~\eqref{appeq:gda_starting} yields:\hskip16mm 
$\dps \sfrac{\delta \dot{\vz} (n\delta) + \frac{\delta^2}{2}\ddot{\vz}(n\delta) + \mathcal{O}(\delta^3)}{\gamma} 
   = - V \big(\vz(n\delta)\big)$.

Finally, setting $\delta=\sqrt{2\gamma}$ yields the following HRDE:
$$
\ddot\vz(t) = - \sqrt{\sfrac{2}{\gamma}} \dot\vz(t) - V\big(\vz(t)\big) \,,
$$
or equivalently, denoting $\dot\vz (t) = \vomega (t)$:
\begin{equation}\label{appeq:gda_hrde_diff_step-size_ratio}
\dot\vomega = - \sqrt{\sfrac{2}{\gamma}} \vomega(t) - V\big(\vz(t)\big).
\end{equation}
Note that this does not change the sign in front of the terms in \eqref{appeq:gda_hrde_diff_step-size_ratio}, 
thus  one can use the proof technique in Appendix~\ref{app:bilinear_gda_convergence} to show that the above 
dynamics diverges for bilinear games.

In summary, using different step size ratios results in \emph{different} HRDEs. 
However, the convergence properties of the resulting HRDEs will be closer to that of the starting method, 
relative to the low-resolution ODE, as our convergence results show.

\subsection{Detailed derivation of LA3-GDA}\label{app:la3-gda}

In this section, we provide a more detailed derivation of the HRDE of LA3-GDA. 
The presentation in~\S\ref{sec:hrde_sp} omitted details  as the derivation is analogous to that of LA2-GDA.

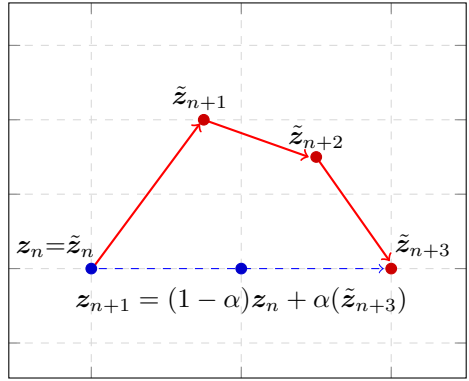
\begin{figure}[!ht]
    \centering
\begin{tikzpicture}
\begin{axis}[
  axis equal,
  grid=major,
  grid style={dashed, gray!30},
  enlargelimits=false,
  width=7.6cm,
  xmin=-1.1,
  xmax=5,
  ymin=0,
  ymax=2.1,
  yticklabels={,,},  
  xticklabels={,,} 
]

\addplot[color=blue!80!black, only marks, style={mark=*, fill=blue!80!black}] coordinates {(0,0)};
\addplot[color=blue!80!black, only marks, style={mark=*, fill=blue!80!black}] coordinates {(2,0)};
\addplot[color=red!80!black, only marks, style={mark=*, fill=red!80!black}] coordinates {(1.5,2)};
\addplot[color=red!80!black, only marks, style={mark=*, fill=red!80!black}] coordinates {(3,1.5)};
\addplot[color=red!80!black, only marks, style={mark=*, fill=red!80!black}] coordinates {(4,0)};

\draw[->, thick, red](115,5) -- (255,195);
\draw[->, thick, red](260,200) -- (400,150);
\draw[->, thick, red](410,150) -- (508,10);
\draw[->, dashed, blue](120,00) -- (500,0);

\node[left] at (130, 30) {\footnotesize $\vz_{n}{=}\tilde \vz_{n}$};
\node[below] at (310, -10) {\footnotesize $\vz_{n+1} = (1-\alpha)\vz_n+ \alpha (\tilde \vz_{n+3})$};
\node[right] at (500, 30) {\footnotesize $\tilde \vz_{n+3}$};
\node[above] at (410, 145) {\footnotesize $\tilde \vz_{n+2}$};
\node[above] at (255, 198) {\footnotesize $\tilde \vz_{n+1}$};
\end{axis}
\end{tikzpicture}
\caption{Pictorial representation of LA3--GDA  ($k=3$), with $\alpha=\frac{1}{2}$. 
The joint vector field $V(\cdot)$ is depicted with a red line. 
The lookahead-minmax iterate with $k=3$ is obtained by taking a point on the line between $\vz_n$ and 
$\tilde\vz_{n+3}$, depicted as a dashed blue line.}
\label{fig:la-gda_k3}
\end{figure}

The predicted iterates for $k=1, \dots, 3$, given \ref{eq:gda} as a base optimizer are as follows:
\begin{align*}
    \tilde\vz_{n+1} 
&= \vz_n - \gamma V(\vz_n)
    \tag{$\tilde\vz_{n+1}^{\text{GDA}}$}\label{eq:predicted_gda_1}\\[1mm]
    \tilde\vz_{n+2} 
&= \tilde\vz_{n+1} - \gamma V(\tilde\vz_{n+1}) = \vz_n - \gamma V(\vz_n) - \gamma V(\vz_n - \gamma V (\vz_n) )
    \tag{$\tilde\vz_{n+2}^{\text{GDA}}$}\label{eq:predicted_gda_2}\\[1mm]
    \tilde\vz_{n+3} 
&= \tilde\vz_{n+2} - \gamma V(\tilde\vz_{n+2}) 
\tag{$\tilde\vz_{n+3}^{\text{GDA}}$}\label{eq:predicted_gda_3} \\
& = \vz_n \!-\! \gamma V(\vz_n) \!-\! \gamma V(\vz_n \!-\! \gamma V (\vz_n) ) 
    \!-\! \gamma V\big( \vz_n \!-\! \gamma V(\vz_n) \!-\! \gamma V(\vz_n \!-\! \gamma V (\vz_n) ) \big). \hskip5mm\notag
\end{align*}
As depicted in Fig.~\ref{fig:la-gda_k3}, for~\ref{eq:lookahead_mm}3-GDA we have:
\begin{align}\label{eq:la3_gda_interm1}
    \vz_{n+1} 
& = \vz_n + \alpha(\tilde \vz_{n+3} - \vz_n) \\
& = \vz_n \!+\! \alpha\gamma\big[
    \!-\! V(\vz_n) \!-\! V\big(\vz_n \!-\! \gamma V(\vz_n)\big)
    \!-\! \underbrace{V\big( \vz_n \!-\! \gamma V(\vz_n) \!-\! \gamma V(\vz_n 
    \!-\! \gamma V(\vz_n)) \big)}_{(\star)} \big]. \notag
\end{align}

\vm6
Similarly, by performing a Taylor expansion in coordinate space for ($\star$),  we get:
\begin{align*}
&    V\big( \vz_n - \gamma V(\vz_n) - \gamma V(\vz_n - \gamma V(\vz_n)) \big) \\
& = V(\vz_n) - \gamma J (\vz_n)V(\vz_n) - \gamma J(\vz_n)V\big(\vz_n - \gamma V(\vz_n) \big) + \mathcal{O}(\gamma ^2) \\[1mm]
& = V(\vz_n) - 2\gamma J (\vz_n)V(\vz_n)  + \mathcal{O}(\gamma ^2) \,,
\end{align*}
where in the second row we do an additional Taylor expansion of the last term in the preceding row.

Thus, using \eqref{eq:time-taylor} as well as replacing the above in~\eqref{eq:la3_gda_interm1} we have: 
\begin{align*}
\sfrac{\dot{\vz}(n\delta)\delta \!+\! \frac{1}{2} \ddot{\vz}(n\delta)\delta^2  + \mathcal{O}(\delta^3)}{\gamma} 
& = \alpha \Big\{ - 3 V(\vz(n\delta)) + 3\gamma J (\vz(n\delta)) V(\vz(n\delta)) + \mathcal{O} (\gamma^2)
\Big\}
\end{align*}
Setting $\delta \!=\! \gamma$ and keeping the $\mathcal{O}(\gamma)$ terms yields: 
\begin{align*}
\dot\vz(t) + \sfrac{\gamma}{2}\ddot\vz(t) = - 3\alpha V(\vz(t)) + 3 \alpha \gamma J (\vz(t)) V(\vz(t)) \,,
\end{align*}
Rewriting this result in phase space gives:
\begin{equation}\tag{LA3-GDA-HRDE}\label{eq:la3-gda_hrde3}
\begin{split}
  \dot{\vz}(t) & = \vomega(t) \\
  \dvom(t)     & = - \sfrac{2}{\gamma} \vomega(t) - \sfrac{6\alpha}{\gamma} V(\vz(t)) + 6 \alpha \cdot J(\vz(t)) \cdot V(\vz(t))  \,.
\end{split} 
\end{equation}

\vm4
\subsection{Proof of Proposition \ref{prop:unique_solution}: Unique solution of the HRDEs}\label{app:unique_solution}

The uniqueness of the solutions for our HRDEs for the case of bilinear games follows from standard results 
on the uniqueness of the solutions in the case of linear dynamical systems. 

For the more general case of Proposition~\ref{prop:unique_solution}, and for the HRDE of OGDA we follow the 
proof of Proposition 2.1 of \cite{shi2018hrde}. 
The differential equations of OGDA have the form $\dot{\vu} = \mathbf{G}(\vu)$, where $\vu = (\vz, \vomega)$ 
and $\mathbf{G}$ is the vector field defined in \eqref{eq:ogda_hrde}. 
If we show that $\mathbf{G}$ is Lipschitz then we can apply Theorem 10 of \cite{shi2018hrde} and Proposition 
\ref{prop:unique_solution} follows. 
Now to show the Lipschitzness of $\mathbf{G}$ we have from our assumptions the following:
\begin{equation}\tag{L1}
   \norm{V(\vz_1) - V(\vz_2)}_2 \leq L_1 \norm{\vz_2 - \vz_1}_2;
\end{equation}
\begin{equation}\tag{L2}
   \norm{J(\vz_1) - J(\vz_2)}_2 \leq L_2 \norm{\vz_1  - \vz_2}_2.
\end{equation}
Also, from the two Lyapunov functions $\LL_1$, $\LL_2$ that we define in \S\ref{app:ogda_hr_mvi_convergence} 
we have that for any initial condition $(\vz_0, \vomega_0)$ both the norm of $\vomega(t)$ and the norm of 
$\vz(t)$ will be upper bounded for all future times by a constant $\mathcal{C}_1$ that depends only on $\vz_0$ 
and $\vomega_0$. 
Because of the Lipschitzness of $V$ and $J$ and the boundedness of $\vomega(t)$ we find that the norm of 
$\dot{w}(t)$ is also bounded by a constant $\mathcal{C}_2$ that also only depends on $\vz_0$ and $\vomega_0$. 
Therefore we have the following:
\begin{gather}\tag{C1}
    \sup_{0\leq t \leq \infty} ||\vomega(t)|| \leq \mathcal{C}_1,\\
\tag{C2}
   \sup_{0\leq t \leq \infty} ||\dvom(t)|| \leq \mathcal{C}_2.
\end{gather}
From \ref{eq:ogda} we have:\ \ \ 
$\dps \frac{d}{dt} \begin{bmatrix}  \vz_s \\ \vomega_s \end{bmatrix} 
= \begin{bmatrix} \vomega_s \\  - \frac{2}{\gamma} \vomega_s - \frac{2}{\gamma} V(\vz_s) 
   - 2 J(\vz_s)\vomega_s   \end{bmatrix}$ \,.

Thus, for any $[\vz_1, \vomega_1]^\intercal$, $[\vz_2, \vomega_2]^\intercal$ with the initial condition
bounds on the norm of $\vz_1$, $\vz_2$, $\vomega_1$, $\vomega_2$, we have:
\vm2
\begin{align*}
   & \begin{Vmatrix} \begin{bmatrix} \vomega_1 \\  - \frac{2}{\gamma} \vomega_1 - \frac{2}{\gamma} V(\vz_1) - 2 J(\vz_1)\vomega_1     \end{bmatrix}  
    - \begin{bmatrix} \vomega_2 \\  - \frac{2}{\gamma} \vomega_2 - \frac{2}{\gamma} V(\vz_2)  - 2 J(\vz_2)\vomega_2    \end{bmatrix} \end{Vmatrix} \\
&\leq \begin{Vmatrix} 
      \vomega_1 - \vomega_2 \\   -\frac{2}{\gamma} (\vomega_1-\vomega_2)
      \end{Vmatrix} 
   + \frac{2}{\gamma} 
     \begin{Vmatrix} 
     0 \\ (V(\vz_1)-V(\vz_2)) 
     \end{Vmatrix}   
   + 2 \begin{Vmatrix} 0 \\ J(\vz_1)(\vomega_1 - \vomega_2) \end{Vmatrix}  \\
&\hskip5mm + 2 \begin{Vmatrix} 0 \\ \vomega_2( J(\vz_1)-J(\vz_2) ) \end{Vmatrix}     \\
\noalign{\good}
& \leq \sqrt{1 + \sfrac{4}{\gamma^2}} ||\vomega_1-\vomega_2|| + \sfrac{2}{\gamma}L_1 ||\vz_1-\vz_2||  
   + 2 \underbrace{||J(\vz_1)||}_{\leq L_2 ||\vz_1||} ||\vomega_1 - \vomega_2 || \\
&\hskip5mm + 2 \underbrace{||\vomega_2 ||}_{\leq \mathcal{C}_1} 
     \underbrace{|| J(\vz_1) - J (\vz_2) ||}_{\leq L_2 ||\vz_1 - \vz_2||} \\
&\leq   (\sfrac{2}{\gamma}L_1 + 2\mathcal{C}_1 L_2 ) ||\vz_1 - \vz_2|| +  (\sqrt{1 + \sfrac{4}{\gamma^2}}   + (L_2 ||\vz_1||)   )   ||\vomega_1-\vomega_2|| \\
&\leq 2 \max \{  \sfrac{2}{\gamma}L_1 + 2\mathcal{C}_1 L_2 , \sqrt{1 + \sfrac{4}{\gamma^2}}   + (L_2 ||\vz_1||) \}
   \begin{Vmatrix}
   \begin{bmatrix} \vz_1\\\vomega_1 \end{bmatrix} - \begin{bmatrix} \vz_2\\\vomega_2 \end{bmatrix}
   \end{Vmatrix},
\end{align*}
and finally, we can use the initial bounds that we have on the norms of $\vz_1$, $\vz_2$, $\vomega_1$, $\vomega_2$ 
to obtain that the vector field $\mathbf{G}$ of OGDA is Lipschitz.  
Hence we can apply Theorem 10 of \cite{shi2018hrde} and Proposition \ref{prop:unique_solution} follows.

\subsection{Proof of Theorem \ref{thm:equivalent_ogda_hr_mvi}:  Equivalent forms of \ref{eq:ogda_hrde}}

\proof[\textbf{Proof of Theorem~\ref{thm:equivalent_ogda_hr_mvi}}]
We first show that any $\vz(t)$ that is a solution to (OGDA- HRDE) is also a solution to \eqref{eq:ogda_hrde2}. 
To do this, we define the function:
\[ 
\vw(t) = -\sfrac{2}{\beta} \dot{\vz}(t) - \sfrac{4}{\beta} V(\vz(t)) - \vz(t) \Rightarrow \ddot{\vz}(t) 
   = - \sfrac{\beta}{2} \dot{\vw}(t) - 2 J(\vz(t)) \dot{\vz}(t) - \sfrac{\beta}{2} \dot{\vz}(t) \,. 
\]
Observe that $\vw(t)$ is a continuous differentiable function and using 
\eqref{eq:ogda_hrde} we have that
\vm1
\begin{equation}
  \dot{\vw}(t) = \dot{\vz}(t) + 2 V(\vz(t)). \label{eq:proof:equivalence:1:1}
\end{equation}
Now if we substitute into the definition of $\vw$ we obtain:
\[
\vw(t) = -\sfrac{2}{\beta} \dot{\vw}(t) - \vz(t) \Rightarrow \dot{\vw}(t) 
     = - \sfrac{\beta}{2} \vz(t) - \sfrac{\beta}{2} \vw(t). 
\]
Combining this with \eqref{eq:proof:equivalence:1:1} we have that
\[ 
\dot{\vz}(t) = - \sfrac{\beta}{2} \vz(t) - \sfrac{\beta}{2} \vw(t) - 2 V(\vz(t)), 
\]
and therefore the tuple $(\vz(t), \vw(t))$ also satisfies the system of equations (OGDA- HRDE-2).

For the other direction, we first differentiate the first equation of \eqref{eq:ogda_hrde2}, and we get that
\[ 
\ddot{\vz}(t) = -\sfrac{\beta}{2} \dot{\vz}(t) - \sfrac{\beta}{2} \dot{\vw}(t) - 2 J(\vz(t)) \cdot \dot{\vz}(t).
\]
We then observe that subtracting the two equations of \eqref{eq:ogda_hrde2} yields
\[ 
\dot{\vw}(t) = \dot{\vz}(t) + 2 V(\vz(t)). 
\]
Finally, substituting the first equation into the second one and setting 
$\vomega(t) = \dot{\vz}(t)$ we obtain the equations \eqref{eq:ogda_hrde}.
\endproof

\subsection{HRDE-OGDA as an asymptotic pseudotrajectory of OGDA: Using Decreasing Step size}\label{app:apt_ogda}

Following Theorem 1 of \cite{hsieh2020limits}, we show that \ref{eq:ogda_hrde} closely models an interpolated 
version of the discrete~\ref{eq:ogda} method over time, assuming strongly monotone variational 
inequalities.

{\bf OGDA with decreasing step size.}
In this section, in line with stochastic analyses, we consider OGDA with variable step size $\gamma_n$, that is: 
\begin{equation}\label{eq:ogda-gamma-n} \tag{OGDA-$\gamma_n$}
    \vz_{n+1} = \vz_n - 2 \gamma_n V(\vz_n) + \gamma_n V(\vz_{n-1}) \,.
\end{equation}
Combining \eqref{eq:ogda-gamma-n} with \eqref{eq:time-taylor} we have:
$$
(\dot{\vz} (n\delta)\delta + \sfrac{1}{2} \ddot{\vz} (n \delta) \delta^2)/\gamma(n\delta) 
   = - V(\vz(n\delta)) - \delta J\big(\vz(n\delta) \big) \dot{\vz} (n\delta).
$$
Choosing that at time $n\delta$,  $\gamma(n\delta) = \delta$, and keeping the $\mathcal{O}\big(\gamma(n\delta)\big)$ 
terms we obtain:
$$\
dot{\vz}(t) + \sfrac{\gamma(t)}{2} \ddot{\vz}(t) =  -V(\vz(t)) - \gamma(t) J(\vz(t))\dot{\vz}(t) \,,
$$
yielding the following HRDE: 
\begin{equation}\tag{OGDA-HRDE-$\gamma(t)$}\label{eq:ogda_hrde_gamma-t}
\begin{split}
\dot{\vz}(t) & = \vomega(t) \\
\dvom(t) &= - \sfrac{2}{\gamma(t)} \cdot \vomega(t) - \sfrac{2}{\gamma(t)} \cdot V(\vz(t)) - 2 \cdot J(\vz(t)) \cdot \vomega(t).
\end{split}
\end{equation}
Similarly to obtaining~\ref{eq:ogda_hrde2} of the~\ref{eq:ogda_hrde}, and taking into account that now $\kappa$ depends on $t$ we have:
 \begin{equation} \tag{OGDA-HRDE-$\gamma(t)$-2} \label{eq:ogda_hrde_gamma-t_2}
        \begin{split}
        \dot{\vz}(t)     & = - \kappa(t) \cdot \vz(t) - \kappa(t) \cdot \vw(t) - 2 V(\vz(t)) \\[1mm]
        \dot{\vw}(t)     & = - \kappa(t) \cdot \vz(t) - \kappa(t) \cdot \vw(t)
        \end{split} \,,
    \end{equation}

\vm4
where $\dps\kappa(t) = \sfrac{\beta(t)}{2} = \sfrac{1}{\gamma(t)}$.

Later, we will use the following Lemma on the convergence of the~\ref{eq:ogda_hrde_gamma-t_2} dynamics for monotone VI problems.

\vp2
\begin{lemma}\label{lm:conv_ogda_hrde_gamma-t_2}
{\rm(Convergence of \ref{eq:ogda_hrde_gamma-t_2} on strongly monotone VIs)}\ \ 
If $\beta(t) \dot \beta(t) < 4 \mu$, where $\beta=\frac{1}{\gamma}$ and $\mu$ is the strong monotonicity constant; 
then, the high-resolution ODE for the \eqref{eq:ogda-gamma-n} dynamics \eqref{eq:ogda_hrde_gamma-t_2} converges 
to the solution of the $\mu$-strongly monotone variational inequality problem as per Definition 
{\rm\ref{def:monotone}}. 
\end{lemma}

\vm4
\begin{proof}
We let
\begin{align} \tag{OGDA-$\gamma(t)$-2-L} \label{eq:lyapunov_ogda2_gamma2}
    \LL(\vz, \vw) & = 
    \sfrac{1}{2}\left( \norm{\beta(t) \vz(t) + \vw(t)}^2_2 + \norm{ \vw(t) - \beta(t)\vz(t) }_2^2 \right) 
\end{align}
be the Lyapunov function. 
In the remaining, we drop $(t)$.

From the definition of the monotone variational inequality problem and from the positivity of the norms, we 
have that $\LL(\vz, \vw) > 0$ for any $(\vz, \vw) \neq (0, 0)$, and also $\LL(0, 0) = 0$. 
Next, we expand $\dot{\LL}$:
\begin{align}\label{eq:proof_ogda2_monotone_1st-derv2}
\dot{\LL}(\vz, \vomega) &= 
\underbrace{\langle \beta\vz + \vw,  \dot\beta \vz + \beta \dot{\vz} + \dot{\vw} \rangle}_{(1)} 
     + \underbrace{\langle \vw - \beta\vz,  \dot{\vw} - \beta \dot{\vz} - \dot\beta \vz \rangle}_{(2)} \,.
\end{align}

\vm4
We now analyze each of these terms separately, 
\begin{enumar}
\item
Replacing $\dot{\vz}$ and $\dot{\vw}$ from ~\ref{eq:ogda_hrde_gamma-t_2}, we get 
$\langle \beta\vz + \vw,  \dot\beta \vz + \beta \dot{\vz} + \dot{\vw} \rangle = - 
\norm{\vz + \vw}_2^2 - 2 \langle \vz, V(\vz) \rangle - 2 \langle \vw, V(\vz) \rangle + \beta\dot\beta \norm{\vz}^2_2 + \dot\beta \langle  \vomega, \vz \rangle$;

\item\vp1
Replacing $\dot{\vz}$ and $\dot{\vw}$ from ~\ref{eq:ogda_hrde_gamma-t_2}, we get 
$\langle \vw - \beta\vz,  \dot{\vw} - \beta \dot{\vz}  - \dot\beta\vz \rangle 
     = - 2 \langle \vz, V(\vz) \rangle + 2 \langle \vw, V(\vz) \rangle 
     -\dot\beta \langle \vomega, \vz \rangle + \beta\dot\beta \norm{\vz}^2_2$.
\end{enumar}

Using all of the above, we get that
\begin{align} 
    \dot{\LL}(\vz, \vw) & = -  \norm{\vz + \vw}_2^2 - 4 \langle \vz, V(\vz) \rangle  + 2 \beta\dot\beta \norm{\vz}^2_2 \\
    & \leq -  \norm{\vz + \vw}_2^2  + (2 \beta\dot\beta - 4\mu) \norm{\vz}^2_2 \,.
\end{align}
where for the last inequality, we used the strong monotonicity property of $V$ that:
\begin{equation} 
    \langle \vz, V(\vz) \rangle \ge \mu \norm{\vz}_2^2.
\end{equation}
We define the set $S = \{ (\vz, \vw) \in \R^{2d} : V(\vz) = 0\}$. 
We have that if $\beta\dot\beta < 2\mu$, then $\dot{\LL}(\vz, \vomega) \le 0$,
for all $(\vz, \vw) \in \R^{2 d}$, and also
$\dot{\LL}(\vz, \vomega) < 0$ for all $(\vz, \vw) \not\in S$. 
Hence, $\LL$ is a Lyapunov function for our problem, provided we appropriately choose the step size and rate 
of decrease to the $\mu$ constant.
\end{proof}

{\bf Preliminaries.}
The discrete \emph{Robbins-Monro} scheme is defined as:
\begin{equation}
    \tag{RM}\label{eq:rm}
    \vz_{n+1} = \vz_{n} + \gamma_n [-V(\vz_n) + \mathcal{W}_{n+1}] \,,
\end{equation}
where $\mathcal{W}_{n+1}\!=\!\mathcal{B}_{n+1} + \mathcal{R}_{n+1}$ is an abstract term composed of both systematic (bias) component $\mathcal{B}_{n+1}$ and a random zero-mean noise  $\mathcal{R}_{n+1}$, at step $n\!+\!1$.  
Defined as such, an RM scheme can represent all the min-max optimizers listed in this work (and, in fact, all 
first-order methods), see \cite[\S 3.2.]{hsieh2020limits} for the exact values of the bias and the random noise 
terms for each of these. 
In particular, the bias term in the~\ref{eq:rm} framework for \ref{eq:ogda-gamma-n} is
\ \ $\mathcal{B}_{n+1} = - V(\vz_n) +  V(\vz_{n-1})$.

To obtain a continuous-time variable from the discrete $\{\vz_n\}$ iterates we define $\bar{\vz}(t)$ as the piecewise linear interpolation between the iterates $\{\vz_n\}$:
\begin{equation}
    \tag{LI}\label{eq:linear_interpolation}
    \bar{\vz}(t) = \vz_n + \sfrac{t-\tau_n}{\tau_{n+1}-\tau_n} (\vz_{n+1} - \vz_n ) \,, \quad \forall t \in [\tau_n, \tau_{n+1}], n \geq 1.
\end{equation}
We consider the flow $\Phi : R_+ \times \R^d\mapsto\R^d$, which given an initial point $\vz (0) = \vz \in \R^d$ 
and time $t \in \R_+$ gives the position $z(t)= \vz \in \R^d$ at time $t$. 
We denote the first argument (time $t$) as a subscript of the flow, i.e., $\Phi_t$. 
As per \cite{benaim_hirsch_1995}, we also need the following definition. 

\vp1
\begin{Definition}\label{def:apt}
$\bar\vz(t)$ is an \textit{asymptotic pseudotrajectory (APT)} of a flow $\Phi$ defined by a continuous-time 
ODE if, for all $T>0$, we have:
\begin{equation}\tag{APT}\label{eq:apt}
        \lim_{t\rightarrow\infty} \sup_{0\leq h \leq T} || \bar \vz (t + h ) - \Phi_h(\vz(t)) || = 0 \,.
    \end{equation}
\end{Definition}

\good

$T$ can be seen as window size, and the $\sup$ seeks the largest difference of the expression within the window, 
and the condition should hold for \emph{any} window size $T$. 
Thus, Definition~\ref{def:apt} posits that in a \emph{very strong sense}, $\vz$ will eventually approach the 
solution $\Phi$ of the mean dynamics.

For convenience, we state Gr\"onwall's lemma that we will use shortly in the proof:

\vp2
\begin{lemma}\label{lmma:gl}
{\rm(Gr\"onwall)}\ \ If a positive function f(t) has the property
\vm1
    \begin{equation}\tag{GL}\label{eq:gronwall_lemma}
    0 \leq f(t) \leq k  + c \int_0^t f(s) ds, \quad 
\forall \ t \geq 0 \,,\quad \text{then}\quad f(t) \leq k e^{ct} \,.
\end{equation} 
\end{lemma}

The following Theorem determines if $\bar{\vz}(t) $  of OGDA is indeed an \ref{eq:apt} of the \emph{flow} 
$\Phi$ defined by the~\ref{eq:ogda_hrde}.

\vp2
\begin{theorem}\label{thm:hrde_apt_ogda}
{\rm(OGDA-$\gamma_n$ APT)}\ \ 
   Assuming $\mu$-strongly monotone variational inequality and that:
\vm6
\begin{align}\tag{A1}
\gamma_n^2 \rightarrow 0 \quad \text{and} \quad \sum_{n=1}^{\infty} \E [\gamma_n\mathcal{B}_n] < \infty \,, \\[-1mm]
       \tag{A2}
       \sum_{n=1}^\infty \E[\gamma_n^2(1+\mathcal{B}_n^2 + \sigma_n^2 )] < \infty \,,\\
       \beta(t) \dot \beta(t) < 2\mu \tag{A3} \,,
   \end{align}
where $\beta (t) = 2 / \gamma(t)$, then the interpolation $\bar\vz(t)$ of the~\ref{eq:ogda-gamma-n} iterates 
$\vz_n$, is an APT of \ref{eq:ogda_hrde} w.p.1.
\end{theorem}

\vm5
\proof
Notice from Eq.~\ref{eq:linear_interpolation} that $\bar{\vz}(\tau_n) = \vz_n$, and by unfolding \ref{eq:ogda-gamma-n} 
we get:
\begin{align*}
\bar{\vz}(\tau_n) &= \vz_{n-1} + \gamma_{n-1} \big(-2V(\vz_{n-1}) + V(\vz_{n-2}) + \mathcal{R}_{n}\big)\\
& = \vz_{n-2} + \gamma_{n-2} \big(-2 V(\vz_{n-2}) + V(\vz_{n-2}) + \mathcal{R}_{n-1}\big) \\
&\hskip5mm + \gamma_{n-1} \big(-2V(\vz_{n-1}) + V(\vz_{n-2}) + \mathcal{R}_{n}\big) \\
    &= \dots = \vz_0 - 2 \sum_{k=0}^{n-1} \gamma_k V(\vz_k)  
    + \underbrace{\sum_{k=0}^{n-1} \gamma_k V(\vz_{k-1})
    + \sum_{k=0}^{n-1} \gamma_k  \mathcal{R}_{k+1}}_{\zeta_{n}} \,.
\end{align*}

\vm5
For convenience, we introduce the indexing ``continuous-to-discrete'' correspondence $\upkappa(s)$ that, given 
$s\in \R$ and a sequence of points $\{\tau_k\}_0^\infty, \tau_k \in \R$, returns the index $k$ of the largest 
$\tau_k$  that satisfies $\tau_k \leq s$; i.e., $\upkappa(s) = \sup {k \geq 1: s \geq \tau_k}$.
Then by definition, we get:
\vm2
\begin{align}
\bar{\vz}(\tau_n) 
& = \vz_0 - 2\sum_{k=0}^{n-1} \int_{\tau_k}^{\tau_{k+1}} V(\vz_{\upkappa(s)}) ds  + \zeta_{n} 
     = \vz_0 - 2\int_0^t V(\vz_{\upkappa(s)}) ds + \zeta_{n} \nonumber\\
&= \vz_0 - 2\int_0^t V\big( \bar\vz(\tau_{\upkappa(s)}) \big) ds + \zeta_{n}
     = \vz_0 - 2\int_0^t V\big( \tilde\vz(s) \big) ds + \zeta_{n}\,, 
\label{eq:sum_to_int}
\end{align}
where in the second row, we used the fact that the sum of integration is the integral itself, the third equality 
follows by definition, and in the last row with $\tilde\vz(t)$ we denote the piecewise constant interpolation; 
see Hsieh et al. \cite{hsieh2020limits}.

\good

Applying the above at time $t+u$, denoting $\tau_m\triangleq t+u$ and the corresponding perturbation as 
$\zeta_m$ we have that:
\vm3
$$
\bar\vz(t+u) = \vz_0 - 2 \int_0^{t+u} V\big( \tilde \vz (s) \big)  + \zeta_m \,.
$$
For the expression in Definition~\eqref{def:apt} we have:
\begin{align*}
\mathcal{D}(u,t) 
& \triangleq \norm{ \bar\vz(t+u) - \Phi_u \big(\bar\vz(t)\big) }
     = \norm{ \bar\vz(t+u) - \bar\vz(t) + \bar\vz(t) - \Phi_u \big(\bar\vz(t)\big) } \\
& = \norm{ - 2\int_t^{t+u} V\big(\tilde\vz(s)\big) ds + \zeta_m - \zeta_n - \int_0^u \mathcal{H} \big[ \Phi_s\big(\tilde\vz(t)\big) \big] ds }\,,
\end{align*}
where with $\mathcal{H}$ we denote the solution of the mean~\ref{eq:ogda_hrde} dynamics.

\vm1
By substituting in $\mathcal{D}(u,t)$ and doing change of variables for $\int_t^{t{+}u}V\big(\tilde\vz(s)\big)ds$ we have:
\begin{align*}
\mathcal{D}(u,t) 
& = \left\| 
     \begin{bmatrix}
     {-} 2\int_0^u V\big(\tilde\vz(t+s)\big)ds {+} \zeta_m {-} \zeta_n \\ \mathbf{0}
     \end{bmatrix} \right. \\[-1mm]
&\hskip10mm - \int_0^u \left.
     \begin{bmatrix}
     -2V\Big(\Phi_s \big( \tilde\vz (t) \big)\Big)
     - \kappa(t) \big( \tilde\vz(t+s) + \vomega (t+s) \big)\\
     - \kappa(t) \big( \tilde\vz(t+s) + \vomega (t+s) \big)
     \end{bmatrix}\right\| ds  \\[1mm]
&\leq \norm{ 
   \begin{bmatrix}
   \zeta_m {-} \zeta_n \\ \mathbf{0}
   \end{bmatrix}
   - \int_0^u 
   \begin{bmatrix}
   - \kappa(t) \big( \tilde\vz(t+s) + \vomega (t+s) \big)\\
   - \kappa(t) \big( \tilde\vz(t+s) + \vomega (t+s) \big)
   \end{bmatrix} ds } \\
&\hskip10mm + \norm{ 2 \int_0^u 
   \begin{bmatrix}
   V\Big(\Phi_s \big( \tilde\vz (t) \big) \Big) {-} V\big(\tilde\vz(t+s)\big) {+}  \\ \mathbf{0}
   \end{bmatrix} ds } \\ 
&\leq \norm{ 
   \begin{bmatrix}
   \zeta_m {-} \zeta_n \\ \mathbf{0}
   \end{bmatrix}
   - \int_0^u 
   \begin{bmatrix}
   - \kappa(t) \big( \tilde\vz(t+s) + \vomega (t+s) \big)\\
   - \kappa(t) \big( \tilde\vz(t+s) + \vomega (t+s) \big)
   \end{bmatrix} ds } \\
&\hskip10mm + 2L\norm{ \int_0^u \Big(\Phi_s \big( \tilde\vz (t) \big)  {-} \tilde\vz(t+s)\Big) ds } \\
&\leq \norm{
   \begin{bmatrix}
   \zeta_m {-} \zeta_n \\ \mathbf{0}
   \end{bmatrix}
   - \int_0^u 
   \begin{bmatrix}
   \dot\vomega (t+s) \\ \dot \vomega (t+s)
   \end{bmatrix} ds }e^{2Lu},
\end{align*}
where in the second row, we used the triangle inequality; 
in the third, the assumption that the vector field is $L$-Lipschitz; and in the fourth row, we used Lemma~\ref{lmma:gl} and used $\vomega(t)=-\frac{1}{\kappa(t)} \dot\vomega (t) - \vz(t)$ from~\ref{eq:ogda_hrde_gamma-t_2}  and $\kappa(t)=\frac{1}{\gamma(t)}$.
Using A3 and Lemma~\ref{lm:conv_ogda_hrde_gamma-t_2} that the~\ref{eq:ogda_hrde_gamma-t_2} dynamics converges to a solution point for strongly monotone variational inequalities, we have that in the limit $t\rightarrow\infty$ the term $\int_0^u \dot\vomega(t+s)ds \rightarrow 0$ for strongly monotone variational inequalities.

The terms $\zeta_m {-} \zeta_n $ consist of the difference between the remaining bias terms and the random noise. 
For the bias terms of~\ref{eq:ogda-gamma-n} we get:
\begin{align*}
\zeta_m {-} \zeta_n 
&= \sum_{k=0}^{m-1} \gamma_k V(\vz_{k-1}) - \sum_{k=0}^{n-1} \gamma_k V(\vz_{k-1}) +  \sum_{k=0}^{m-1} \gamma_k \mathcal{R}_{k+1} - \sum_{k=0}^{n-1} \gamma_k \mathcal{R}_{k+1} \\
&= \sum_{k=n}^{m-1} \gamma_k V(\vz_{k-1}) + \underbrace{\sum_{k=0}^{n-1} \gamma_k \mathcal{R}_{k+1}}_{\mathcal{R}} \,.
\end{align*}

\good

As in~\ref{eq:sum_to_int}, by replacing $\tilde\vz(t)$ and using Taylor expansion we have:
\begin{align*}
\zeta_m {-} \zeta_n 
&= \sum_{k=n}^{m-1} \Big( 
\gamma_k  V(\tilde\vz(\tau_k)) + \mathcal{O}(\gamma_{k}^2)  \Big)
+ \mathcal{R}.
\end{align*}
Thus, setting $t\rightarrow \infty$ and taking $\gamma^2\rightarrow0$ gives:
\begin{align*}
\zeta_m {-} \zeta_n 
&= \int_0^u V(\vz) + \mathcal{R} \,.
\end{align*}

We have $\mathcal{R}\rightarrow0$ as $t\rightarrow\infty$ \cite{hsieh2020limits}, and for strongly monotone 
variational inequalities, we have $V(\vz)\rightarrow 0 $ as $t\rightarrow\infty$, and thus the proof follows.
\endproof

Note that we used the strong MVI assumption solely for Theorem~\ref{thm:hrde_apt_ogda}; proving the analogous result for a more general setup is an interesting direction for future research.

\section{Convergence analysis for bilinear games}\label{app:bilinear}

In this section, we provide the proofs of the convergence results on the bilinear game (\ref{eq:bg}) listed in 
\S\ref{sec:hrde_bilinear}.
As a reminder, we focus on the following problem with full rank $\mA \in \R^{d_1} \times \R^{d_2} $: 
\begin{equation} \tag{BG} 
    \min_{\vx \in \R^{d_1}} \max_{\vy \in \R^{d_2}} \vx^\intercal \mA \vy \,.
\end{equation}
The joint vector field of~\ref{eq:bg} and its Jacobian are:
\begin{equation} \tag{BG:JVF}\label{eq:bg_jac}
V_{\text{BG}}(\vz) = 
\begin{bmatrix}
 \mA \vy \\
-\mA^\intercal \vx \\
\end{bmatrix} 
\end{equation}

\begin{equation} \tag{BG:Jac-JVF} \label{eq:bg_jac_jvf}
J_{\text{BG}}(\vz) = 
\begin{bmatrix}
0 &  \mA \\
-\mA^\intercal & 0
\end{bmatrix}
\end{equation}
By substituting these definitions in the derived  HRDEs one can analyze the convergence of the corresponding 
method on the \ref{eq:bg} problem. 
Moreover, as the obtained system is linear this can be done using standard tools from dynamical systems, without 
the need for Lyapunov functions, as described in~\S\ref{sec:hrde_bilinear}. 
Using the Routh-Hurwitz criterion, the theorems below state the convergence results for the studied methods, where the different matrices $\mC$ for each  method are denoted with subscripts. 

\vp1
\begin{theorem}\label{thm:bilinear_gda_convergence}
{\rm(Divergence of~\ref{eq:gda-hrde} on~\ref{eq:bg})}\ \ 
For any step size $\gamma > 0$, there exist an eigenvalue of $\mC_{\text{GDA}}$  whose real part is non-negative, 
and there exists ${\lambda_i} \in Sp(\mC_{\text{GDA}})$, subject to $\mathfrak{R}(\lambda_i ) \geq 0$.
Thus, the $\mathcal{O}(\gamma)$-HRDE of the Gradient Descent Ascent method ~\eqref{eq:gda-hrde} diverges on 
the \ref{eq:bg} problem for any choice of nonzero step size (and any $\mA$).
\end{theorem}

\begin{remark}
(On the implications for the discrete counterpart)\ \ 
The larger the step size, the faster the discrete \ref{eq:gda} method diverges on this problem 
[see \S 4 in \cite{chavdarova2021lamm}]. 
Thus, combined with the above theorem, we have that the discrete GDA method also diverges for any step size $\gamma>0$.
\end{remark}

\begin{remark}
(On the implications for higher resolution than $\mathcal{O}(\gamma)$)\ \ 
From the derivation of~\ref{eq:gda-hrde} (in \S~\ref{sec:hrde_sp}) we can observe that by increasing the 
resolution the right hand side, $-V(z(n\delta))$ therein, does not change.  
It is easy to see from the proof that the HRDEs of higher resolution than the $\mathcal{O}(\gamma)$-HRDE 
for GDA will also be divergent.
\end{remark}

\begin{theorem}\label{thm:bilinear_eg_convergence}
{\rm(Convergence of \ref{eq:eg_hrde} on \ref{eq:bg})}\ \ 
For any $\gamma$, the real part of the eigenvalues of $\mC_{\text{EG}}$ is always negative,
$\mathfrak{R}(\lambda_i ) < 0, \forall{\lambda_i} \in Sp(\mC_{\text{EG}})$,
thus the $\mathcal{O}(\gamma)$-HRDE of the Extragradient method \eqref{eq:eg_hrde} dynamics converges on the~\ref{eq:bg} problem for any step size. 
\end{theorem}

\vp1
\begin{theorem}\label{thm:bilinear_ogda_convergence}
{\rm(Convergence of \ref{eq:ogda_hrde} on~\ref{eq:bg})}\ \ 
For any $\gamma$, the real part of the eigenvalues of $\mC_{\text{OGDA}}$ is always negative,
$\mathfrak{R}(\lambda_i ) < 0, \forall{\lambda_i} \in Sp(\mC_{\text{OGDA}}) $,
thus the $\mathcal{O}(\gamma)$-HRDE of the Optimistic Gradient Ascent Descent method~\eqref{eq:ogda_hrde} converges on the~\ref{eq:bg} problem for any step size. 
\end{theorem}

\vp1
\begin{theorem}\label{thm:bilinear_la2-gda_convergence}
{\rm(Convergence of \ref{eq:la2-gda_hrde} on~\ref{eq:bg})}\ \ 
For any $\gamma$, the real part of the eigenvalues of $\mC_{\text{LA2-GDA}}$ is always negative,
$\mathfrak{R}(\lambda_i ) < 0, \forall{\lambda_i} \in Sp(\mC_{\text{LA2-GDA}})$,
thus the $\mathcal{O}(\gamma)$-HRDE of the LA2-GDA method \eqref{eq:la2-gda_hrde} converges on the 
\ref{eq:bg} problem for any step size. 
\end{theorem}

\vp1
\begin{theorem}\label{thm:bilinear_la3-gda_convergence}
{\rm(Convergence of \ref{eq:la3-gda_hrde} on~\ref{eq:bg})}\ \ 
For any $\gamma$, the real part of the eigenvalues of $\mC_{\text{LA3-GDA}}$ is always negative,
$\mathfrak{R}(\lambda_i ) < 0, \forall{\lambda_i} \in Sp(\mC_{\text{LA3-GDA}})$,
thus the $\mathcal{O}(\gamma)$-HRDE of the LA3-GDA method \eqref{eq:la3-gda_hrde}  converges on the ~\ref{eq:bg} problem for any step size. 
\end{theorem}

We observe that the convergence of their associated HRDEs is in line with their observed performances on 
this problem, see Fig.~\ref{fig:bg}. 
In the following subsections, we prove these theorems separately.

\subsection{Proof of Theorem \ref{thm:bilinear_gda_convergence}: 
Divergence of \ref{eq:gda-hrde} on \ref{eq:bg}}\label{app:bilinear_gda_convergence}

Recall that for the \ref{eq:gda} optimizer we have the following (Eq.~\ref{eq:gda-hrde}):
\begin{equation}\notag
\begin{split}
\dot{\vz}(t) & = \vomega(t) \\[1mm]
\dvom(t) & = - \beta \cdot \vomega(t) - \beta \cdot V(\vz(t)).
\end{split}
\end{equation}
By denoting $\dvx(t) = \vomega_x(t) $, $\dvy(t) = \vomega_y(t) $, and $b=\frac{2}{\gamma}$, for \ref{eq:gda} 
we have the following:
\begin{align*}
  \begin{bmatrix}
    \dvx(t) \\
    \dvy(t) \\
     \dvom_x(t) \\
     \dvom_y(t)
  \end{bmatrix} = 
  \underbrace{\begin{bmatrix}
    0 & 0 & \mI & 0 \\
    0 & 0 & 0 & \mI \\
    0 & -  \beta \mA & - \beta \mI & 0 \\
    \beta \mA^\intercal & 0  & 0       & - \beta \mI
  \end{bmatrix}}_{ \triangleq \mC_{\text{GDA}}} \cdot
  \begin{bmatrix}
    \vx(t) \\
    \vy(t) \\
    \vomega_x(t) \\
    \vomega_y(t)
  \end{bmatrix}.
\end{align*}

\proof[\textbf{Proof of Theorem~\ref{thm:bilinear_gda_convergence}}]
To obtain the eigenvalues $\lambda \in \CC$ of $\mC_{\text{GDA}}$ we have:
\begin{align*}
& \det(\mC_{\text{GDA}} - \lambda \mI) 
   = \det\left(
   \begin{bmatrix}
    - \lambda\mI    & 0     & \mI   & 0 \\
    0   & - \lambda\mI      & 0     & \mI \\
    0   & - \beta \mA   & -(\beta+\lambda)\mI & 0 \\
    \beta \mA^\intercal & 0     & 0     & -(\beta+\lambda)\mI
   \end{bmatrix} \right) \\[1mm]
& = \det \left( \lambda(\beta+\lambda) \mI + \beta  
   \begin{bmatrix} 0 & \mA \\ -\mA^\intercal & 0 \end{bmatrix} \right) 
   = \det \left( 
   \begin{bmatrix} \lambda(\beta+\lambda)\mI & \beta\mA \\ -\beta\mA^\intercal  & \lambda(\beta+\lambda)\mI  \end{bmatrix} 
   \right) \\[1mm]
& =  \det \left( 
   \begin{bmatrix} \lambda^2(\beta+\lambda)^2\mI + \beta^2\mA\mA^\intercal \end{bmatrix} \right) 
   = \beta^{2n} \det \left( 
   \begin{bmatrix} \frac{\lambda^2}{\beta^2}(\beta+\lambda)^2\mI + \mA\mA^\intercal \end{bmatrix} \right) \,,
\end{align*}

\vm3
where we successively used \ $\det\left( \begin{bmatrix} \mB_1 & \mB_2 \\ \mB_3 & \mB_4 \end{bmatrix} \right) = \det(\mB_1\mB_4-\mB_2\mB_3)$.

Let $\kappa$ denote an eigenvalue of $-\mA\mA^\intercal$. 
Since $-\mA\mA^\intercal$ is symmetric and negative definite and $\mA$ is full rank, we have that $\kappa \in \R$, $\kappa<0$. 
From the last expression we observe that $\kappa = \frac{\lambda^2}{\beta^2}(\beta+\lambda)^2$ is the eigenvalue of $-\mA\mA^\intercal$.
Thus, we have the following polynomial with real coefficients:
\begin{align*}
\lambda^2(\beta+\lambda)^2 = \kappa \beta^2\,, \quad \text{or}\\[1mm]
\lambda^4 +2\beta\lambda^3+\beta^2\lambda^2+0\lambda-\kappa \beta^2 = 0 \,.
\end{align*}

\vp{10}

The Routh array is then:

\vm{18}
\begin{table}[h]
\hskip60mm
    \begin{tabular}{ccc}
\cellcolor[gray]{0.9}         $1$& $\beta^2$ &$-\kappa \beta^2$  \\
\cellcolor[gray]{0.9}         $2\beta$& $0$ & $0$  \\ 
\cellcolor[gray]{0.9}         $\beta^2$ &$-\kappa \beta^2$ &$0$\\
\cellcolor[gray]{0.9}         $2 \kappa \beta$ & $0$ & $0$ \\ 
\cellcolor[gray]{0.9}         $-\kappa \beta^2$ & $0$ & $0$,
    \end{tabular}
\end{table}

where we recall that $\beta=\frac{2}{\gamma} >0$ and $\kappa<0$. 
We observe  that  the first column of the Routh array for GDA has a change of signs, indicating that 
there is an eigenvalue $\lambda_i \in \text{Spec}\{\mC_{\text{GDA}}\}$ whose real part is positive  
$\mathfrak{R}\{\lambda_i\} > 0$. 
Due to the Routh-Hurwitz criterion, the system is unstable.

Thus, for any choice of step size $\gamma$ the \ref{eq:gda} dynamics \emph{diverges} on the \ref{eq:bg} problem.
\endproof

\vm6
\subsection{Proof of Theorem~\ref{thm:bilinear_eg_convergence}: Convergence of \ref{eq:extragradient} 
on \ref{eq:bg}}\label{app:bilinear_eg_convergence}

Recall that for the \ref{eq:extragradient} optimizer we have the following (Eq.~\ref{eq:eg_hrde}):
\begin{equation}\notag
\begin{split}
\dot{\vz}(t) & = \vomega(t) \\
\dvom(t) & = - \beta \cdot \vomega(t) - \beta \cdot V(\vz(t)) + 2 \cdot J(\vz(t)) \cdot V(\vz(t)) \,,
\ \text{where $\beta=\sfrac{2}{\gamma}$}.
\end{split}
\end{equation}
By denoting $\dvx(t) = \vomega_x(t) $ and $\dvy(t) = \vomega_y(t) $ for \ref{eq:extragradient} we have the following:
\begin{align*}
  \begin{bmatrix}
    \dvx(t) \\
    \dvy(t) \\
     \dvom_x(t) \\
     \dvom_y(t)
  \end{bmatrix} = 
  \underbrace{\begin{bmatrix}
    0 & 0 & \mI & 0 \\
    0 & 0 & 0 & \mI \\
    - 2 \mA \mA^\intercal & - \beta \mA & - \beta \mI & 0 \\
    \beta \mA^\intercal & - 2 \mA^\intercal \mA & 0       & - \beta \mI
  \end{bmatrix}}_{ \triangleq \mC_{\text{EG}}} \cdot
  \begin{bmatrix}
    \vx(t) \\
    \vy(t) \\
    \vomega_x(t) \\
    \vomega_y(t)
  \end{bmatrix}.
\end{align*}

\good

\proof[\textbf{Proof of Theorem~\ref{thm:bilinear_eg_convergence}}]
To obtain the eigenvalues $\lambda \in \CC$ of $\mC_{\text{EG}}$ we have:
\begin{align*}
\det(\mC_{\text{EG}} - \lambda \mI) &=
\det\left(
\begin{bmatrix}
    - \lambda\mI & 0 & \mI & 0 \\
    0 & - \lambda\mI & 0 & \mI \\
    - 2 \mA \mA^\intercal & - \beta \mA & - (\beta + \lambda) \mI & 0 \\
    \beta \mA^\intercal & - 2 \mA^\intercal \mA & 0       & - (\beta + \lambda) \mI
\end{bmatrix}
\right) \\ &= \det \left( \lambda(\beta + \lambda) \mI - 
\underbrace{\begin{bmatrix}
- 2 \mA \mA^\intercal & - \beta \mA\\
\beta \mA^\intercal & - 2 \mA^\intercal \mA\\
\end{bmatrix}}_{\triangleq \mD} 
\right)   \,,
\end{align*}
where we used $ \det\left( \begin{bmatrix} \mB_1 & \mB_2 \\ \mB_3 & \mB_4 \end{bmatrix} \right) = \det(\mB_1\mB_4-\mB_2\mB_3)$.

Let $\mu = \mu_1 + \mu_2 i \in \CC$ denote the eigenvalues of $\mD$. 
We have: $\lambda (\beta+\lambda) - \mu = 0$. 
Using the generalized Hurwitz theorem for polynomials with complex coefficients \cite{xie1985}, we obtain 
the following generalized Hurwitz array:
\begin{table}[h]
    \centering
    \begin{tabular}{lccc}
    $\lambda^2$& \cellcolor[gray]{0.9}    $1$        & $0$       & $\mu_1$  \\
    $\lambda^1$& \cellcolor[gray]{0.9}  $\beta$        & $\mu_2$   & $0$  \\ 
               &    $-\mu_2$   & $\beta\mu_2$  & $0$  \\
    $\lambda^0$& \cellcolor[gray]{0.9} $-\mu_2^2-\beta^2\mu_1$ & $0$ & $0$,
    \end{tabular}
\end{table}
where the terms whose change of sign determines the stability of the polynomial are highlighted. 
As $\beta>0$, it follows that the system is stable iff $\mu_1 < - \frac{1}{\beta^2}\mu_2^2$.

Thus, it suffices to show that:
\vm1
\begin{equation}\label{eq:prf-eg-bg_ineq}
\mathfrak{R}(\mu(\vz)) < - \sfrac{1}{\beta^2} \big(\mathfrak{I}(\mu(\vz))\big)^2   \,.
\end{equation}

\vm5
We have that:
\vm2
\begin{align*}
\mu(\vz) &= \Bar{\vz}^\intercal\mD\vz =
\begin{bmatrix}
    \Bar{\vx}^\intercal && \Bar{\vy}^\intercal
\end{bmatrix}
   \begin{bmatrix}
   - 2 \mA \mA^\intercal & - \beta \mA\\ \beta \mA^\intercal & - 2 \mA^\intercal \mA\\
   \end{bmatrix}
   \begin{bmatrix}
    \vx \\ \vy
   \end{bmatrix} \\[1mm]
&= -2 ||\mA^\intercal\vx||^2_2 - 2||\mA\vy||^2_2  + \beta \big(\Bar{\vy}^\intercal \mA^\intercal\vx  
     -  \Bar{\vx}^\intercal \mA\vy \big) \\[3mm]
&= \underbrace{-2 ||\mA^\intercal\vx||^2_2 - 2||\mA\vy||^2_2}_{\mathfrak{R}\big(\mu(\vz)\big)}  + \underbrace{2\beta \cdot \mathfrak{I} (\Bar{\vx}^\intercal\mA\vy)}_{\mathfrak{I}\big(\mu(\vz)\big)} \cdot i \,,
\end{align*}

\vm3
where the last equality follows from the fact that $\Bar{\vx}^\intercal \mA\vy$ is a complex conjugate of $\Bar{\vy}^\intercal \mA^\intercal\vx$, thus $\Bar{\vy}^\intercal \mA^\intercal\vx - \Bar{\vx}^\intercal \mA\vy = 2\mathfrak{I}\big( \Bar{\vx}^\intercal\mA\vy \big)\cdot i$. 
By replacing this in Eq.~\ref{eq:prf-eg-bg_ineq}, it follows that we need to show:
\vm1
$$
-2 \big( ||\mA^\intercal\vx||^2_2 + ||\mA\vy||^2_2 \big) \leq - 4 \mathfrak{I}^2 (\Bar{\vx}^\intercal \mA \vy ) \,,  
$$
$$
\links{or:}{41.5}
2 \big( ||\mA^\intercal\vx||^2_2 + ||\mA\vy||^2_2 \big) \geq  4 \mathfrak{I}^2 (\Bar{\vx}^\intercal \mA \vy ) \,.
$$
Thus, it suffices to show that:
\begin{equation} \label{eq:eg_bg_final_inequality}
    ||\mA^\intercal\vx||^2_2 + ||\mA\vy||^2_2 \geq  2 | \Bar{\vx}^\intercal \mA \vy |^2 \,.
\end{equation}

\good

Consider the case $||\vx||_2 \leq ||\vy||_2 \leq 1 $. 
We have two sub-cases.

{\bf Case 1:}\ \ $||\mA^\intercal\vx||_2^2 \leq ||\mA\vy||_2^2 $. 
We can set $ ||\mA \vy||_2 = ||\vy||_2 $, and we have:
\begin{align}
||\mA^\intercal\vx||_2^2 + ||\mA^\intercal\vy||_2^2 
& = ||\mA^\intercal\vx||_2^2 +  ||\vy||_2^2 \geq \sfrac{1}{2} (||\mA^\intercal\vx||_2^2 + ||\vy||_2^2)^2 \nonumber\\
& \geq 2 ||\mA^\intercal\vx||_2^2||\vy||_2^2 \geq 2 |\bar{\vx}^\intercal \mA \vy |_2^2 \,,
\end{align}
where the last inequality follows from the Cauchy--Schwarz inequality.

{\bf Case 2:}\ \ $||\mA\vy||_2^2 \leq ||\mA^\intercal\vx||_2^2 $.
We can set $ ||\mA^\intercal \vx||_2 = ||\vx||_2 $, and we have:
\begin{align}
||\mA^\intercal\vx||_2^2 + ||\mA^\intercal\vy||_2^2 
& = ||\vx||_2^2 + ||\mA\vy||_2^2  \geq \sfrac{1}{2} (||\vx||_2^2 + ||\mA\vy||_2^2)^2 \nonumber\\
& \geq 2 ||\vx||_2^2||\mA\vy||_2^2 \geq 2 |\bar{\vx}^\intercal \mA \vy |_2^2 \,,
\end{align}
where the last inequality follows from the Cauchy--Schwarz inequality.

The case $||\vy||_2 \leq ||\vx||_2 \leq 1 $ can be shown analogously.
\endproof

\subsection{Proof of Theorem \ref{thm:bilinear_ogda_convergence}: 
Convergence of \ref{eq:ogda} on \ref{eq:bg}}\label{app:bilinear_ogda_convergence}

Recall that for the \ref{eq:ogda} optimizer we have the following (Eq.~\ref{eq:ogda_hrde}):
\begin{equation}\notag
\begin{split}
\dot{\vz}(t) & = \vomega(t) \\
\dvom(t) &= - \beta \cdot \vomega(t) - \beta \cdot V(\vz(t)) - 2 \cdot J(\vz(t)) \cdot \vomega(t)\,,
\ \ \text{where  $\beta=\sfrac{2}{\gamma}$.}
\end{split}
\end{equation}
By denoting $\dvx(t) = \vomega_x(t) $ and $\dvy(t) = \vomega_y(t)$, for \ref{eq:ogda} we have the following:
\begin{align}\label{eq:c_ogda_bg}
  \begin{bmatrix}
    \dvx(t) \\
    \dvy(t) \\
     \dvom_x(t) \\
     \dvom_y(t)
  \end{bmatrix} = 
  \underbrace{\begin{bmatrix}
    0 & 0 & \mI & 0 \\
    0 & 0 & 0 & \mI \\
    0 & - \beta \mA & - \beta \mI & -2\mA \\
    \beta \mA^\intercal & 0  & 2\mA^\intercal  & - \beta \mI
  \end{bmatrix}}_{ \triangleq \mC_{\text{OGDA}}}
  \begin{bmatrix}
    \vx(t) \\
    \vy(t) \\
    \vomega_x(t) \\
    \vomega_y(t)
  \end{bmatrix}.
\end{align}

\proof[\textbf{Proof of Theorem~\ref{thm:bilinear_ogda_convergence}}]
To obtain the eigenvalues $\lambda \in \CC$ of $\mC_{\text{OGDA}}$  we have:  
\begin{align*}
\det(\mC_{\text{OGDA}} - \lambda \mI) &=
\det\left(
\begin{bmatrix}
    - \lambda\mI & 0 & \mI & 0 \\
    0 & - \lambda\mI & 0 & \mI \\
    0 & - \beta \mA & - (\beta+\lambda) \mI & -2\mA  \\
    \beta \mA^\intercal & 0 & 2\mA^\intercal  & - (\beta+\lambda) \mI
\end{bmatrix}\right) \\[2mm]
& = \det \left( 
\begin{bmatrix}
\lambda(\beta+\lambda)\mI & 2\lambda\mA \\
-2\lambda\mA ^\intercal & \lambda(\beta+\lambda)\mI\\
\end{bmatrix}
- 
\begin{bmatrix}
0 & - \beta \mA\\
\beta \mA^\intercal & 0\\
\end{bmatrix} \right)   \\[2mm]
&= \det \left( 
   \begin{bmatrix}
   \lambda(\beta+\lambda)\mI & (2\lambda+\beta)\mA \\
   -(2\lambda+\beta)\mA ^\intercal & \lambda(\beta+\lambda)\mI\\
   \end{bmatrix}\right)\,,
\end{align*} 
where we used \ $\det \left( \begin{bmatrix} \mB_1 & \mB_2 \\ \mB_3 & \mB_4 \end{bmatrix} \right) 
   = \det(\mB_1\mB_4-\mB_2\mB_3)$.

\paragraph{Square $\mA$.} For simplicity, we will first consider the case when $\mA$ is square matrix. 
For the eigenvalues of $\mC_{\text{OGDA}}$ we have:
\begin{align*}
\det(\mC_{\text{OGDA}} - \lambda \mI) &= \det \big(  
\begin{bmatrix}
\lambda^2 (\beta+\lambda)^2\mI + (2\lambda+\beta)^2\mA\mA^\intercal
\end{bmatrix}
\big)\\ &= 
(2\lambda + \beta)^{2n} \det \big(
\begin{bmatrix}
    \frac{\lambda^2 (\beta+\lambda)^2}{(2\lambda+\beta)^2}\mI + \mA\mA^\intercal
\end{bmatrix}
\big)\,,
\end{align*}
where we used \ $\det\left( \begin{bmatrix} \mB_1 & \mB_2 \\ \mB_3 & \mB_4 \end{bmatrix} \right) = \det(\mB_1\mB_4-\mB_2\mB_3)$.

Let $\kappa$ denote an eigenvalue of $-\mA\mA^\intercal$, thus $\kappa\in\R, \kappa<0$.
From the last expression we observe that $\kappa = \frac{\lambda^2 (\beta+\lambda)^2}{(2\lambda+\beta)^2}$ is an eigenvalue of $-\mA\mA^\intercal$.
Thus, we have the following polynomial with real coefficients:
\begin{align*}
\lambda^2(\beta+\lambda)^2 = \kappa (2\lambda+\beta)^2\,, \\[1mm]
\links{or}{34.5}
\lambda^4 + 2\beta\lambda^3 + (\beta^2-4\kappa)\lambda^2 - 4\beta\kappa\lambda -\kappa \beta^2 = 0 \,.
\end{align*}
The Routh array is then:
\begin{table}[h]
    \centering
    \begin{tabular}{>{\columncolor[gray]{.9}}ccc}
         $1$            & $\beta^2-4\kappa$     &$-\kappa \beta^2$  \vst{13}{0}\\
         $2\beta$           & $-4\beta\kappa$       & $0$  \vst{13}{0}\\ 
         $\beta^2-2\kappa$  & $-\kappa \beta^2$ &$0$ \vst{14}{0}\\
         $\frac{(-2\beta\kappa)(3\beta^2-4\kappa)}{(\beta^2-2\kappa)} > 0$      & $0$ & $0$ \vst{14}{9}\\ 
         $\frac{(-2\beta\kappa)(3\beta^2-4\kappa)(-\kappa \beta^2)}{(\beta^2-2\kappa)} > 0$ & $0$ & $0$,
    \end{tabular}
\end{table}
where we recall that $\beta=\frac{2}{\gamma} >0$ and $\kappa<0$. 
We observe  that  the first column of the Routh array  has only positive elements and no change of signs, implying 
that all eigenvalues  $\mathfrak{R}\{\lambda_i\}<0, \enspace \forall{\lambda_i}  \in Sp\{\mC_{\text{OGDA}}\}$.  
Thus, by the Routh-Hurwitz criterion, the system is stable.

\vp2
{\bf General $\mA$.} When $\mA$ is not necessarily square, using the  fact that $\lambda(\beta+\lambda) \mI$ is 
invertible -- since $\mC_{\text{OGDA}}$ is full rank -- and using 
$$
\det\left( \begin{bmatrix} \mB_1 & \mB_2 \\ \mB_3 & \mB_4 \end{bmatrix} \right) = \det(\mB_4)\det(\mB_1-\mB_2\mB_4^{-1}\mB_3),
$$
for the eigenvalues of $\mC_{\text{OGDA}}$ we have:
\begin{align*}
\det(\mC_{\text{OGDA}} - \lambda \mI) &= 
\det\big( \lambda(\beta+\lambda)\mI \big)
\det\big(  \lambda(\beta+\lambda)\mI + (2\lambda+\beta)^2(\lambda(\beta+\lambda))^{-1} \mA\mA^\intercal  \big) \\
&= 
\underbrace{\det\big( \lambda(\beta+\lambda)\mI \big)}_{(1)}
\underbrace{\det\big( \frac{\lambda^2(\beta+\lambda)^2}{(2\lambda+\beta)^2} \mI + \mA\mA^\intercal  \big)}_{(2)}.
\end{align*}

\vm4
As $(1) \neq 0$, it has to hold that $(2) = 0$.
Let $\kappa \in Sp\{ -\mA^\intercal\mA \}$ (thus $\kappa\in\R$  and $\kappa<0$), and thus we get the same 
polynomial with real coefficients, as for the case when $\mA$ is square:
\begin{align*}
\lambda^2(\beta+\lambda)^2 = \kappa (2\lambda+\beta)^2\,, \quad \text{or}\\[1mm]
\lambda^4 + 2\beta\lambda^3 + (\beta^2-4\kappa)\lambda^2 - 4\beta\kappa\lambda -\kappa \beta^2 = 0 \,.
\end{align*}
Hence, we get the same Routh array as above, and the same proof follows. 

Thus, for any $\gamma$ the~\ref{eq:ogda_hrde} dynamics is \emph{stable} on the~\ref{eq:bg} problem.
\endproof

\subsection{Proof of Theorem~\ref{thm:bilinear_la2-gda_convergence}: 
Convergence of LA2-GDA on \ref{eq:bg}}\label{app:bilinear_la2-gda_convergence}

Recall that for the \ref{eq:lookahead_mm}2-GDA optimizer we have the following (see Equation \ref{eq:la2-gda_hrde}):
\begin{equation}\notag
\begin{split}
\dot{\vz}(t) & = \vomega(t) \\
\dvom(t) & = - \beta \vomega(t)  - 2 \alpha \beta \cdot V(\vz(t)) + 2 \alpha J(\vz(t)) V(\vz(t))\,,
\ \ \text{where $\beta=\sfrac{2}{\gamma}$.}
\end{split} 
\end{equation}
By denoting $\dvx(t) = \vomega_x(t) $ and $\dvy(t) = \vomega_y(t)$ for \ref{eq:lookahead_mm}2-GDA we have the following:
\begin{align*}
  \begin{bmatrix}
    \dvx(t) \\
    \dvy(t) \\
     \dvom_x(t) \\
     \dvom_y(t)
  \end{bmatrix} = 
  \underbrace{\begin{bmatrix}
    0 & 0 & \mI & 0 \\
    0 & 0 & 0 & \mI \\
    -2\alpha \mA \mA^\intercal             & - \frac{4\alpha}{\gamma} \mA & - \frac{2}{\gamma} \mI & 0                        \\
    \frac{4\alpha}{\gamma} \mA^\intercal   & -2\alpha \mA^\intercal \mA           & 0                      & - \frac{2}{\gamma} \mI
  \end{bmatrix}}_{ \triangleq \mC_{\text{LA2-GDA}}} 
  \begin{bmatrix}
    \vx(t) \\
    \vy(t) \\
    \vomega_x(t) \\
    \vomega_y(t)
  \end{bmatrix}.
\end{align*}

\proof[\textbf{Proof of Theorem~\ref{thm:bilinear_la2-gda_convergence}}]
As in \S\ref{app:bilinear_eg_convergence}, to obtain the eigenvalues $\lambda \in \CC$ of $\mC_{\text{LA2-GDA}}$ we have:
\vm3
\begin{align*}
\det(\mC_{\text{LA2-GDA}} - \lambda \mI) &=
\det\left(
\begin{bmatrix}
    - \lambda\mI & 0 & \mI & 0 \\
    0 & - \lambda\mI & 0 & \mI \\
    - 2 \alpha \mA \mA^\intercal           & - \frac{4\alpha}{\gamma} \mA & - (\frac{2}{\gamma}+\lambda) \mI     & 0 \\
    \frac{4\alpha}{\gamma} \mA^\intercal   & - 2 \alpha \mA^\intercal \mA         & 0                     & - (\frac{2}{\gamma}+\lambda) \mI
\end{bmatrix}
\right) \\
&= \det \left( \lambda(\frac{2}{\gamma}+\lambda) \mI - 
\underbrace{\begin{bmatrix}
- 2 \alpha \mA \mA^\intercal & - \frac{4\alpha}{\gamma} \mA\\
\frac{4\alpha}{\gamma} \mA^\intercal & - 2 \alpha \mA^\intercal \mA\\
\end{bmatrix}}_{\triangleq \mD} 
\right)   \,.
\end{align*}
We observe that  only the lower left  block matrix $\mD$ differs from the proof for  EG in~\S\ref{app:bilinear_eg_convergence},  thus the inequality Eq.~\ref{eq:prf-eg-bg_ineq} needs to be proven for this optimizer too.

However, due to the different lower left block matrices for LA2-GDA we have: 
\begin{align*}
\mu(\vz) &= \Bar{\vz}^\intercal\mD\vz =
\begin{bmatrix}
    \Bar{\vx}^\intercal && \Bar{\vy}^\intercal
\end{bmatrix}
\begin{bmatrix}
- 2 \alpha \mA \mA^\intercal           & - \frac{4\alpha}{\gamma} \mA\\
\frac{4\alpha}{\gamma}  \mA^\intercal  & - 2 \alpha \mA^\intercal \mA\\
\end{bmatrix}
\begin{bmatrix}
    \vx \\
    \vy
\end{bmatrix} \\
&= -2\alpha ||\mA^\intercal\vx||^2_2 - 2\alpha ||\mA\vy||^2_2  + \frac{4\alpha}{\gamma} \big(\Bar{\vy}^\intercal \mA^\intercal\vx  -  \Bar{\vx}^\intercal \mA\vy \big) \\
&= \underbrace{-2\alpha \big( ||\mA^\intercal\vx||^2_2 - ||\mA\vy||^2_2\big)}_{\mathfrak{R}\big(\mu(\vz)\big)}  + \underbrace{\frac{4\alpha}{\gamma} \cdot \mathfrak{I} (\Bar{\vx}^\intercal\mA\vy)}_{\mathfrak{I}\big(\mu(\vz)\big)} \cdot i \,,
\end{align*}

\vm4
where the last equality follows from the fact that $\Bar{\vx}^\intercal \mA\vy$ is a complex conjugate of 
$\Bar{\vy}^\intercal \mA^\intercal\vx$, thus $\Bar{\vy}^\intercal \mA^\intercal\vx - \Bar{\vx}^\intercal \mA\vy = 2\mathfrak{I}\big( \Bar{\vx}^\intercal\mA\vy \big)\cdot i$. 
By replacing this in Eq.~\ref{eq:prf-eg-bg_ineq}, we need to show that:
\vm2
\begin{gather*}
-2 \alpha \big( ||\mA^\intercal\vx||^2_2 + ||\mA\vy||^2_2 \big) 
     \leq  - \sfrac{\gamma^2}{4}  \sfrac{4^{2}\alpha^2}{\gamma^2}  
     \mathfrak{I}^2 (\Bar{\vx}^\intercal \mA \vy ) \,, \\[1mm]
\links{or:}{38}
2\alpha \big( ||\mA^\intercal\vx||^2_2 + ||\mA\vy||^2_2 \big) 
     \geq  4 \alpha^2 \mathfrak{I}^2 (\Bar{\vx}^\intercal \mA \vy ) \,.
\end{gather*}

\good

Thus, it suffices to show that (assume $\alpha>0$):
$$
||\mA^\intercal\vx||^2_2 + ||\mA\vy||^2_2 \geq  2 
     \underbrace{\alpha}_{\parbox{30mm}{\footnotesize only difference with EG, and $\alpha \in (0,1)$}}
     | \Bar{\vx}^\intercal \mA \vy |^2 \,.
$$
As $\alpha \in (0,1)$ and the inequality~\eqref{eq:eg_bg_final_inequality} is tighter, the same proof for EG (\S\ref{app:bilinear_eg_convergence}) follows.
\endproof

\subsection{Proof of Theorem~\ref{thm:bilinear_la3-gda_convergence}: 
Convergence of LA3-GDA on \ref{eq:bg}}\label{app:bilinear_la3-gda_convergence}

Recall that for the \ref{eq:lookahead_mm}3-GDA optimizer we have the following HRDE (Eq.~\ref{eq:la3-gda_hrde}):
\vm1
\begin{equation}\notag
\begin{split}
\hskip8mm  \dot{\vz}(t) & = \vomega(t) \\
  \dvom(t)     & = - \frac{2}{\gamma} \vomega(t) 
  - \sfrac{6\alpha}{\gamma} V(\vz(t)) 
  + 6 \alpha \cdot J(\vz(t)) \cdot V(\vz(t)) \,.
\end{split} 
\end{equation}
By denoting $\dvx(t) = \vomega_x(t) $ and $\dvy(t) = \vomega_y(t)$ for \ref{eq:lookahead_mm}3-GDA we have the following:
\begin{align*}
  \begin{bmatrix}
    \dvx(t) \\
    \dvy(t) \\
     \dvom_x(t) \\
     \dvom_y(t)
  \end{bmatrix} = 
  \underbrace{\begin{bmatrix}
    0 & 0 & \mI & 0 \\
    0 & 0 & 0 & \mI \\
    -6\alpha \mA \mA^\intercal             & - \frac{6\alpha}{\gamma} \mA & - \frac{2}{\gamma} \mI & 0                        \\
    \frac{6\alpha}{\gamma} \mA^\intercal   & -6\alpha \mA^\intercal \mA           & 0                      & - \frac{2}{\gamma} \mI
  \end{bmatrix}}_{ \triangleq \mC_{\text{LA3-GDA}}} 
  \begin{bmatrix}
    \vx(t) \\
    \vy(t) \\
    \vomega_x(t) \\
    \vomega_y(t)
  \end{bmatrix}.
\end{align*}

\vm3
\proof[\textbf{Proof of Theorem~\ref{thm:bilinear_la3-gda_convergence}}]
As in \S\ref{app:bilinear_eg_convergence}, to obtain the eigenvalues $\lambda \in \CC$ of $\mC_{\text{LA3-GDA}}$ we have:
\vm3
\begin{align*}
\det(\mC_{\text{LA2-GDA}} - \lambda \mI) &=
\det\left(
\begin{bmatrix}
    - \lambda\mI & 0 & \mI & 0 \\
    0 & - \lambda\mI & 0 & \mI \\
    - 6 \alpha \mA \mA^\intercal           & - \frac{6\alpha}{\gamma} \mA & - (\frac{2}{\gamma}+\lambda) \mI     & 0 \\
    \frac{6\alpha}{\gamma} \mA^\intercal   & - 6 \alpha \mA^\intercal \mA         & 0                     & - (\frac{2}{\gamma}+\lambda) \mI
\end{bmatrix}
\right) \\
&= \det \left( \lambda(\frac{2}{\gamma}+\lambda) \mI - 
\underbrace{\begin{bmatrix}
- 6 \alpha \mA \mA^\intercal & - \frac{6\alpha}{\gamma} \mA\\
\frac{6\alpha}{\gamma} \mA^\intercal & - 6 \alpha \mA^\intercal \mA\\
\end{bmatrix}}_{\triangleq \mD} 
\right)   \,.
\end{align*}

\vm4
We observe that  only the lower left  block matrix $\mD$ differs from the proof for  EG in 
\S\ref{app:bilinear_eg_convergence},  thus the inequality Eq.~\ref{eq:prf-eg-bg_ineq} needs 
to be proven for this optimizer too.

\vm1
However, due to the different lower left block matrices for LA2-GDA, we have: 
\begin{align*}
\mu(\vz) &= \Bar{\vz}^\intercal\mD\vz =
\begin{bmatrix}
    \Bar{\vx}^\intercal && \Bar{\vy}^\intercal
\end{bmatrix}
\begin{bmatrix}
- 6 \alpha \mA \mA^\intercal           & - \frac{6\alpha}{\gamma} \mA\\
\frac{6\alpha}{\gamma}  \mA^\intercal  & - 6 \alpha \mA^\intercal \mA\\
\end{bmatrix}
\begin{bmatrix}
    \vx \\
    \vy
\end{bmatrix} \\
&= -6\alpha ||\mA^\intercal\vx||^2_2 - 6\alpha ||\mA\vy||^2_2  + \frac{6\alpha}{\gamma} \big(\Bar{\vy}^\intercal \mA^\intercal\vx  -  \Bar{\vx}^\intercal \mA\vy \big) \\
&= \underbrace{-6\alpha \big( ||\mA^\intercal\vx||^2_2 - ||\mA\vy||^2_2\big)}_{\mathfrak{R}\big(\mu(\vz)\big)}  + \underbrace{\frac{6\alpha}{\gamma} \cdot \mathfrak{I} (\Bar{\vx}^\intercal\mA\vy)}_{\mathfrak{I}\big(\mu(\vz)\big)} \cdot i \,,
\end{align*}

\good

where the last equality follows from the fact that $\Bar{\vx}^\intercal \mA\vy$ is a complex conjugate of $\Bar{\vy}^\intercal \mA^\intercal\vx$, thus $\Bar{\vy}^\intercal \mA^\intercal\vx - \Bar{\vx}^\intercal \mA\vy = 2\mathfrak{I}\big( \Bar{\vx}^\intercal\mA\vy \big)\cdot i$. By replacing this in Eq.~\ref{eq:prf-eg-bg_ineq}, we need to show that:
\vm2
\begin{gather*}
-6 \alpha \big( ||\mA^\intercal\vx||^2_2 + ||\mA\vy||^2_2 \big) \leq 
   - \sfrac{\gamma^2}{4}  \sfrac{6^{2}\alpha^2}{\gamma^2}  \mathfrak{I}^2 (\Bar{\vx}^\intercal \mA \vy ) \,, \\[1mm]
\links{or:}{37}
6\alpha \big( ||\mA^\intercal\vx||^2_2 + ||\mA\vy||^2_2 \big) \geq  \sfrac{6^2}{4} \alpha^2 \mathfrak{I}^2 (\Bar{\vx}^\intercal \mA \vy ) \,.
\end{gather*}
Thus, it suffices to show that (assuming $\alpha>0$):
$$
 ||\mA^\intercal\vx||^2_2 + ||\mA\vy||^2_2 \geq  2 
\underbrace{\sfrac{3}{4}\cdot\alpha}_{\parbox{30mm}{\footnotesize only difference with EG, and $\alpha \!\in\! (0,1)$}}
     | \Bar{\vx}^\intercal \mA \vy |^2 \,.
$$
As $\alpha \in (0,1)$ and the inequality~\eqref{eq:eg_bg_final_inequality} is tighter, the same proof as for EG (\S\ref{app:bilinear_eg_convergence}) follows.
\endproof

\section{Convergence analysis for monotone variational inequalities}\label{app:mvi}

In this section, we provide the proofs of Theorems~\ref{thm:ogda_hr_mvi_convergence}--\ref{thm:ogda_i_mvi_convergence}.

\subsection{Proof of Theorem \ref{thm:ogda_hr_mvi_convergence}:  
Last-iterate convergence of \ref{eq:ogda_hrde} for MVIs} \label{app:ogda_hr_mvi_convergence}

We will use the following Lyapunov function
\begin{equation} \tag{OGDA-L} 
\LL_{\text{OGDA}}(\vz, \vomega) = 
    \norm{\beta\vz \!+\! \vomega}^2_2 \!+\! \norm{\vomega}_2^2 \!+\! 4\beta\vz^\intercal V(\vz) 
   \!+\! \norm{V(\vz) \!+\! \vomega}_2^2 \!+\! \norm{V(\vz)}_2^2.
\end{equation}
We split this Lyapunov function into the following two parts:
\begin{align}
    \LL_{1}(\vz, \vomega) & = 
    \sfrac{1}{2}\left( \norm{\beta\vz + \vomega}^2_2 + \norm{\vomega}_2^2 + 4\beta\vz^\intercal V(\vz) \right) \tag{OGDA-L1}\label{eq:lyapunov_ogda:1} \\
    \LL_{2}(\vz, \vomega) & = 
    \sfrac{1}{2} \left( \norm{V(\vz)+\vomega}_2^2 + \norm{V(\vz)}_2^2 \right). \tag{OGDA-L2}\label{eq:lyapunov_ogda:2}
\end{align}
From the definition of the monotone variational inequality problem and from the positivity of the norms we 
have that $\LL_1(\vz, \vomega) > 0$ and $\LL_2(\vz, \vomega) > 0$ for any $(\vz, \vw) \neq (0, 0)$, and also 
$\LL_1(0, 0) = \LL_2(0, 0) = 0$. 
Next we expand $\dot{\LL}_1$ and $\dot{\LL}_2$.
\begin{align}\label{eq:proof_ogda_monotone_1st-derv}
\dot{\LL}_1(\vz, \vomega) &= 
\underbrace{\langle \beta\vz + \vomega,  \beta \dot{\vz} 
   + \dot{\vomega} \rangle}_{(1)} 
   + \underbrace{\langle \vomega, \dot{\vomega} \rangle}_{(2)} 
   + \underbrace{2 \beta\dot{\vz}^\intercal V(\vz)}_{(3)} 
   + \underbrace{2 \beta \vz^\intercal \dot{\big(V(\vz)\big)} }_{(4)} \\
\dot{\LL}_2(\vz, \vomega) & =
\underbrace{\langle V(\vz)+\vomega, \dot{\big(V(\vz)\big)} 
   + \dot{\vomega} \rangle}_{(5)} 
   + \underbrace{\langle V(\vz), \dot{\big(V(\vz)\big)} \rangle}_{(6)}.
\end{align}
We  now analyze each of these terms separately, using the following equality:
\begin{equation} \label{eq:dot_of_joint_verctor_field}
    \dot{\big(V(\vz)\big)} = J(\vz)\dot{\vz} = J(\vz) \cdot \vomega.
\end{equation}
\beginlitem{10}
\litem{(1)}
Replacing $\dot{\vomega}$ of ~\ref{eq:ogda_hrde}, we get\\[.5mm]
$\langle \beta\vz + \vomega,  \beta \dot{\vz} + \dot{\vomega} \rangle 
   = \langle \beta\vz+\vomega, -\beta V(\vz) -2 J(\vz)\vomega \rangle$.\\[.5mm]
Simplifying it  gives: 
$- \beta^2\vz^\intercal V(\vz) - \beta  \vomega^\intercal V(\vz) - 2\beta \vz^\intercal J(\vz) \vomega 
     - 2 \vomega^\intercal J(\vz) \vomega$;

\litem{(2)}\vp1
Similarly, by replacing $\dot{\vomega}$ of ~\ref{eq:ogda_hrde}, we get\\[.5mm]
$\langle \vomega, \dot{\vomega} \rangle = - \beta \norm{\vomega}_2^2 - \beta\vomega^\intercal V(\vz) 
- 2\vomega^\intercal J(\vz) \vomega$;

\litem{(3)}\vp{1.5}
$2\beta\dot{\vz}^\intercal V(\vz)$ = $2\beta\vomega^\intercal V(\vz)$;

\litem{(4)}\vp{1.5}
Using  the fact \eqref{eq:dot_of_joint_verctor_field} we get:\\
$2\beta \vz^\intercal \dot{\big(V(\vz)\big)} = 2 \beta \vz^\intercal J(\vz)\vomega$;

\litem{(5)}\vp1
Using \eqref{eq:ogda_hrde} and~\ref{eq:dot_of_joint_verctor_field} we get:\\
$\langle V(\vz)+\vomega, \dot{\big(V(\vz)\big)} +\dot{\vomega} \rangle = \langle V(\vz)+\vomega, 
     -\beta \big(V(\vz) + \vomega\big) - J(\vz)\vomega \rangle\\[1mm]
= -\beta ||V(\vz)+\vomega ||_2^2 - V^\intercal(\vz) J(\vz)\vomega - \vomega^\intercal J(\vz)\vomega$; and 

\litem{(6)}\vp1
Using Eq.~\ref{eq:ogda_hrde} we get: $\langle V(\vz), \dot{\big(V(\vz)\big)} \rangle = V^\intercal(\vz) J(\vz)\vomega$.
\endlitem

Using all the above, we get that
\begin{align} 
    \dot{\LL}_{1}(\vz, \vomega) & = - \beta \norm{\vomega}_2^2 - \beta^2 \vz^\intercal V(\vz) - 4 \vomega^\intercal J(\vz) \vomega. \label{eq:lyapunov_ogda:der:1} \\[1mm]
    \dot{\LL}_{2}(\vz, \vomega) & = 
    -\beta \norm{V(\vz) + \vomega}_2^2 - \vomega^\intercal J(\vz)\vomega. \label{eq:lyapunov_ogda:der:2}
\end{align}
Using the monotonicity of $V$ and the positive definiteness of $J$ we have 
$\dot{\LL}_{1}(\vz, \vomega) \le 0$ and $\dot{\LL}_{2}(\vz, \vomega) \le 0$. 
Hence, both $\LL_1$ and $\LL_2$ are Lyapunov functions for our problem. 
Therefore, by observing that $\LL_{\text{OGDA}}(\vz, \vomega) = 2 \LL_1(\vz, \vomega) + 2 \LL_2(\vz, \vomega)$, 
we have that $\LL_{\text{OGDA}}$ is also a Lyapunov function of our problem. 
To establish the convergence rate we start with the following lemma.

\vp1
\begin{lemma} \label{lem:convergenceRate:hr_ogda:1}
We assume the initial conditions $\vz(0) = \vz_0$ and $\vomega(0) = 0$. 
If for all $t \in [0, T]$ it holds that $\max\{\norm{\vomega(t)}_2, \norm{V(\vz(t)}_2\} \ge \eps$
  then we have that
\vm1
\[ 
T \le \left( \beta + 8 + \left(8 + \sfrac{2}{\beta}\right) \cdot L^2 \right) \cdot \sfrac{\norm{\vz_0}_2^2}{\eps^2}. 
\]
\end{lemma}

\vm6
\begin{proof}
With abuse of notation, we use $\LL_{\text{OGDA}}(t) \triangleq \LL_{\text{OGDA}}(\vz(t), \vomega(t))$ and we get
\begin{align}
    \LL_{\text{OGDA}}(0) & = \beta^2 \norm{\vz_0}_2^2 + 4 \beta \vz_0^\intercal V(\vz_0) + 2 \norm{V(\vz_0)}_2^2 \nonumber \\
       & \le (\beta^2 + 8 \beta) \norm{\vz_0}_2^2 + (2 + 8 \beta) \norm{V(\vz_0)}_2^2 \nonumber \\
       & \le (\beta^2 + 8 \beta + (8 \beta + 2) \cdot L^2) \norm{\vz_0}_2^2, \label{eq:initialPotential:continuousTime}
  \end{align}
where in the last inequality, we used the Lipschitzness of $V$ and the fact that $\vz^{\star} = 0$ is a solution 
and hence $V(0) = 0$.
  
Now using \eqref{eq:lyapunov_ogda:der:1}, \eqref{eq:lyapunov_ogda:der:2} together with the monotonicity of $V$ 
and the positive definiteness of $J$ we get that 
\begin{equation*}
\dot{\LL}_{\text{OGDA}}(t) 
     \le - 2 \beta \left(\norm{\vomega(t)}_2^2 + \norm{V(\vz(t)) + \vomega}_2^2 \right) 
     \le - \beta \left( \norm{V(\vz(t))}_2^2 \right),
\end{equation*}
where we have used the fact that for any vectors $\vx$, $\vy$ we have\\[3mm]
\cli{ $\norm{\vx + \vy}_2^2 \le 2 \norm{\vx}_2^2 + 2 \norm{\vy}_2^2$.}

\good

Also, we have that\ \ $\dot{\LL}_{\text{OGDA}}(t) \le - 2 \beta \left(\norm{\vomega(t)}_2^2 \right)$; 
hence overall we have that
\begin{equation}
\dot{\LL}_{\text{OGDA}}(t)  \le -\beta (\max\{\norm{\vomega(t)}_2^2, \norm{V(\vz(t))}_2^2\}). \label{eq:speedLowerBound:continuousTime}
\end{equation}
Now if for every $t \in [0, T]$ it holds that $\max\{\norm{\vomega(t)}_2, \norm{V(\vz(t)}_2\} \ge \eps$ then 
we have that for all $t \in [0, T]$ it holds
\[ 
\dot{\LL}_{\text{OGDA}}(t) \le -\beta \cdot \eps^2. 
\]
Now since $\LL_{\text{OGDA}}(t) \ge 0$ and using the above upper bound on the time derivative of $\LL$ together 
with the Mean Value Theorem, we have that
\[ \left( \LL_{\text{OGDA}}(0) - \LL_{\text{OGDA}}(t) \right) \ge (\beta \cdot \eps^2) \cdot T. 
\]
Finally, using the bound on the initial conditions \eqref{eq:initialPotential:continuousTime} the lemma follows.
\end{proof}

Lemma \ref{lem:convergenceRate:hr_ogda:1} shows that after a finite time both $\norm{\vomega(t)}_2$ and 
$\norm{V(\vz(t))}_2$ will become less than $\eps$. 
What remains to show the last-iterate convergence is to show that this upper bound will remain for later times. 
For this we need the following lemma.

\vp1
\begin{lemma} \label{lem:convergenceRate:hr_ogda:2}
Let $t^{\star} > 0$ a time such that both $\norm{\vomega(t^{\star})}_2 \le \eps$ and 
$\norm{V(\vz(t^{\star}))}_2 \le \eps$, then we have that for all $t > t^{\star}$ it holds that
\[
\norm{V(\vz(t))}_2 \le \sqrt{5} \eps. 
\]
\end{lemma}

\begin{proof}
Using the assumptions of the lemma we have the following\\[1mm]
(1)\ \ $\LL_2(t^{\star}) \le \frac{5}{2} \eps^2$,\ \ 
(2)\ \ $\dot{\LL}_2(t) \le 0$, and \ 
(3)\ \ $\norm{V(\vz(t))}_2 \le \sqrt{2 \LL_2(t)}$.

Combining the first two properties we obtain that $\LL_2(t) \le \frac{5}{2} \eps^2$ for all $t > t^{\star}$, 
and applying property (3) the lemma follows.
\end{proof}

If we combine Lemma \ref{lem:convergenceRate:hr_ogda:1} and Lemma \ref{lem:convergenceRate:hr_ogda:2}, then 
the first part of Theorem~\ref{thm:ogda_hr_mvi_convergence} follows.

\vp2
\textbf{Strongly monotone variational inequalities.} \ For the second 
part of Theorem \ref{thm:ogda_hr_mvi_convergence} we observe that
\begin{align} \label{eq:upperBoundToLyapunov}
  \LL_{\text{OGDA}}(t) \le 2 \beta^2 \norm{\vz(t)}_2^2 + 5 \norm{\vomega(t)}_2^2 + 2 \norm{V(\vz(t))}_2^2 + 4 \beta (\vz(t))^{\intercal} V(\vz(t)). 
\end{align}
Also, using the strong monotonicity of $V$ and the fact that $V(0) = 0$, which
we have assumed without loss of generality; then we have that
$(\vz(t))^{\intercal} V(\vz(t)) \ge \mu \norm{\vz(t)}_2^2$. 

\vm1
Applying this together with \eqref{eq:lyapunov_ogda:der:1} and \eqref{eq:lyapunov_ogda:der:2} we
get the following upper bounds for the time derivative of 
$\LL_{\text{OGDA}}$:
\begin{align}
  \sfrac{1}{\mu} \cdot \dot{\LL}_{\text{OGDA}}(t) & \le - \sfrac{2}{\mu} \beta^2 (\vz(t))^\intercal V(\vz(t)) \le - 2 \beta^2 \cdot \norm{\vz(t)}_2^2 \\
  \sfrac{5}{2 b} \cdot \dot{\LL}_{\text{OGDA}}(t) & \le -5 \norm{\vomega(t)}_2^2 \\
  \sfrac{2}{b} \cdot \dot{\LL}_{\text{OGDA}}(t) & \le - 4 \norm{\vomega(t)}_2^2 - 4 \norm{V(\vz(t)) + \vomega(t)}_2^2 - 4 \beta (\vz(t))^{\intercal} V(\vz(t)) \nonumber \\
        & \le - 2 \norm{V(\vz(t))}_2^2 - 4 \beta (\vz(t))^{\intercal} V(\vz(t)).
\end{align}
Combining these with \eqref{eq:upperBoundToLyapunov} and setting $\dps \kappa = \sfrac{1}{\mu} + \sfrac{9}{2 b}$ 
we get that
\vm2
\[ 
\LL_{\text{OGDA}}(t) \le - \kappa \cdot \dot{\LL}_{\text{OGDA}}(t), 
\]
where if we use the fact that $\LL_{\text{OGDA}}(t) \ge 0$ we get that
\[ 
- \kappa \cdot \sfrac{\dot{\LL}_{\text{OGDA}}(t)}{\LL_{\text{OGDA}}(t)} \ge 1, 
\quad \text{and hence}\quad
- \kappa \sfrac{d}{dt} \left( \ln (\LL_{\text{OGDA}}(t))\right) \ge 1. 
\]
Now integrating both parts from $0$ to $t$ we get
\[ 
\ln (\LL_{\text{OGDA}}(t)) \le \ln (\LL_{\text{OGDA}}(0)) - \sfrac{1}{\kappa} t, 
\]
and then applying the exponential function, we get
\[ 
\LL_{\text{OGDA}}(t) \le \LL_{\text{OGDA}}(0) \cdot \exp\left( - \sfrac{1}{\kappa} \cdot t \right). 
\]
 Using \eqref{eq:initialPotential:continuousTime} together with the fact that
$\LL_{\text{OGDA}}(t) \ge \norm{V(\vz(t)}_2^2 + 4 \beta \mu \norm{\vz(t)}_2^2$ the
second part of Theorem \ref{thm:ogda_hr_mvi_convergence} follows.

\subsection{Proof of Theorem \ref{thm:ogda_hr_mvi2_convergence}:  
Last-iterate convergence of\\ \ref{eq:ogda_hrde2} for MVIs} \label{app:ogda_hr_mvi2_convergence}

We use the following Lyapunov function:
\begin{align} \tag{OGDA2-L} 
& \LL_{\text{OGDA2}}(\vz, \vw) \\
& = \kappa^2 \norm{\vz + \vw}^2_2 + \kappa^2 \norm{\vz - \vw}_2^2 
   + \norm{ \kappa (\vz \!+\! \vw) + V(\vz)}_2^2 + \norm{V(\vz)}_2^2. \hskip18mm\notag
\end{align}
We split the Lyapunov function into the following two parts:
\begin{align}
    \LL_{3}(\vz, \vw) & = 
    \sfrac{1}{2}\left( \norm{\vz + \vw}^2_2 + \norm{\vz - \vw}_2^2 \right) \tag{OGDA2-L3}\label{eq:lyapunov_ogda2:1} \\
    \LL_{4}(\vz, \vomega) & = 
    \sfrac{1}{2} \left(\norm{ \kappa (\vz + \vw) + V(\vz)}_2^2 + \norm{V(\vz)}_2^2 \right). \tag{OGDA2-L4}\label{eq:lyapunov_ogda2:2}
\end{align}
From the definition of the monotone variational inequality problem and from the positivity of the norms we 
conclude that $\LL_3(\vz, \vw) > 0$ and $\LL_4(\vz, \vomega) > 0$ for any $(\vz, \vw) \neq (0, 0)$, and also 
$\LL_3(0, 0) = \LL_4(0, 0) = 0$. 
Next we expand $\dot{\LL}_3$ and $\dot{\LL}_4$:
\begin{align}\label{eq:proof_ogda2_monotone_1st-derv}
\dot{\LL}_3(\vz, \vomega) 
& = \underbrace{\langle \beta\vz + \vw,  \beta \dot{\vz} 
   + \dot{\vw} \rangle}_{(1)} + \underbrace{\langle \beta\vz - \vw,  \beta \dot{\vz} 
   - \dot{\vw} \rangle}_{(2)} \\
\dot{\LL}_4(\vz, \vomega) 
& = \underbrace{\langle \kappa (\vz + \vw) + V(\vz), \kappa (\dot{\vz} + \dot{\vw}) 
   + \dot{\big(V(\vz)\big)}\rangle}_{(3)} 
   + \underbrace{\langle V(\vz), \dot{\big(V(\vz)\big)} \rangle}_{(4)}.
\end{align}
We now analyze each of these terms separately, using the following inequality, which follows straightforwardly 
from the monotonicity of $V$:
\begin{equation} \label{eq:dot_of_joint_verctor_field:2}
    \langle \dot{\vz}, \dot{\big(V(\vz)\big)} \rangle \ge 0.
\end{equation}

\beginlitem{10}
\litem{(1)}
Replacing $\dot{\vz}$ and $\dot{\vw}$ from ~\ref{eq:ogda_hrde2}, we get\\[.5mm]
$\langle \beta\vz + \vw,  \beta \dot{\vz} + \dot{\vw} \rangle 
     = -\kappa \norm{\vz + \vw}_2^2 - 2 \langle \vz, V(\vz) \rangle - 2 \langle \vw, V(\vz) \rangle$;

\litem{(2)}\vp{1.5}
Replacing $\dot{\vz}$ and $\dot{\vw}$ from ~\ref{eq:ogda_hrde2}, we get\\[.5mm]
$\langle \beta\vz - \vw,  \beta \dot{\vz} - \dot{\vw} \rangle 
     = - 2 \langle \vz, V(\vz) \rangle + 2 \langle \vw, V(\vz) \rangle$;

\litem{(3)}\vp{1.5}
Similarly, we have that the time derivative of the term (3) is\\[.5mm]
$-2 \kappa \norm{\kappa(\vz + \vw) + V(\vz)}_2^2- 2 \langle \dot{\vz},  
   \dot{\big(V(\vz)\big)} \rangle - \langle V(\vz),  \dot{\big(V(\vz)\big)} \rangle$;

\litem{(4)}\vp{.5}
The derivative is directly equal to $\langle V(\vz), \dot{\big(V(\vz)\big)} \rangle$.
\endlitem

Using all of the above, we get that
\begin{align} 
    \dot{\LL}_{3}(\vz, \vw) & = - \kappa \norm{\vz + \vw}_2^2 - 4 \langle \vz, V(\vz) \rangle \label{eq:lyapunov_ogda:der:3} \\
    \dot{\LL}_{4}(\vz, \vw) & = -2 \kappa \norm{\kappa(\vz + \vw) + V(\vz)}_2^2- 2 \langle \dot{\vz},  \dot{\big(V(\vz)\big)} \rangle. \label{eq:lyapunov_ogda:der:4}
\end{align}
We define the set $S = \{ (\vz, \vw) \in \R^{2d} : V(\vz) = 0\}$. 
Using the monotonicity of $V$ and \eqref{eq:dot_of_joint_verctor_field:2} we have that 
$\dot{\LL}_{3}(\vz, \vomega) \le 0$, $\dot{\LL}_{4}(\vz, \vomega) \le 0$ for all $(\vz, \vw) \in \R^{2 d}$ 
and also $\dot{\LL}_{\text{OGDA2}} < 0$ for all $(\vz, \vw) \not\in S$. 
Hence, both $\LL_3$ and $\LL_4$ are Lyapunov functions for our problem. 
To establish the convergence rate, we start with the following lemma.

\vp1
\begin{lemma} \label{lem:convergenceRate:hr_ogda:1:2}
    We assume the initial conditions $\vz(0) = \vz_0$ and $\vw(0) = -\vz_0$. If
  for all $t \in [0, T]$ it holds that 
  $\max\{\norm{\kappa(\vz(t) + \vw(t))}_2, \norm{V(\vz(t)}_2\} \ge \eps$
  then we have that
\vm1
\[ 
T \le 2 \left( 2 \kappa + \sfrac{L^2}{\kappa} \right) \cdot \sfrac{\norm{\vz_0}_2^2}{\eps^2}. 
\]
\end{lemma}

\vm4
\begin{proof}
With abuse of notation, we use 
  $\LL_{\text{OGDA2}}(t) \triangleq \LL_{\text{OGDA2}}(\vz(t), \vomega(t))$ and
  we have that
\begin{equation}
\LL_{\text{OGDA2}}(0)  
     = 4 \cdot \kappa^2 \norm{\vz_0}_2^2 + 2 \norm{V(\vz_0)}_2^2 
     \le (4 \cdot \kappa^2 + 2 \cdot L^2) \norm{\vz_0}_2^2, \label{eq:initialPotential:continuousTime:2}
\end{equation}
where in the last inequality, we used the Lipschitzness of $V$ and the fact that $\vz^{\star} = 0$ is a solution 
and hence $V(0) = 0$.
  
Now using \eqref{eq:lyapunov_ogda:der:3}, \eqref{eq:lyapunov_ogda:der:4} 
  together with \eqref{eq:dot_of_joint_verctor_field:2} we have that
  \begin{align}
    \dot{\LL}_{\text{OGDA2}}(t) & \le - \kappa \left(\norm{\kappa(\vz(t) + \vw(t))}_2^2 + \norm{\kappa(\vz(t) + \vw(t)) + V(\vz(t))}_2^2 \right) \nonumber \\
       & \le - \kappa \left( \norm{V(\vz(t))}_2^2 \right), \nonumber
\intertext{where we have used the fact that $\norm{\vx + \vy}_2^2 \le 2 \norm{\vx}_2^2 + 2 \norm{\vy}_2^2$ 
holds for any vectors $\vx$, $\vy$. 
Also, we have that}
    \dot{\LL}_{\text{OGDA2}}(t) & \le - \kappa \left(\norm{\kappa(\vz(t) + \vw(t))}_2^2 \right); \nonumber
\intertext{hence overall we have that}
    \dot{\LL}_{\text{OGDA2}}(t) & \le -\kappa \left( \max\{\norm{\kappa(\vz(t) + \vw(t))}_2, \norm{V(\vz(t)}_2\}\right). \label{eq:speedLowerBound:continuousTime:2}
  \end{align}
Now if for every $t \in [0, T]$ it holds that $\max\{\norm{\kappa(\vz(t) + \vw(t))}_2,\, 
\norm{V(\vz(t)}_2\} \ge \eps$ then we have that for all $t \in [0, T]$ it holds
\[ 
\dot{\LL}_{\text{OGDA2}}(t) \le - \kappa \cdot \eps^2. 
\]
Now since $\LL_{\text{OGDA2}}(t) \ge 0$ and using the above upper bound on the time derivative of $\LL$ together 
with the Mean Value Theorem we have that
\[ 
\left( \LL_{\text{OGDA2}}(0) - \LL_{\text{OGDA2}}(t) \right) \ge (\kappa \cdot \eps^2) \cdot T. 
\]
Finally using the bound on the initial conditions \eqref{eq:initialPotential:continuousTime:2} the lemma follows.
\end{proof}

It remains to show that this upper bound persists for later times to establish the convergence of the last-iterate. 
For this purpose, we need the following lemma.

\vp1
\begin{lemma} \label{lem:convergenceRate:hr_ogda:2:2}
\ \ Let $t^{\star} > 0$ be a time such that both $\norm{\kappa(\vz(t^{\star} + \vw(t^{\star}))}_2 \le \eps$ and 
$\norm{V(\vz(t^{\star}))}_2 \le \eps$, then we have that for all $t > t^{\star}$ it holds that 
\vm1
\[
\norm{V(\vz(t))}_2 \le \sqrt{3} \eps. 
\]
\end{lemma}

\begin{proof}
Using the assumptions of the lemma, we have the following\\[2mm]
(1)\ \ $\LL_4(t^{\star}) \le \frac{3}{2} \eps^2$,\ \ 
(2)\ \ $\dot{\LL}_4(t) \le 0$, and\ \ 
(3)\ \ $\norm{V(\vz(t))}_2 \le \sqrt{2 \LL_4(t)}$.\\[2mm]
Combining the first two properties, we obtain $\LL_4(t) \le \frac{3}{2} \eps^2$ for all $t > t^{\star}$, 
and applying property (3) the lemma follows.
\end{proof}

Now if we combine Lemma \ref{lem:convergenceRate:hr_ogda:1:2} and Lemma \ref{lem:convergenceRate:hr_ogda:2:2}, 
then the first part of Theorem \ref{thm:ogda_hr_mvi2_convergence} follows. 
We omit the details of the second part because it is essentially identical to the corresponding argument for 
proving the second part of Theorem \ref{thm:ogda_hr_mvi_convergence}.

\subsection{Proof of Theorem \ref{thm:ogda_mvi2_convergence}:  
Best-iterate convergence of \ref{eq:ogda} for MVIs} \label{app:ogda_mvi2_convergence}

  Inspired by the analysis of the continuous-time dynamics that we presented in the previous section, we  use the following Lyapunov function:
\begin{align}
  \LL_5(\vz, \vw) = \norm{\vz - \vw}_2^2 + \norm{\vz + \vw + 2 \gamma V(\vz)}_2^2. \tag{OGDA-L5} \label{eq:OGDA-L5}
\end{align}
Before analyzing this Lyapunov function, we show the following useful lemma.

\vp1
\begin{lemma} \label{lem:ogda_ratios}
     Let $\vz_0 = \vz_1$ and $\vz_n$ follows the \eqref{OGDA} dynamics for 
  $n \ge 2$, also assume that $V$ is $L$-Lipschitz and that $\gamma \le 1/8 L$
  then for all $n \ge 0$ it holds that
\[ 
\sfrac{1}{2} \le \sfrac{\norm{V(\vz_{n + 1})}_2}{\norm{V(\vz_n)}_2} \le \sfrac{3}{2}. 
\]
\end{lemma}

\begin{proof}
Using the Lipschitzness of $V$ we have that
  \begin{align}
    \norm{V(\vz_{n + 1}) - V(\vz_n)}_2 
& \le L \cdot \norm{\vz_{n + 1} - \vz_n}_2 
     = L \cdot \gamma \cdot \norm{2 V(\vz_n) - V(\vz_{n - 1})}_2 \hskip10mm\nonumber \\[1mm]
& \le L \cdot \gamma \cdot \left( 2 \norm{V(\vz_n)}_2 + \norm{V(\vz_{n - 1})}_2 \right). \label{eq:proof:lem:ogda_ratios:1}
  \end{align}
Using the triangle inequality together with the above we have that
  \begin{align}
    \norm{V(\vz_{n + 1})}_2 & \le \norm{V(\vz_{n + 1}) - V(\vz_n)}_2 + \norm{V(\vz_n)}_2 \nonumber \\[1mm]
        & = (1 + 2 \cdot L \cdot \gamma) \cdot \norm{V(\vz_n)}_2 + L \cdot \gamma \norm{V(\vz_{n - 1})}_2, \label{eq:proof:lem:ogda_ratios:2}
  \end{align}

\good

and also that
  \begin{align}
    \norm{V(\vz_{n + 1})}_2 & \ge \norm{V(\vz_n)}_2 - \norm{V(\vz_{n + 1}) - V(\vz_n)}_2 \nonumber \\[1mm]
        & = (1 - 2 \cdot L \cdot \gamma) \cdot \norm{V(\vz_n)}_2 - L \cdot \gamma \norm{V(\vz_{n - 1})}_2. \label{eq:proof:lem:ogda_ratios:3}
  \end{align}
Now letting $b_n = \frac{\norm{V(\vz_{n})}_2}{\norm{V(\vz_{n - 1})}_2}$  we can rewrite 
\eqref{eq:proof:lem:ogda_ratios:2} and \eqref{eq:proof:lem:ogda_ratios:3} as follows
  \begin{align}
    b_{n + 1} \le (1 + 2 \cdot L \cdot \gamma) + \sfrac{L \cdot \gamma}{b_n} \label{eq:proof:lem:ogda_ratios:4} \\
    b_{n + 1} \ge (1 - 2 \cdot L \cdot \gamma) - \sfrac{L \cdot \gamma}{b_n}. \label{eq:proof:lem:ogda_ratios:5}
  \end{align}
We use induction to show that $b_n \in [1/2, 3/2]$. 
The base step holds since $\vz_0 = \vz_1$ and hence $b_1 = 1$. 
For the inductive step, we assume that $b_n \in [1/2, 3/2]$ and from \eqref{eq:proof:lem:ogda_ratios:4} we have that
\[ 
b_{n + 1} \le 1 + 4 \cdot L \cdot \gamma, 
\]
whereas from \eqref{eq:proof:lem:ogda_ratios:5} we have that
\[ 
b_{n + 1} \ge 1 - 4 \cdot L \cdot \gamma, 
\]
so it suffices to have that $\gamma \le 1 / (8 \cdot L)$.
\end{proof}

We continue with the analysis of the Lyapunov function $\LL_5$. 
It is a matter of algebraic calculations to verify that
\begin{gather*}
\LL_5(\vz_{n + 1}, \vw_{n + 1}) - \LL_5(\vz_n, \vw_n) = - \norm{\vz_n + \vw_n}_2^2 
     - 8 \gamma \langle \vz_n, V(\vz_n) \rangle \\
+ 4 \gamma^2 \norm{V(\vz_{n + 1})}_2^2 + 4 \gamma^2 \norm{V(\vz_{n})}_2^2 
     - 8 \gamma^2 \langle V(\vz_n), V(\vz_{n + 1}) \rangle. 
\end{gather*}
Adding the two equations of \eqref{OGDA-S} we can see that $\vz_{n + 1} + \vw_{n + 1} = 2 \gamma V(\vz_{n})$, 
from which we obtain
\begin{align*}
&  \LL_5(\vz_{n + 1}, \vw_{n + 1}) - \LL_5(\vz_n, \vw_n) \\
& = - 8 \gamma \langle \vz_n, V(\vz_n) \rangle \!-\! 4 \gamma^2 \left( \norm{V(\vz_{n - 1})}_2^2 
     \!-\! \norm{V(\vz_{n + 1})}_2^2 \!-\! \norm{V(\vz_{n})}_2^2 \!+\! 2 \langle V(\vz_n), V(\vz_{n + 1}) \rangle \right) \\
& = - 8 \gamma \langle \vz_n, V(\vz_n) \rangle - 4 \gamma^2 \left( \norm{V(\vz_{n - 1})}_2^2 - \norm{V(\vz_{n + 1}) 
     - V(\vz_{n})}_2^2 \right) \\
& = - 8 \gamma \langle \vz_n, V(\vz_n) \rangle - 2\gamma^2 \norm{V(\vz_{n - 1})}_2^2 
 - 2 \gamma^2 \left(\norm{V(\vz_{n - 1})}_2^2 - 2\norm{V(\vz_{n + 1}) - V(\vz_{n})}_2^2 \right).
\end{align*}
Now using \eqref{eq:proof:lem:ogda_ratios:1} we have that
$$
\norm{V(\vz_{n - 1})}_2^2 - 2\norm{V(\vz_{n + 1}) - V(\vz_{n})}_2^2 
\ge \left(1\! -\! 2 \!\cdot\! L \!\cdot\! \gamma \right) \cdot \norm{V(\vz_{n - 1})}_2^2 
     \!-\! 4 \!\cdot\! L \!\cdot\! \gamma \norm{V(\vz_{n})}_2^2,
$$
and assuming that $\gamma \le 1/(4 \cdot L)$ we obtain
$$
\norm{V(\vz_{n - 1})}_2^2 - 2\norm{V(\vz_{n + 1}) - V(\vz_{n})}_2^2 
   \ge \sfrac{1}{2} \cdot \norm{V(\vz_{n - 1})}_2^2 - 4 \cdot L \cdot \gamma \norm{V(\vz_{n})}_2^2.
$$
Using Lemma \ref{lem:ogda_ratios} we get that
$$
\norm{V(\vz_{n - 1})}_2^2 - 2\norm{V(\vz_{n + 1}) - V(\vz_{n})}_2^2 
  \ge \left(\sfrac{1}{4} - 4 \cdot L \cdot \gamma \right) \cdot \norm{V(\vz_{n - 1})}_2^2.
$$

\good

Hence, since $\gamma \le 1/(16 \cdot L)$, this implies that
$$
\norm{V(\vz_{n - 1})}_2^2 - 2\norm{V(\vz_{n + 1}) - V(\vz_{n})}_2^2 \ge 0.
$$
Now we substitute this inequality in the last expression above for the difference
$\LL_5(\vz_{n + 1}, \vw_{n + 1}) - \LL_5(\vz_n, \vw_n)$ and obtain:
\begin{align*}
  \LL_5(\vz_{n + 1}, \vw_{n + 1}) - \LL_5(\vz_n, \vw_n) \le - 8 \gamma \langle \vz_n, V(\vz_n) \rangle - 2\gamma^2 \norm{V(\vz_{n - 1})}_2^2, 
\end{align*}
but due to the monotonicity of $V$ we have that $\langle \vz_n, V(\vz_n) \rangle \ge 0$ and therefore we have that
\vm2
\begin{align}
  \LL_5(\vz_{n + 1}, \vw_{n + 1}) - \LL_5(\vz_n, \vw_n) \le - 2 \gamma^2 \norm{V(\vz_{n - 1})}_2^2. \label{eq:ogdaProofLyapunovDifferences}
\end{align}
We next bound the initial value of $\LL_5$. 
The initial conditions that 
$\vz_0 = \vz_1$ are equivalent to $\vw_0 = - \vz_0 - 4 \gamma V(\vz_0)$ and 
hence we have that
\begin{align}
 \LL_5(\vz_0, \vw_0) 
& = \norm{2 \vz_0 + 4 \gamma V(\vz_0)}_2^2 + 4 \gamma^2 \norm{V(\vz_0)}_2^2 \nonumber \\[1mm]
& \le 8 \norm{\vz_0}_2^2 + 36 \gamma^2 \norm{V(\vz_0)}_2^2 
     \le (8 + 36 \gamma^2 L^2) \norm{\vz_0}_2^2. \label{eq:ogdaProofLyapunovInitial}
\end{align}
Finally, combining \eqref{eq:ogdaProofLyapunovInitial} and \eqref{eq:ogdaProofLyapunovDifferences} Theorem 
\ref{thm:ogda_mvi2_convergence} follows.

\subsection{Proof of Theorem \ref{thm:ogda_i_mvi_convergence}: 
Last-iterate convergence of an implicit\\ discretization of \ref{eq:ogda_hrde}}
\label{sec:implicitDiscretizationProof}

In this section, we analyze the implicit discretization of the OGDA 
continuous-time dynamics that we presented in the previous section, for with 
implicit discretization we integrate the \eqref{eq:ogda_hrde} from 
$t = n \cdot \gamma$ to $t' = (n + 1) \cdot \gamma$, we use 
$\vz_n = \vz(n \cdot \gamma)$ and we make use of the following approximations:
\begin{enumerate}
\item[$\triangleright$] 
$\int_{t}^{t'} \vomega(\tau) d \tau \approx \frac{\gamma}{2} (\vomega_{n + 1} + \vomega_n)$,
  
\item[$\triangleright$]\vp2
 $\int_{t}^{t'} J(\vz(\tau))) \cdot \vomega(\tau) d \tau = \int_{t}^{t'} \frac{d}{d \tau} (V(\vz(\tau))) d \tau = V(\vz_{n + 1}) - V(\vz(n))$,
  
\item[$\triangleright$]\vp2
$\int_{t}^{t'} V(\vz(\tau)) d \tau \approx \frac{\gamma}{2} (V(\vz_{n + 1})+ V(\vz_n))$, $\quad$ and
  
\item[$\triangleright$]\vp2
$\int_{t}^{t'} \dot{\vz}(\tau) d \tau = \vz_{n + 1} - \vz_n$, 
     \ $\int_{t}^{t'} \dot{\vomega}(\tau) d \tau = \vomega_{n + 1} - \vomega_n$,\ \ $\beta = 2/\gamma$.
\end{enumerate}
Using all the above we get the following implicit discrete-time dynamics:
\begin{equation} \tag{OGDA-I} \label{eq:ogda_i2}
\begin{aligned}
\vz_{n + 1} & \!=\! \vz_{n} + \sfrac{\gamma}{2} \cdot \left( \vomega_{n + 1} + \vomega_n \right) \\
  \vomega_{n + 1} & \!=\! - V(\vz_{n + 1}) \!-\! \sfrac{1}{2} \left(V(\vz_{n + 1}) - V(\vz_n)\right)
\end{aligned}
\end{equation}
Again, we are going to use two Lyapunov functions to complete our last-iterate argument
\vm2
\begin{align}
    \LL_{1}(\vz, \vomega) & = \norm{\beta\vz + \vomega}^2_2 + \norm{\vomega}_2^2 + 2 \beta\vz^\intercal V(\vz) \tag{OGDA-I-L1}\label{eq:lyapunov_ogda_i:1} \\[1mm]
    \LL_{2}(\vz, \vomega) & = \norm{V(\vz)+\vomega}_2^2 + \norm{V(\vz)}_2^2. \tag{OGDA-I-L2}\label{eq:lyapunov_ogda_i:2}
\end{align}
We use $(\LL_1)_n$ to denote the value $\LL_1(\vz_n, \vomega_n)$ and correspondingly for $\LL_2$. 
The discrete-time differences of these functions are : 
\begin{align}
&  (\LL_1)_{n + 1} - (\LL_1)_n \nonumber\\[2mm]
& = \underbrace{\langle \beta (\vz_{n + 1} + \vz_n) + (\vomega_n + \vomega_{n + 1}),  \beta (\vz_{n + 1} - \vz_n) 
     + (\vomega_{n + 1} - \vomega_{n}) \rangle}_{(1)} \nonumber \\
&\hskip5mm + \underbrace{\langle \vomega_n + \vomega_{n + 1}, \vomega_{n + 1} 
   - \vomega_{n} \rangle}_{(2)} +
     \underbrace{\beta \cdot \vz_{n + 1}^\intercal V(\vz_{n + 1}) 
   - \beta \cdot \vz_{n}^\intercal V(\vz_{n})}_{(3)} \\
& (\LL_2)_{n + 1} - (\LL_2)_n \nonumber\\[2mm]
& = \underbrace{\langle (V(\vz_{n + 1}) + V(\vz_{n})) + (\vomega_{n + 1} + \vomega_{n}), (V(\vz_{n + 1}) 
     - V(\vz_{n})) + (\vomega_{n + 1} - \vomega_{n}) \rangle}_{(4)} \nonumber \\
&\hskip5mm + \underbrace{\langle V(\vz_{n + 1}) + V(\vz_{n}), V(\vz_{n + 1}) 
   - V(\vz_{n}) \rangle}_{(5)} 
\end{align}
Before analyzing each of these terms, we observe that we can write the difference
\begin{align} \label{eq:timeDifferenceW}
  \vomega_{n + 1} - \vomega_n 
& = - (\vomega_{n + 1} + \vomega_n) - (V(\vz_{n + 1}) + V(\vz_n)) - 2 \left(V(\vz_{n + 1}) - V(\vz_n)\right)
\end{align}
Separately for the above terms, we get:
\beginlitem{8}
\litem{(1)}
Replacing the difference $\vz_{n + 1} - \vz_n$ with $\frac{\gamma}{2} \cdot (\vomega_{n + 1} + \vomega_n)$ 
and using the fact that $\beta = \frac{2}{\gamma}$ we get that this term (1) is equal to
\vm1
\begin{align*}
& \langle \beta (\vz_{n + 1} + \vz_n) + (\vomega_n + \vomega_{n + 1}),  - (V(\vz_{n + 1}) + V(\vz_n)) - 2 \left(V(\vz_{n + 1}) - V(\vz_n)\right) \rangle \\[1mm]
& = - \beta \cdot \langle \vz_{n + 1} + \vz_n, V(\vz_{n + 1}) + V(\vz_n) \rangle 
     - 2 \cdot \beta \cdot \langle \vz_{n + 1} + \vz_n, V(\vz_{n + 1}) - V(\vz_n) \rangle \\[1mm]
&\hskip5mm - \langle \vomega_{n + 1} + \vomega_n, V(\vz_{n + 1}) + V(\vz_{n}) \rangle 
     - 2 \cdot \langle \vomega_{n + 1} + \vomega_n, V(\vz_{n + 1}) - V(\vz_{n}) \rangle.
\end{align*}
\litem{(2)}\vp1
Using \eqref{eq:timeDifferenceW} we get that the term (2)
\vm1
\begin{align*}
\langle \vomega_n & + \vomega_{n + 1}, \vomega_{n + 1} - \vomega_{n} \rangle \\
= ~ & - \norm{\vomega_{n + 1} + \vomega_n}_2^2 - \langle \vomega_{n + 1} + \vomega_n, V(\vz_{n + 1}) 
     + V(\vz_n) \rangle \\[1mm]
& - 2 \langle \vomega_{n + 1} + \vomega_n, V(\vz_{n + 1}) - V(\vz_n) \rangle.
\end{align*}
\litem{(3)}\vp1
For this term, we have two different expressions. The first one is
\begin{align*}
& \beta \cdot \vz_{n + 1}^\intercal V(\vz_{n + 1}) - \beta \cdot \vz_{n}^\intercal V(\vz_{n}) \\[1mm]
& =  \beta \cdot (\vz_{n + 1} - \vz_n)^\intercal V(\vz_{n + 1}) - \beta \cdot \vz_n^\intercal (V(\vz_n) 
     - V(\vz_{n + 1})) \\[1mm]
& = \langle \vomega_{n + 1} + \vomega_n, V(\vz_{n + 1})\rangle + \beta \cdot \vz_n^\intercal (V(\vz_{n + 1}) - V(\vz_n))
\end{align*}
and the other one is
\begin{align*}
& \beta \cdot \vz_{n + 1}^\intercal V(\vz_{n + 1}) - \beta \cdot \vz_{n}^\intercal V(\vz_{n}) \\[1mm]
& =  \beta \cdot \vz_{n + 1}^\intercal (V(\vz_{n + 1}) - V(\vz_n)) - \beta 
     \cdot (\vz_n - \vz_{n + 1})^\intercal V(\vz_n) \\[1mm]
& = \beta \cdot \vz_{n + 1}^\intercal (V(\vz_{n + 1}) - V(\vz_n)) + \langle \vomega_{n + 1} 
     + \vomega_n, V(\vz_{n})\rangle.
\end{align*}
If we combine the two above expressions, we obtain
\vm1
\begin{align*}
& 4 \beta \cdot \vz_{n + 1}^\intercal V(\vz_{n + 1}) - 4 \beta \cdot \vz_{n}^\intercal V(\vz_{n})  \\[1mm]
& = 2 \cdot \beta \cdot \langle \vz_{n + 1} + \vz_n, V(\vz_{n + 1}) - V(\vz_n) \rangle 
     + 2 \cdot \langle \vomega_{n + 1} + \vomega_n, V(\vz_{n + 1}) + V(\vz_{n}) \rangle.
\end{align*}
\litem{(4)}
Using the dynamics for $\vomega_{n + 1} \!-\! \vomega_n$ we conclude that this fourth term is equal to
\vm1
\begin{align*}
& \langle (V(\vz_{n + 1}) + V(\vz_{n})) + (\vomega_{n + 1} + \vomega_{n}), \\
&\quad - (\vomega_{n + 1} + \vomega_n) - (V(\vz_{n + 1}) + V(\vz_n)) - \left(V(\vz_{n + 1}) - V(\vz_n)\right)\rangle \\[.5mm]
& = - \norm{V(\vz_{n + 1}) + V(\vz_{n}) + \vomega_{n + 1} + \vomega_n}_2^2 -
 \langle V(\vz_{n + 1}) - V(\vz_n), V(\vz_{n + 1}) + V(\vz_n) \rangle \\
&\quad - \langle V(\vz_{n + 1}) - V(\vz_n), \vomega_{n + 1} + \vomega_{n} \rangle.
\end{align*}
\litem{(5)}
For this term, we do not make any changes
\begin{align*}
  \langle V(\vz_{n + 1}) + V(\vz_{n}), V(\vz_{n + 1}) - V(\vz_{n}) \rangle.
\end{align*}
\endlitem

Therefore we have that
\begin{align}
  (\LL_1)_{n + 1} - (\LL_1)_n = 
& - \beta \cdot \langle \vz_{n + 1} + \vz_n, V(\vz_{n + 1}) + V(\vz_n) \rangle - \norm{\vomega_{n + 1} + \vomega_n}_2^2 \nonumber \\[1mm]
  & - 8 \cdot \beta \cdot \langle \vz_{n + 1} - \vz_n, V(\vz_{n + 1}) - V(\vz_n) \rangle \\
(\LL_2)_{n + 1} - (\LL_2)_n = 
& - \norm{V(\vz_{n + 1}) + V(\vz_{n}) + \vomega_{n + 1} + \vomega_n}_2^2 \nonumber \\[1mm]
& - 2 \cdot \beta \cdot \langle V(\vz_{n + 1}) - V(\vz_n), \vz_{n + 1} - \vz_n \rangle.
\end{align}
The differences of $\LL_2$ are clearly negative. 
For the differences of $\LL_1$ the only thing that remains is to show that 
\[ 
\langle \vz_{n + 1} + \vz_n, V(\vz_{n + 1}) + V(\vz_n) \rangle + \langle \vz_{n + 1} 
     - \vz_n, V(\vz_{n + 1}) - V(\vz_n) \rangle \ge 0 
\]
but the left-hand side is equal to\ \ $2 \vz_{n + 1}^\intercal V(\vz_{n + 1}) + 2 \vz_n^{\intercal} V(\vz_n)$\ \ 
which is clearly positive due to the monotonicity of $V$ and the without loss of generality assumption of 
the optimality of $\vz = 0$, which implies $V(0) = 0$.

The proofs of the convergence rates are almost identical to the case of continuous time since the Lyapunov 
functions are almost identical except maybe some constants in some of the terms. 
For this reason, we refer to Appendix \ref{app:ogda_hr_mvi_convergence} for proof of the rates in the Theorem 
\ref{thm:ogda_i_mvi_convergence}.

\subsection{On the convergence of other HRDEs for EG and LA for MVI\\ problems}\label{sec:mvi:eg}

Although for the case of bilinear games, we showed that the HRDEs for \ref{eq:extragradient}, \ref{eq:ogda}, and \ref{eq:lookahead_mm}--GDA have last-iterate convergence to solutions of the corresponding MVI problem, we only showed the convergence of the \ref{eq:ogda_hrde} dynamics for the general monotone variational inequality problem. 
In this section, we show that a similar convergence result can \emph{not} be shown for the \ref{eq:extragradient} method using its $O(\gamma)$-HRDE. 
In particular, we consider a simple counterexample of a strongly convex-strongly concave zero-sum game for which infinitely many points could be fixed points of the HRDE of \ref{eq:extragradient} although they are far from the solution, which is at the origin.

\good

\textbf{Counterexample.} We set $f(x, y) = x^4 - y^4$. 
We observe that
for this function we have that the Jacobian matrix is equal to
$J(r, r) = 12 r^2 \mI$ hence for $\vz_0 = (r, r)$, $\vomega_0 = 0$ and 
$\beta = 12 r^2$ the EG dynamics remain to the point $\vz_0 = (r, r)$ forever. 
So for every value of $\beta$, there exists a point 
$\vz_0 = (\sqrt{\beta/12}, \sqrt{\beta/12})$, $\vomega_0 = 0$ that is a fixed point of the \ref{eq:eg_hrde} dynamics although this point is far from the equilibrium point at $0$. 
Similar counterexamples exist for the LA-GDA method as well.

The above counterexample contrasts sharply with the known results that establish not only the average iterate convergence \cite{nemirovski2004prox} but also the last-iterate convergence of the extragradient method \cite{golowich20}.
The reason for this is that in the case of the \ref{eq:extragradient} method, the first level of higher resolution is only able to capture the convergence for the bilinear games and not the method's convergence for general monotone variational inequalities. 
For the latter, we believe that higher-order differential equations are needed. 
We leave it as an interesting open problem to find a higher-order differential equation that captures the whole convergence behavior of the  EG and LA-GDA methods for general monotone variational inequalities.

\vp2
{\bf Acknowledgements.}\ \ We acknowledge support from the Swiss National Science Foundation (SNSF), grants 
P2ELP2\_199740 and P500PT\_214441, and from the Mathematical Data Science program of the Office of Naval 
Research under grant number N00014-18-1-2764. We thank Ya-Ping Hsieh for insightful discussion and feedback.

\end{document}